\documentclass[a4paper,10pt]{article}
\usepackage[utf8]{inputenc}
\usepackage{amsmath, amsfonts, amsthm,amssymb,  bbm, dsfont, lmodern,color}
\usepackage{math}
\usepackage{hyperref,mathtools} 
\usepackage[left=2cm,right=2cm,top=2.5cm,bottom=2.5cm]{geometry}
\usepackage{mathrsfs}
\usepackage{verbatim}
\newtheorem{lemma}{Lemma}[section]
\newtheorem{theorem}[lemma]{Theorem}
\newtheorem{proposition}[lemma]{Proposition}

\newtheorem{definition}[lemma]{Definition}
\theoremstyle{remark}
\newtheorem{remark}[lemma]{Remark}

\usepackage{pst-plot}
\usepackage{tikz-cd}
\usepackage{scalerel}

\usepackage{float}

 \usepackage{mathrsfs}

\def\beq{\begin{equation}}   \def\eeq{\end{equation}}
\def\bea{\begin{eqnarray}}  \def\eea{\end{eqnarray}}

\newcommand{\ch}{{{\mathtt c}_{\mathtt h}}}

\newcommand{\de}{{\rm d}}

\newcommand{\rvline}{\hspace*{-\arraycolsep}\vline\hspace*{-\arraycolsep}}
\renewcommand{\red}[1]{{\textcolor{red}{#1}}}

\newcommand{\reg}{\mathtt{r}}

\usepackage{todonotes}

\title{Full description of Benjamin-Feir instability  \\ for  generalized Korteweg-de Vries  equations}

\numberwithin{equation}{section}

 \author{Alberto Maspero, Antonio Milosh Radakovic\footnote{
International School for Advanced Studies (SISSA), Via Bonomea 265, 34136, Trieste, Italy.
 \textit{Emails: } \texttt{amaspero@sissa.it}, \texttt{aradakov@sissa.it}
 }}

\date{}

\begin{document}

\maketitle

\begin{abstract}
In this paper we consider a family of generalized  Korteweg-de Vries equations and study the linear modulational instability of small amplitude traveling waves solutions. 
Under  explicit non-degeneracy conditions on the dispersion relation, 
we completely describe  the spectrum near the origin of the  linearized operator at such solutions and prove that the unstable spectrum (when present) is composed by   branches depicting {\em always} a closed figure ``8''.
We apply our abstract theorem to several equations such as the Whitham, the gravity-capillary Whitham and the Kawahara equations, confirming that the unstable spectrum of the corresponding    linearized operators 
exhibits a figure ``8'' instability,  as it was   observed before only numerically.
Our method of proof uses a symplectic version of Kato’s theory of similarity transformation to reduce the problem to determine the eigenvalues of a $3 \times 3$  complex Hamiltonian and reversible matrix. 
Then,  via a block-diagonalization procedure, we  conjugate  such matrix into a block-diagonal one composed by a   $2\times 2$ Hamiltonian and reversible matrix, describing  the  unstable spectrum,  and a single purely imaginary element describing the stable eigenvalue.
\end{abstract}

\section{Introduction}\label{sec1}

The generalized Korteweg-de Vries (gKdV)  equations are a family of dispersive partial differential equations of the form 
\begin{equation}\label{eq}
    \partial_t u +  \cM(D) u_x +  (u^2)_x =0 ,  \quad x \in \R \ , 
\end{equation}
where  $u(t,x)$ is a  real valued function  and $\cM(D)$  a   Fourier multiplier with a real valued symbol  $\fm(\xi)$, defined by
\begin{equation}\label{M}
\widehat{[\cM(D) u]}(\xi):= \fm (\xi) \hat u (\xi)  \ . 
\end{equation}
Here $\hat u(\xi) := \int_\R u(x) e^{- \im \xi x}$.
The  family of equations \eqref{eq} includes many well-known dispersive PDEs such as
the Korteweg-de-Vries, the Benjamin-Ono, the Whitham, the Kawahara equations. 
In Table \ref{table1} below we report the  values of the corresponding symbols. 

Under very general assumptions on the symbol $\fm(\xi)$, 
equation \eqref{eq} 
admits global in time solutions in the form of space-periodic waves traveling at constant speed (they are steady in a moving frame).
A classical problem is to determine their  stability/instability when subjected to  long wave perturbations.
Such  problem, usually called modulational instability,  goes back to the pioneering works by Benjamin-Feir \cite{BF}, Whitham \cite{Wh}, Lighthill \cite{Li}, Zakharov \cite{Z0} in the context of water waves, recently fully and rigorously justified by \cite{BrM, NS, BMV1, BMV2, BMV3, BMV4, HY, CNS}.

The first, fundamental step to understand modulational instability is to   study the spectrum of the linear operator obtained  linearizing \eqref{eq}  at a traveling wave solution:
in case it intersects the positive real plane, the traveling wave is called linearly modulationally unstable. 
Then,  at least  for  equation \eqref{eq}, 
 the general theory developed by Jin-Liao-Lin in \cite{JLL} can be used to upgrade the linear instability to a nonlinear one.

Linear modulational instability for equation \eqref{eq} has been studied both for small amplitudes waves  \cite{johnson2013stability, hur2014modulational, hur2015modulational, pava2017stability, HP} and even large amplitudes  ones    in case of KdV  \cite{HK, BJ, BHJ} (corresponding to $\fm(\xi) = \xi^2$). 
All these works investigate the existence of unstable spectrum locally around the origin, in order to eventually determine the linear modulational instability of these systems.
In \cite{BD}, the presence of unstable spectrum is investigated, in the entirety of the complex plane. Leveraging the integrable structure of the KdV equation, the authors demonstrate the complete linear modulational stability of the equation. This method has moreover proven effective in determining both nonlinear and orbital modulational stability for the KdV equation. 
The integrable structure of the KdV equation has also been exploited in \cite{DK}, to prove the nonlinear orbital stability of the cnoidal periodic traveling waves of the KdV equation.
%
%Very recently  Berti-Maspero-Ventura in \cite{BMV1, BMV2} were able to prove the conjecture stating that the Benjamin-Feir spectrum of water waves equation depicts a closed "8" figure, continuously parametrized by the Floquet exponent.\\
%Using the same spectral approach, we have been able to prove an analogous result for the a class of partial differential equations in Equation \ref{eq}.

 The goal of the present paper, following the approach recently developed by Berti-Maspero-Ventura \cite{BMV1, BMV2, BMV3},  is to completely characterize  the spectrum  near the origin of the complex plane of the operator obtained linearizing  \eqref{eq} at small amplitude traveling waves. 
 We prove that, provided the   symbol $\fm(\xi)$ fulfills certain explicit and mild non-degeneracy conditions (see \eqref{tedet}--\eqref{formulaewb}), such spectrum is either purely imaginary  or it possesses unstable spectral branches  outside the imaginary axis depicting {\em always} a complete figure “8”. 
 This dichotomy is determined by the sign of a certain coefficient, see \eqref{eWB}, depending only on few values of $\fm(\xi)$ and its derivatives.

We apply our abstract criterion to several equations in the family of generalized KdV, among which the finite depth  Whitham equation and its  gravity-capillary and vorticity-modified variants,  and  the Kawahara equation. 
In all these cases we prove that the unstable spectrum near the origin depicts a complete figure ``8'', as was observed before only numerically \cite{hur2014modulational, hur2015modulational, creedon2021high}.

\smallskip 
We now describe precisely our result. 
The first set of assumptions that we make on the symbol $\fm(\xi)$ in \eqref{M} are rather mild and   used to guarantee the existence of  small amplitude traveling wave solutions of \eqref{eq}, see Theorem \ref{existPTW} below: \\
 
\noindent{\bf Assumption A:}\label{A} The symbol $\fm(\xi)$ of the Fourier multiplier $\cM(D)$ in \eqref{M} fulfills: 
\begin{itemize}
\item[(A1)] $\fm(\xi)$ is real valued and even,  $\fm(\xi) = \fm(-\xi)$ for any $\xi \in \R$;
\item[(A2)] $\fm(\xi)$ is a function in $\cC^\reg (\R)$, $\reg \geq 3$,  fulfilling for some $m \geq -1$ and $\xi$ sufficiently large: there exist $C_1, C_2 >0$ such that 
\begin{equation}\label{orderpsedodiff}
C_1 \la \xi \ra^m \leq \fm(\xi) \leq  C_2 \la \xi \ra^{m }  \ ,   \ \ \forall |\xi| \gg 1 \ ; 
\end{equation}
here we put $\la \xi \ra:= (1 + \xi^2)^{1/2}$;
\begin{comment} \red{come si fa per $\fm(\xi) = |\xi|^\alpha$?}
moreover \red{qualcosa sulle derivate?}\\
\red{PROPOSTA:}
\begin{itemize}
    \item $\fm$ admits right (and left) derivative  in zero up to the second order.
\end{itemize}
\end{comment}

\item[(A3)]there exists $ c_0 >0$ such that 
\begin{equation}\label{m1diff}
\inf_{n \in \N\setminus\{1\}} | \fm(n) - \fm(1) | \geq c_0  \ . 
\end{equation}
\end{itemize}
\begin{comment}
\begin{remark}
    As we will show below, we can focus just on $\mu>0$. This motivates the right differentiability condition asked in Assumption (A2), that will allow us to expand the operator near $\mu=0$ up to the second order. \red{SCRIVERE MEGLIO QUESTA COSA}
\end{remark}
\end{comment}
We list now some symbols $\fm(\xi)$ fulfilling Assumption A, and the corresponding  equations \eqref{eq} generated: 
\begingroup

\setlength{\tabcolsep}{12pt} % Default value: 6pt
\renewcommand{\arraystretch}{2.} % Default value: 1
\begin{table}[ht]\label{table1}
\caption{The correspondence between the symbol and the equation of form \eqref{eq} obtained.}
\centering
\begin{tabular}{|c  | c | } 
 \hline
Symbol  $\fm(\xi) $  & Equation \\
 \hline
$ \xi^2$ & KdV  \\ 

 $|\xi|^\alpha, \ \ \alpha\geq 3$ & fractional KdV  \\

 $|\xi|$, & Benjamin-Ono \\

 $ \sqrt{\dfrac{\tanh(\tth \xi)}{\xi}}$, \ \ $\tth >0$,  & Whitham equation \\

  $ \sqrt{(1 + \kappa \xi^2) \frac{\tanh(\tth \xi)}{\xi}}$,  \ \ $\tth , \kappa>0$, & Whitham  with surface tension \\

  $\xi \coth(\tth \xi) - \tth^{-1}$ & intermediate long-wave \\
  
  $  \xi^2 + \tb \xi^4$ & Kawahara \\
  
  \hline
\end{tabular}
\end{table}
\begin{remark}
   Actually, in view of property \textit{(ii)} at page \pageref{symspectrum}, it is sufficient to ask that $m(\xi)$ is $\cC^3$ at any $\xi>0$, has right derivatives up to third order at $\xi=0$, and every derivative is continuous at $\xi=0$. With this extension, Assumption A is satisfied also in the case of the Benjamin-Ono equation. In this case, the notation $\frac{\de^n}{\de \mu^n} \fm (0)$, corresponds to the $n$-th right derivative of $\fm$ at $0$, and so $\ddot \fm (0) = 0$.
\end{remark}
\endgroup

Under Assumption A,  the gKdV equation in  \eqref{eq} admits families of small amplitude traveling waves solutions, which are periodic in space, i.e. of the form   $u(t,x) = \breve u(x-ct)$ for some real $c$ -- the speed -- and $2\pi$ periodic function $\breve u(x)$. The following theorem is essentially known in the literature, see e.g. \cite{1f611663399640f8bf7f8a8812b5fafe, hur2014modulational}.
For $\sigma, s \in \R$, we denote by $H^{\sigma,s}(\T):= \{u:\T\to\R \text{ s.t. }\sum_{k\in\Z} |u_k|^2 e^{2\sigma|k|} \langle k\rangle^{2s}<  \infty \}$,  and $B(a)\subset \R$ the ball of center zero and radius $a$.

\begin{theorem}[\cite{1f611663399640f8bf7f8a8812b5fafe, hur2014modulational}] \label{existPTW}
 Assume Assumption {\rm A}. 
 Fix $\sigma >0$, $s > \frac12$, with $m$ the  order of the Fourier multiplier in \eqref{orderpsedodiff}.
There exist  $\e_0,\tb_0>0$ and a locally unique non-zero family of  functions and real numbers $(u_{\e,\tb}(x), c_{\e,\tb})$, analytic as a map $B(\e_0)\times B(\tb_0) \mapsto H^{\sigma, s} (\T)\times \R$, solution of the stationary equation 
\begin{equation}\label{stateq}
 c \, u- \cM(D) u - u^2 + (\fm(0)-c)^2 \tb   = 0
    \end{equation}
    so that
    $u_{\e,\tb}(x)$ is $2\pi$-periodic,  even, real valued and having at $\tb=0$  the asymptotic expansion
\begin{align}\label{expansionsPTWsol}
   & u_\e(x):= u_{\e,\tb=0}(x)=\e \cos(x)+ \e^2 \left(u_2^{[0]} + u_2^{[2]} \cos(2x)\right)  + \cO(\e^3) , \\
   & c_\e:= c_{\e,\tb=0} = \fm(1)  + \e^2 c_2 + \cO(\e^3) , 
    \end{align}
where
    \begin{align}
    &   u_2^{[0]}:=\frac{1}{2}\frac{1}{\fm(1) - \fm(0)} , \qquad u_2^{[2]} := \frac{1}{2}\frac{1}{\fm(1)-\fm(2)} , \\
    &   c_2 := \frac{1}{\fm(1) - \fm(0)} + \frac12 \frac{1}{\fm(1) - \fm(2)} \ . 
    \end{align}
    Moreover $u_{\e,\tb}$, $c_{\e,\tb}$ are given by 
    $$ u_{\e,\tb} (x) = u_\e(x) + v(\e,\tb)\, , \qquad c_{\e,\tb} = c_\e + 2v(\e,\tb)\, ,$$ for a certain real valued constant $v=v(\e,\tb)\in \R$, analytic as a map $B(\e_0)\times B(\tb_0) \mapsto \R$.
 \end{theorem}
    \begin{proof} We sketch the proof for completeness, since the exact statement of the theorem is not present in literature.
The function  $F(c,u,\tb):= c \, u- \cM(D) u - u^2 + (\fm(0)-c)^2 \tb $ is analytic as a map $\R \times H^{\sigma, s}(\T) \times \R  \mapsto H^{\sigma,s-m}(\T)$ for any $\sigma\geq 0$, $s>\frac 12$.
 When $\tb = 0$, the proof is a standard application of the Crandall-Rabinowitz bifurcation theorem (provided one restricts to functions even in space, to obtain a one-dimensional kernel). This yields
 an analytic map  $B(\e_0) \ni \e \to (c_\e,u_\e) \in \R \times H^{\sigma,s}(\T)$ such that  $ F(c_\e, u_\e, 0) \equiv 0$ $\forall |\e| \leq \e_0$.

When $|\tb| \leq \tb_0$ and $\tb_0$ sufficiently small, define the analytic function 
$B(\e_0)\times B(\tb_0) \ni (\e, \tb) \to v(\e,\tb) \in \R$  as the solution close to zero of the quadratic equation 
    $$
    (\fm (0)- c_\e)^2 \tb - v(\e,\tb)(\fm (0)-c_\e)(4\tb +1 ) + v(\e,\tb)^2(4\tb + 1) \equiv 0 \ .
    $$
    One verifies  that $u_{\e,\tb} (x) = u_\e(x) + v(\e,\tb)$, $c_{\e,\tb} = c_\e + 2v(\e,\tb)$ solve 
    $ F(c_{\e, \tb}, u_{\e,\tb} , \tb) \equiv 0$ for any $\e,\tb$ small.
   \end{proof}
We shall need also a second set of assumptions on the symbol $\fm(\xi)$, involving  non-degeneracy conditions: \\

\noindent{\bf Assumption B:}\label{B}  
Denote by $\dot \fm(\xi):= \frac{\di}{\di \xi} \fm(\xi)$ and $\ddot \fm(\xi):= \frac{\di^2}{\di \xi^2} \fm(\xi)$.
The numbers 
\begin{equation}\label{tedet}
\begin{aligned}
\te_{12}:= \dot \fm(1) , \qquad 
\te_{d}:= \dot \fm(1) + \fm(1) - \fm(0) \ , \qquad 
\te_{b}:=- \dot \fm(1)- \frac12 \ddot \fm(1) 
\end{aligned}
\end{equation}
and 
\begin{equation}\label{formulaewb}
 \te_{w} :=  \frac{\dot \fm(1)+ 3\fm(1) -2\fm(2) - \fm(0)}{ \big(\fm(1)-\fm(2) \big) \, \big(\dot \fm(1)+\fm(1)- \fm(0)\big)}
\end{equation}
are all well defined and non-zero.
\vspace{1em}

 In order to determine the stability/instability of the traveling wave solution given by Theorem \ref{existPTW},  we linearize the equations of motion \eqref{eq} at $u_\e$ and put ourself in  a  moving frame  at speed $c_\e$, getting the linear system
\begin{equation}\label{lin.pro}
    \pa_t h =     \cL_\e h
    \end{equation}
    where %the linear autonomous operator $\cL_\e$ is explicitly given by
    \begin{equation}\label{def:cLe}
        \cL_\e :=  \pa_x \circ(c_\e -\cM(D) - 2 u_\e)  \ .
    \end{equation}
 Since the operator $\cL_\e$  has $2\pi$-periodic coefficients, its spectrum on $L^2(\R, \C)$ is mostly conveniently described by Bloch-Floquet theory, which guarantees that
    \begin{equation}\label{sp.rel}
    \sigma_{L^2(\R)}(\cL_\e) = \bigcup_{\mu \in [-\frac12, \frac12)} \sigma_{L^2(\T)}(\cL_{\mu,\e})
    \end{equation}
   where the Floquet operator $\cL_{\mu,\e}$ is  given by 
  \begin{equation}\label{Lmue}
    \cL_{\mu,\e}:= e^{-\im \mu x} \, \cL_\e \, e^{\im \mu x} \ , 
  \end{equation}
  and acts as an (unbounded) operator on $L^2(\T, \C)$.
  Recall that    if $\lambda$ is an eigenvalue of $\cL_{\mu,\e}$ on $L^2(\T, \C)$
with eigenvector $v(x) \in L^2(\T, \C)$, then  $h (t,x) = e^{\lambda t} e^{\im \mu x} v(x)$ solves the linear problem  \eqref{lin.pro}.

\begin{itemize}
\item {\bf Notation:} 
 We  denote by
$r (\mu^{m_1}\e^{n_1},\dots,\mu^{m_p}\e^{n_p}) $ {\em real} valued functions  of class at least $\cC^{1}$ fulfilling,  for small values of $(\mu, \e)$, the bound
 $|r (\mu^{m_1}\e^{n_1},\dots,\mu^{m_p}\e^{n_p})| \leq C \sum_{j = 1}^p |\mu|^{m_j}|\e|^{n_j}$.
\end{itemize}

When $(\mu,\e)=(0,0)$, $\cL_{0,0}$ has a triple eigenvalue $0$. Our main result describes the spectrum near $0$ of the operator $\cL_{\mu,\e}$ for any $(\mu,\e)$ close to $(0,0)$:

\begin{theorem}\label{mainres1}
Assume Assumption {\rm A} and {\rm B} and denote
\begin{equation}
\label{eWB}
\te_{\scaleto{\mathrm{WB}}{3pt}} := 
\te_w \, \te_b  = 
 \frac{
\big( \dot \fm(1) + \frac12 \ddot \fm(1)  \big) \ 
\big(\dot \fm(1)+ 3\fm(1) -2\fm(2) - \fm(0)\big)}{ \big(\fm(2)-\fm(1) \big) \, \big(\dot \fm(1)+\fm(1)- \fm(0)\big)} 
\end{equation}
the {\em Whitham-Benjamin} coefficient for the generalized KdV \eqref{eq}.
There exist $\e_1,\mu_0>0$  such that the following holds true:
\begin{itemize}
\item[(i)] If $\te_{\scaleto{\mathrm{WB}}{3pt}} >0$, there exists  a  function $\underline{\mu}:[0,\e_1)\to [0,\mu_0)$ of class $\cC^{\reg-2}$
 of the type
    $$
    \underline{\mu}(\e)= \sqrt{\frac{\te_{\scaleto{\mathrm{WB}}{3pt}} }{\te_{b}^2}}  \, (\e+r(\e^2))  
    $$
    such that, for any $\e\in [0,\e_1)$, the spectrum of the operator $\cL_{\mu,\e}$ near zero consists in the  purely imaginary eigenvalue 
    \begin{equation}\label{pi}
\lambda_0(\mu,\e) =    \im  \big(\fm(1)-\fm(0)\big)\mu +  \im \mu r(\e^2,\mu^2)
    \end{equation}
    and two eigenvalues $\lambda_1^\pm (\mu,\e)$ of the form
    \begin{equation}\label{figure8param}
    \lambda_1^\pm (\mu,\e) = 
        \begin{cases}
            -\im \mu  \big( \te_{12}   +  r(\e^2,\mu\e,\mu^2) \big) \pm \mu \sqrt{\Delta_{\scaleto{\mathrm{BF}}{3pt}}} \, ,  &\mu \in [0,\underline{\mu}(\e))\\
               -\im \mu  \big( \te_{12}   +  r(\e^3) \big) \, ,  &\mu=\underline{\mu}(\e)\\
               -\im \mu  \big( \te_{12}   +  r(\e^2,\mu\e,\mu^2) \big) \pm \im\mu
             \,  \sqrt{|\Delta_{\scaleto{\mathrm{BF}}{3pt}}|} \, ,  &\mu \in [\underline{\mu}(\e),\mu_0)
        \end{cases}
    \end{equation}
    with 
   \begin{equation}\label{DeltaBF}
   \Delta_{\scaleto{\mathrm{BF}}{3pt}}(\mu,\e) :=\te_{\scaleto{\mathrm{WB}}{3pt}} \, \e^2(1+\hat r_1(\e,\mu)) -  \te_{b}^2 \, \mu^2(1+\hat r_2(\e,\mu))
   \end{equation}
   and some functions $r, \hat r_1, \hat r_2$ of class $\cC^{\reg -2}$. 
    The function $\underline{\mu}(\e)$ is implicitly defined by $\Delta_{\scaleto{\mathrm{BF}}{3pt}}(\underline{\mu}(\e),\e) = 0$.
    \item[(ii)] If $\te_{\scaleto{\mathrm{WB}}{3pt}} <0$,  then for any $|\mu|\leq \mu_0$ and $|\e| \leq \e_1$, the 
operator $\cL_{\mu,\e}$ has near the origin three purely imaginary eigenvalues.
\end{itemize}

\end{theorem}
Let us make some comments.\\

1. \textsc{Universality of Figure ``8'' spectrum  for small amplitude traveling waves:}  Consider the linearized operator $\cL_\e$ at a traveling wave of small amplitude  $\e >0$ and assume the  Whitham-Benjamin coefficient 
\eqref{eWB} is strictly positive.
In view of \eqref{sp.rel}, the unstable spectrum of $\cL_\e$ near the origin is completely described by the two curves  $\mu \mapsto \lambda^\pm_1(\mu,\e)$. 
As 
$\mu$ ranges in the interval $[0, \und{\mu}(\e))$, these two curves 
 depict a closed figure with an ``8'' shape, which  is well approximated by the curve 
\begin{equation}\label{parapp8}
    \mu \in [0,\underline{\mu}(\e)) \longrightarrow \left( \pm \mu\sqrt{\te_{\scaleto{\mathrm{WB}}{3pt}}\e^2
 - \te_b^2\mu^2},\, -\te_{12}\mu \right) \in \R^2 \ , 
\end{equation}
plotted in figure \ref{figure8image} below. 
In case the coefficient $\te_{12} <0$ (recall \eqref{tedet}),  equation \eqref{figure8param} parametrizes the upper part of the eight, the lower part being  parametrized by $\mu<0$, and viceversa in case $\te_{12} >0$. 
Nevertheless, the operator $\cL_{\e}$, defined on $L^2(\R)$, is real, and so the spectrum is invariant also by conjugation with respect to the real axis.  So one obtain a complete figure ``8''.

We find interesting that, whatever the operator $\cM(D)$ (provided Assumptions A and B are met), the spectrum near the origin of $\cL_\e$  is either purely imaginary or  it contains an unstable component having {always}  a figure 	``8'' shape.

% adapt widths of minipages to your needs
\begin{figure}[H]\label{figure8image}
    \centering
    \includegraphics[scale = 0.3]{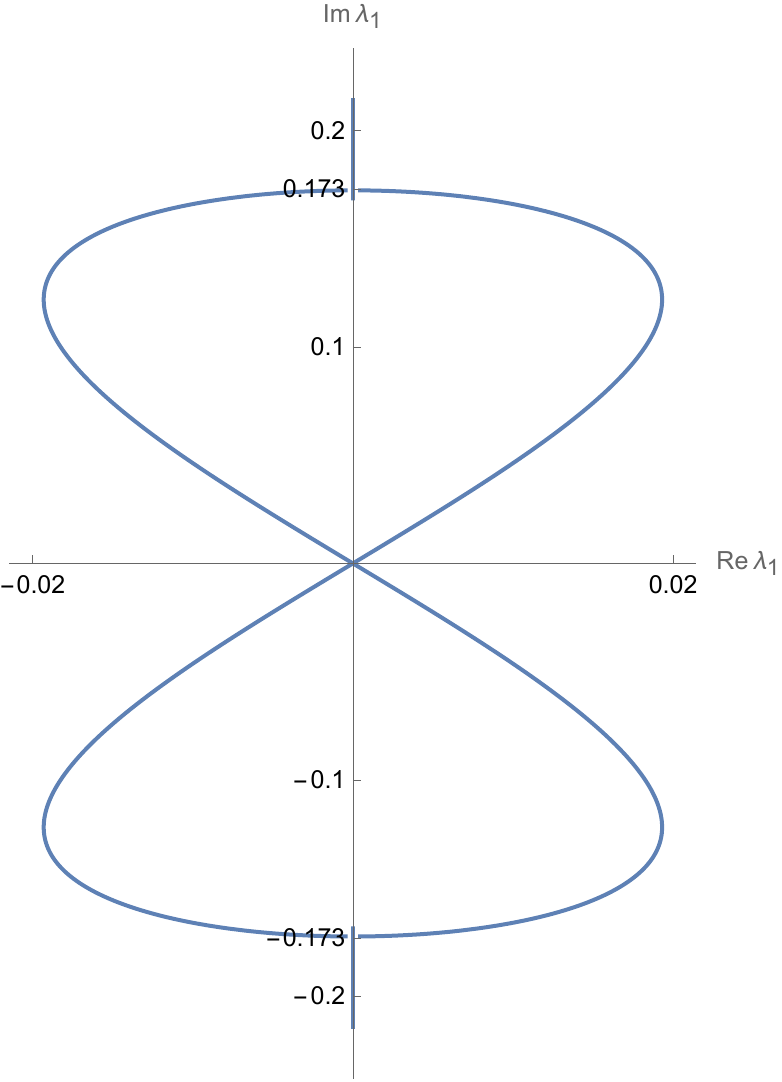} 
    \caption{The figure represents the plot of the curve parametrized in equation \eqref{parapp8}. }
    \end{figure}
    
%:
%:

2. \textsc{Generality of choosing a $2\pi$-periodic traveling wave:} In Theorem \ref{mainres1} we focused on the modulational instability for a space $2\pi$-periodic  traveling wave. 
The instability of   $\frac{2\pi}{\kappa}$-traveling waves, $\kappa \in \N$,  can be easily obtained by rescaling, as we now show. 
The linear rescaling 
$$
u(x)=(\Psi^{\kappa} w)(x) = \kappa^{-1} w(\kappa x)
$$
transforms equation \eqref{eq} in 
\begin{equation}\label{transfeq}
    \pa_t w + \cM_{\kappa} w_x +  (w^2)_x=0 \ , \quad \cM_\kappa(D):= \kappa\cM(\kappa D) \ . 
\end{equation}
%Being $\Phi^{\alpha,\beta}_t$ the flow induced by the vector field $X_{\alpha,\beta}(w) = -\left(\frac1\alpha \cM(\beta D) w + w^2\right)_x$, we have that the following diagram commutes:
%\begin{equation}
%    \begin{tikzcd} 
%        L^2(\R)  \arrow[r, "\Phi^{1,1}_t"] \arrow[d, "\Psi^{\alpha,\beta}"]
%       & L^2(\R)\arrow[d, "\Psi^{\alpha,\beta}"]\\ 
%       L^2(\R) \arrow[r, "\Phi^{\alpha,\beta}_{t/(\alpha\beta)}"] & L^2(\R)
%       \end{tikzcd}
%\end{equation}
Then   a $\frac{2\pi}{\kappa}$-periodic traveling wave solution of  \eqref{eq} with speed $c$, is  transformed in a $2\pi$-periodic traveling wave solution of \eqref{transfeq} with speed $c\kappa $.
To verify its stability/instability it is then  enough to replace $\fm(\xi)$ with $\fm_\kappa(\xi):= \fm(\kappa \xi)$ and recompute the coefficients in \eqref{tedet}, \eqref{formulaewb} for $\fm_\kappa(\xi)$. Thus, our abstract result Theorem \ref{mainres1} is also suitable to establish the modulational instability also for periodic traveling waves of general period $\frac{2\pi}{\kappa}$, provided that the symbol $\fm_\kappa$ satisfies Assumptions A and B.\\
%\\
% Linearizing such equations at the repsective periodic traveling waves, in the reference frames moving at the respective speeds, the resulting linear operators are similar with similarity transformation $\Psi_{1,\beta}$, up to the multiplication by a real constant.
%Moreover, the transformation $\Psi_{\alpha,1}$ takes Equation \ref{eq} into the equation with symbol $\dfrac1\alpha \cM(D)$, up to multiplication by a real constant.\\
%Aware of this results, we can restrict without loss of generality to study the stability of $2\pi$-periodic traveling waves. Notice indeed that, in the above list of equations, the Fourier multipliers depends on a set of free parameters. Considering the stability of a periodic traveling wave, solution of a given equation, with a period different from $2\pi$ corresponds, through some transformation $\Psi_{\alpha,\beta}$, to study the stability of the $2\pi$-periodic traveling wave, solving an equation of the same type, but with a different choice of the parameters.\\
%Moreover, equations whose Fourier multipliers are equal, up to the multiplication by a real constant, are considered equivalent, that in our case means that the spectrums of the respective linearized operetors will be equal up to the multiplication by that constant.\\

3. \textsc{Comparison with Berti-Maspero-Ventura \cite{BMV1, BMV2, BMV3}:}  
Although  the strategy of the proof follows the same lines of  \cite{BMV1, BMV2, BMV3},  
there are  important differences that we now describe. 
The first one is that the gKdV equation in \eqref{eq} is  Hamiltonian  with a symplectic tensor $\pa_x$ which is unbounded, contrary to the case   of water waves in  \cite{BMV1, BMV2} where the symplectic tensor is the standard matrix $\begin{pmatrix} 0 & - \uno \\ \uno & 0 \end{pmatrix}$. 
A consequence  is that, in our case, the symplectic tensor is  conjugated by the Bloch-Floquet transform to the  unbounded operator $\cJ_\mu = \pa_x + \im \mu$ which is $\mu$-dependent and not invertible at $\mu = 0$. 
Then we need to develop a  Kato similarity transformation theory adapted to this situation.
In particular a different notion of symplectic basis is needed (see Definition \ref{def:symp}) which distinguishes the cases  $\mu =0$ (when $\cJ_0$ is not invertible) and $\mu \neq 0$.

A second difference is that, as we shall explain below, the bifurcation problem is reduced to study the spectrum of a $3\times 3$ complex Hamiltonian and reversible matrix, and not of a  $4\times 4$ one as in \cite{BMV1, BMV2}.
In view of this, the block-diagonalization procedure performed in Section 5 conjugates the original $3\times 3$ matrix to a new block-diagonal one in with a $2\times 2$ block and a single diagonal element which is always purely imaginary. 

A final difference is that we require  only a minimum amount of $\cC^3$ regularity  in $\mu$ (contrary to \cite{BMV1, BMV2, BMV3} where the dependence on $\mu$ is analytic). This is needed to deal with the cases of   $\fm(\xi)$ having  only finite regularity, such as for the fractional KdV. 
A consequence is that we need extra care in dealing with the size of the remainders in the expansions of the matrix elements in Proposition \ref{matrixrepprop}. \\

4. \textsc{About nonlinear modulational instability:}  When $\te_{\scaleto{\mathrm{WB}}{3pt}} >0$,  Theorem
\ref{mainres1} guarantees the presence  of unstable spectra. Note that, in case of equation \eqref{eq}, the presence of {\em any} unstable spectrum   is enough to imply nonlinear modulational instability or localized instability, by applying the general theory and methods developed by Jin, Liao, Lin in \cite{JLL}. 
A natural, interesting, question is if the  figure 8 spectrum can give more information on the nonlinear  dynamics. 
\vspace{1em}

\noindent{\bf Ideas and scheme of the paper}. 
In order to prove Theorem \ref{mainres1} we need to compute the 
spectrum near the origin of the operator $\cL_{\mu,\e}$ in \eqref{Lmue} for every value of  $\e>0$ and $\mu>0$ sufficiently small.
We treat this as a  bifurcation problem from the zero eigenvalue  of $\cL_{0,0}$ which has algebraic multiplicity three and geometric multiplicity two. 
In Section \ref{sec:Kato}
we introduce  a modified version of the symplectic   Kato's similarity theory developed in \cite{BMV1} necessary to handle the fact that the symplectic tensor is not invertible. 
In Section \ref{matrixrapsection} we exploit this theory to compute the
 $3\times 3$ complex Hamiltonian and reversible matrix representing the action of $\cL_{\mu,\e}$ on the invariant subspace $\cV_{\mu,\e}$ associated to the spectrum of $\cL_{\mu,\e}$ near the origin. 
Then, in Section 
\ref{sec5},
 we perform a non-perturbative block-diagonalization of the matrix, 
 conjugating it to a new  matrix with two decoupled blocks: a purely imaginary number, and a $2\times 2$ block, whose eigenvalues can be computed explicitly.  As the parameter $\mu$ varies, these two latter eigenvalues parametrize the desired figure ``8''.\\
 Finally, in Section \ref{appl}, we apply our abstract Theorem \ref{mainres1} to Whitham's like equations, Kawahara, the intermediate long-wave and fractional KdV equations (recovering, for the latter two equations, results already known in the literature).\\

\noindent {\bf Acknowledgments.}
A. Maspero and A. M. Radakovic are  supported by the European Union ERC CONSOLIDATOR GRANT 2023 GUnDHam,
Project Number: 101124921. 
A. Maspero is also supported by  PRIN 2020 (2020XB3EFL001) “Hamiltonian and dispersive PDEs”, PRIN 2022 (2022HSSYPN)   "TESEO - Turbulent Effects vs Stability in Equations from Oceanography" and GNAMPA.

\section{The complete Benjamin-Feir spectrum of generalized KdV}
In this section we present the complete spectral Theorem \ref{mainres2}.
We first introduce the  Hamiltonian and reversible structure of the generalized KdV  equation \eqref{eq}, which will play a crucial role in our analysis.\\
We start with the Hamiltonian formulation: 
 equation \eqref{eq} has the Hamiltonian form
$$
\pa_t u = \cJ \grad H(u)
$$
where $\cJ$ is the skew-adjoint  operator
\begin{equation}\label{J}
\cJ:= \pa_x  , 
\end{equation}
$H(u)$ is the real valued Hamiltonian function
$$
H(u) := -\frac12 \int_\R u (\cM(D) u) \di x - \frac13 \int_\R u^3 \di x , 
$$
well defined on $H^{m_*}(\R)$, $m_*:= \max(1, m)$, where $m$ is the order of the Fourier multiplier $\cM$ defined in \eqref{orderpsedodiff},  and  the gradient is taken with respect to 
the real $L^2$ scalar product
\begin{equation} \label{l2prod}
    \la u , v \ra_r := \int_\R u(x) \, v(x) \, \di x  \ . 
\end{equation}
Equation \eqref{eq} enjoys two more symmetries. 
First it is  invariant  by translation, namely its vector field $X(u):=-\pa_x  (\cM u + u^2)$
fulfills
\begin{equation}
    X \circ \tau_\varsigma = \tau_\varsigma \circ X  \ ,
\end{equation}
where $\tau_\varsigma$ is the translation operator
\begin{equation}
\label{trans}
    [\tau_\varsigma u](x) := u(x + \varsigma) \ , \quad \varsigma \in \R \ . 
\end{equation}
Secondly, the equation is time reversible,  namely the vector field $X(u)$ fulfills
\begin{equation}\label{reversibility1}
 X \circ  \rho  =  -  \rho \circ X
\end{equation}
where $\rho$ is the involution
\begin{equation}\label{involution}
[\rho u](x):= u(-x) \ .
 \end{equation}
 The linearized operator $\cL_\e$ in \eqref{def:cLe} inherits such structures: first it is linearly Hamiltonian, i.e. of the form 
 \begin{equation}
 \cL_\e = \cJ \cB_\e
 \end{equation}
  where $\cB_\e$ is the symmetric (with respect to the scalar product \eqref{l2prod}) operator
    \begin{equation}\label{def:cBe}
        \cB_\e := c_\e - \cM(D) - 2 u_\e  \ . 
    \end{equation}
Moreover it is also  reversible,  namely
\begin{equation}\label{rev}
    \cL_\e\circ \rho = - \rho \circ \cL_\e \ . 
    \end{equation}

    \noindent{\bf Bloch-Floquet expansion.}   
    Consider the Floquet operator $\cL_{\mu,\e}$ in \eqref{Lmue}.  
  We remark that (see e.g. \cite{BMV1}):
\begin{itemize}
\item[(i)] if $A = \mathrm{Op}(a) $  
is a pseudo-differential  operator with   
 symbol $ a(x, \xi ) $, which is $2\pi$ periodic in $x$, 
then 
$  A_\mu := e^{- \im \mu x}A  e^{ \im \mu x}  =  \mathrm{Op} (a(x, \xi + \mu )) $.
\item[(ii)]\label{symspectrum} If $ A$ is a real operator then 
$ \overline{ A_\mu} = A_{- \mu } $.  
We recall that $\overline{A} u:= \overline{A \bar u }$.
As a consequence the spectrum  
$ \sigma (A_{-\mu}) = \overline{  \sigma (A_{\mu}) } $ and we can restrict to study  $  \sigma (A_{\mu}) $ only  for $ \mu > 0 $.
\item[(iii)]
 $\sigma(A_{\mu})$ is a 1-periodic set with respect to $\mu$, so  we can restrict  to  $\mu \in [0, \frac12)$.
 \end{itemize}
  In view of these properties,    
     the operator $\cL_{\mu,\e}$ takes the form of the {\em complex Hamiltonian and reversible} operator
    \begin{equation}\label{Lmue2}
        \cL_{\mu,\e}=
        \cJ_\mu \, \cB_{\mu,\e} \ ,
    \end{equation}
    where the skew-adjoint operator $\cJ_{\mu}$ and the self-adjoint operator $\cB_{\mu,\e}$ are defined by 
    \begin{equation}\label{Bmue}
     \cJ_\mu := \pa_x+\im\mu , \quad \cB_{\mu,\e}:= c_\e -\cM(D+\mu) - 2u_\e 
    \end{equation}
    and  $\cM(D+\mu)$ is the Fourier multiplier
    \begin{equation}
[\cM(D+\mu)u]^\wedge (\xi):= \fm(\xi + \mu) \hat u(\xi)  \ .  
\end{equation}
    We regard $\cL_{\mu,\e}$ as a closed,  unbounded operator with domain $H^{m_\star }(\T,\C)$, with $m_\star := \max(1, 1+m)$ and $m$ the order of the Fourier multiplier $\fm(\xi)$,  and range $L^2(\T,\C)$, equipped with the complex scalar product
    \begin{equation}\label{c.sp}
        (f,g) := \frac{1}{\pi} \int_0^{2\pi} f \bar g \, \di x , \qquad \forall f,g \in L^2(\T, \C)  \ . 
    \end{equation}
    We recall that an operator $\cL$ is called  
    \begin{itemize}
        \item[1.] {\em complex Hamiltonian},  if $\cL = \cJ \cB $ with $\cJ$ skew-adjoint and $\cB$ selfadjoint (with respect to the complex scalar product \eqref{c.sp}):
        $$
        \cJ^* = - \cJ  , \qquad \cB^* = \cB \ ;
        $$
        \item[2.] {\em complex reversible}, if
        $$
        \cL \circ \bar\rho  = - \bar \rho \circ \cL 
        $$
        where $\bar \rho$ is the complex involution
        \begin{equation}\label{barrho}
            [\bar \rho u](x) := \bar u(-x) \ . 
        \end{equation}
    \end{itemize}

 Note that  the  operator $\cB_{\mu,\e}$ is  {\em reversibility preserving}, i.e.
        \begin{equation}\label{B.rev}
        \cB_{\mu,\e} \circ \bar\rho  =  \bar \rho \circ \cB_{\mu,\e}  \ . 
        \end{equation}

Due to Assumption A, we have that $(\mu,\e) \mapsto \cL_{\mu,\e}$ and $(\mu,\e) \mapsto \cB_{\mu,\e}$
are functions of class $\cC^\reg$ from 
$\{ |\mu| < \mu_0 \} \times \{ |\e| < \e_0 \} $ to 
$\cL(H^{m_\star+1}(\T),L^2(\T)))$.

 \begin{lemma}\label{lem:LB.im}
The operators 
 $\dot \cL_{0,\e}:=  \pa_\mu \cL_{\mu,\e}\big|_{\mu=0} $  and 
  $\dot  \cB_{0,\e}:= \pa_\mu \cB_{\mu,\e}\big|_{\mu=0}$ are purely imaginary\footnote{An operator $A$ is purely imaginary if $\bar A = - A$, where $\bar A f:= \overline{A \bar f}$.
A purely imaginary operator maps a real vector into a purely imaginary one.}   
and respectively reversible and reversibility preserving.
\end{lemma}
\begin{proof}
Consider first  $\dot  \cB_{0,\e}$. By \eqref{Bmue}, it  equals the Fourier multiplier  $- \dot\cM(D)$, having for  symbol the odd function $-\frac{\di \fm(\xi)}{\di \xi}$. Regarding 
 $\dot \cL_{0,\e}$,  by \eqref{Lmue2},  \eqref{Bmue} one has 
    $    \dot \cL_{0,\e} = \cJ_0 \dot \cB_{0,\e} + \im \cB_{0,\e}$.
    As $\cJ_0$ and $\cB_{0,\e}$ are real operators, whereas $ \dot\cB_{0,\e}$ is purely imaginary, $\dot \cL_{0,\e}$ is purely imaginary as well.
    The reversibility properties follow from those of $\cB_{\mu,\e}$ and $\cL_{\mu,\e}$.
    \end{proof}
 
Recall that the spectrum of a complex  Hamiltonian and reversible operator is symmetric with respect to reflections across the imaginary axis, and so the  eigenvalues can exit the imaginary axis only after collisions of two or more eigenvalues on it.
Therefore, we shall now study  multiple eigenvalues of the operators $\cL_{\mu,0}$ and $\cL_{0,\e}$.

\smallskip
\noindent
    {\bf Spectrum of $\cL_{\mu,0}$.} 
As $c_0 = \fm(1)$ and $u_0 = 0$ by Theorem \ref{existPTW}, the operator    $\cL_{\mu,0} = \cJ_\mu \cB_{\mu,0} $ equals  the Fourier multiplier $(\pa_x + \im \mu) (\fm(1) - \cM(D+\mu))$. 
Its  
 $L^2(\T)$-spectrum is given by 
  $
\sigma_{L^2(\T)}(\mathcal{L}_{\mu,0})=\{\im \, \omega_{j,\mu},  \ \ j \in\mathbb{Z}\} \ 
$ with 
\begin{equation}\label{omega_nmu}
\omega_{j,\mu}:= (j+\mu)\big(\fm(1)-\fm(j+\mu)\big) , \qquad j \in \Z \ , \quad 
    \cL_{\mu,0} \, e^{\im j x} =  \im  \, \omega_{j,\mu}e^{\im jx}\, .
    \end{equation}
We can notice that at $\mu=0$, in view of Assumption A3, one has 
\begin{equation}\label{collision}
\omega_{-1,0} = \omega_{0,0} = \omega_{1,0} = 0 , \quad |\omega_{j,0}| \geq |j| c_0 > 0  \ \  \mbox{ for }  j \neq 0, \pm 1 \ ,
\end{equation}
namely $\ker \cL_{0,0}$ has dimension three.
 We can then decompose the spectrum of $\cL_{0,0}$ in two disjoint subsets as 
\begin{equation}\label{spettrodiviso0}
    \sigma_{L^2(\T)}(\mathcal{L}_{0,0}) = \sigma'_{L^2(\T)} (\mathcal{L}_{0,0}) \cup \sigma_{L^2(\T)}'' (\mathcal{L}_{0,0})
\end{equation}
with 
$\sigma'_{L^2(\T)} (\mathcal{L}_{0,0})=\{0 \},$
and, in view of \eqref{m1diff}, 
$\sigma''_{L^2(\T)} (\mathcal{L}_{0,0}) \subseteq \{\im\lambda\text{, s.t.  } \lambda\in\R\text{, }|\lambda|\geq c_0\}$.
From Kato's perturbation theory (see Lemma
\ref{lem:Kato1} below) for any $\mu,\e$ sufficiently small, the perturbed spectrum admits a disjoint decomposition as 
\begin{equation}
    \sigma_{L^2(\T)}(\mathcal{L}_{\mu,0}) = \sigma'_{L^2(\T)} (\mathcal{L}_{\mu,0}) \cup \sigma_{L^2(\T)}'' (\mathcal{L}_{\mu,0}) 
\end{equation}
with $\sigma'_{L^2(\T)} (\mathcal{L}_{\mu,0}) $ consisting in three 
 distinguished purely imaginary eigenvalues, which are $\cO(\mu)$ close to zero, as it follows directly from equation \eqref{omega_nmu}.\\

\noindent
{\bf Spectrum of $\cL_{0,\e}$.} 
The operator $\cL_{0,\e} \equiv \cL_{\e}$ has $0$ as a defective eigenvalue of algebraic multiplicity 3 and geometric multiplicity 2. This follows since the original equation \eqref{eq} is invariant by translation, and the traveling waves solutions appear in a two parameter family $u_{\e,\tb}$. Precisely we have the following lemma:
\begin{lemma} \label{basis0,eps}
For $0<|\e| \leq \e_0$, the linear operator $\cL_{0,\e}$ has $0$ as an eigenvalue of algebraic multiplicity 3 and geometric multiplicity 2. 
In particular 
$g_\e^-(x) := (u_\e)_x$ and $g_\e^+(x) := (\partial_\e u_\e) -\frac12 \pa_\e c_\e$
are eigenvectors, whereas $g(x)  := \frac{1}{\sqrt2}$ is a generalized eigenvector fulfilling 
$\cL_{0,\e} g = - \sqrt 2 g_\e^-$.
\end{lemma}
\begin{proof}
By \eqref{stateq},  for every $\e$ sufficiently small, $u_\e$ solves
$F(\e, u_\e):= \partial_x (c_\e u_\e -\cM u_\e  - u_\e^2) = 0$.
Derivating such identity  with respect to $x$ yields $\cL_{0,\e}  g_\e^- = 0
$.
Next from \eqref{def:cLe} and \eqref{def:cBe} we get 
$\cL_{0,\e}[ g] = -\sqrt{2}  g_\e^-$.
Finally derivating  $F(\e, u_\e)=0$  with respect to $\e$ yields  
$ \cL_{0,\e}[\pa_\e u_\e] = -  (\pa_\e c_\e) g_\e^-$,
hence  $ \cL_{0,\e}[g_\e^+]=0$. 
\end{proof}

By  Kato's perturbation theory (Lemma \ref{lem:Kato1})
for any $\mu, \e \neq 0$  sufficiently small, the perturbed spectrum
$\sigma\left(\cL_{\mu,\e}\right) $ admits a disjoint decomposition as 
\begin{equation}\label{SSE}
\sigma\left(\cL_{\mu,\e}\right) = \sigma'\left(\cL_{\mu,\e}\right) \cup \sigma''\left(\cL_{\mu,\e}\right) \, ,
\end{equation}
where $ \sigma'\left(\cL_{\mu,\e}\right)$  consists of 3 eigenvalues close to 0.

We denote by $\cV_{\mu, \e}$   the spectral subspace associated with  $\sigma'\left(\cL_{\mu,\e}\right) $, which   has  dimension 3 and it is  invariant by $\cL_{\mu, \e}$.

\begin{remark}\label{rem:V0e}
Lemma \ref{basis0,eps} tells that, for any $0<|\e| \leq \e_0$, $\cV_{0,\e} = \textup{span} \{g_\e^-, g_\e^+, \frac{1}{\sqrt 2} \}$.
\end{remark}

Our  main result gives the full description of the spectrum of the operator
 $ \cL_{\mu,\e}\vert_{\mathcal{V}_{\mu,\e} }\colon \mathcal{V}_{\mu,\e} \to  \mathcal{V}_{\mu,\e} $ whenever   $\mu, \e $ are sufficiently small, and  characterizes for which values of $(\mu,\e)$   eigenvalues  
with  nonzero real parts appear.
Before stating our main result, let us introduce some notations we shall use through  the paper:
\begin{itemize}\label{notation}
\item[$\bullet$] {\bf  Notation:}
We denote by  $\cO^\reg(\mu^{m_1}\e^{n_1},\dots,\mu^{m_p}\e^{n_p})$, $ m_j, n_j \in \N  $ (for us $\N:=\{1,2,\dots\} $),   functions  of $(\mu,\e)$ of class $\cC^\reg$ with values on a Banach space $X$ which satisfy, for some $ C > 0 $, the bound
 $\|\cO^\reg(\mu^{m_j}\e^{n_j})\|_X \leq C \sum_{j = 1}^p |\mu|^{m_j}|\e|^{n_j}$ 
 for small values of $(\mu, \e)$. 
Similarly, we denote $r_k (\mu^{m_1}\e^{n_1},\dots,\mu^{m_p}\e^{n_p}) $
scalar  functions  $\cO^1(\mu^{m_1}\e^{n_1},\dots,\mu^{m_p}\e^{n_p})$ which are also {\em real} valued, note that here we  use $k$ simply as a placeholder to distinguish different functions; it is unrelated to any different use of $k$. 
\end{itemize}

Our complete spectral result is the following:
\begin{theorem}\label{mainres2}
Assume Assumption {\rm A} and {\rm B} at page \pageref{A} and \pageref{B}.   There exist $\e_0, \mu_0>0$ such that, for any $|\e|\leq \e_0$, $0<\mu<\mu_0$, the operator $\cL_{\mu,\e}:\cV_{\mu,\e}\to \cV_{\mu,\e}$ can be represented by the $3\times 3$ block-diagonal matrix
    \begin{equation}
        \begin{pmatrix}
            \mathtt{U} & \rvline & \mathtt{0} \\
            \hline
           \mathtt{0}^\dagger & \rvline & \im \, \mathtt{g}
        \end{pmatrix}
    \end{equation}
    where
    $\mathtt{U}$ is the $2\times 2$ matrix 
        \begin{equation}\label{tU}
       \mathtt{U}:=
          -\im \mu  \big( \te_{12}   +  r_0(\e^2,\mu\e,\mu^2) \big) +
 \mu         \begin{pmatrix}
         \im \big(  \e^2 r_1(1) + \mu^2 r_2(1)  \big)& \te_{b} +  r_3(\mu,\e)\\
            \te_{ w}\,\e^2(1+r_4(\mu,\e) ) -  \te_{b}\, \mu^2 (1+r_5(\mu,\e))  &
              \im  \mu^2 r_{6}(1) 
        \end{pmatrix}
    \end{equation}
    with  $\te_{12}, \te_{w}, \te_{b}$   in \eqref{formulaewb}, \eqref{tedet}, 
 $\mathtt{g}$ is the real number 
    \begin{equation}\label{tg.fin}
    \mathtt{g}:= \big(\fm(1)-\fm(0)\big)\mu + \mu r(\e^2,\mu^2)  \ ,
\end{equation}
    and $\mathtt{0} = (0,0)^\dag$.  The functions $r_k$ are of class $\cC^{\reg-2}$.
    
  The eigenvalues of $\mathtt{U}$ have the form 
  $$
  \lambda_1^\pm (\mu,\e) =   
    -\im \mu  \big(\te_{12}   +  r(\e^2,\mu\e,\mu^2) \big)
  \pm 
  \mu \, \sqrt{
 \underbrace{ \te_{\scaleto{\mathrm{WB}}{3pt}} \, \e^2(1+\hat r_1(\e,\mu)) -  \te_{b}^2 \, \mu^2(1+\hat r_2(\e,\mu))}_{=:   \Delta_{\scaleto{\mathrm{BF}}{3pt}}(\mu,\e)}
 }
  $$
  with $r, \hat r_1, \hat r_2$ in $\cC^{\reg-2}$ and real valued. 
   In particular, they have non trivial real part whenever $\mu \neq 0$ and the Benjamin-Feir discriminant $\Delta_{\scaleto{\mathrm{BF}}{3pt}}(\mu,\e) >0$.
  
\end{theorem}

Theorem 
\ref{mainres1} is a direct consequence of 
Theorem \ref{mainres2}, as we now show.
\begin{proof}[Proof of Theorem \ref{mainres1} assuming Theorem \ref{mainres2}]
  The curve $\underline{\mu}(\e)$ in Theorem \ref{mainres1} is implicitly defined by
     $$
     \mu \sqrt{1+ \hat r_2(\e,\mu)} =\e  \sqrt{\dfrac{ \te_{\scaleto{\mathrm{WB}}{3pt}}}{\te_b^2}}\,  \, \sqrt{1+\hat r_1(\mu,\e)}\ ,
     $$
   which has a unique solution of class $\cC^{\reg-2}$ provided $\reg \geq 3$.
\end{proof}

    \section{Perturbative approach to the separated eigenvalues}\label{sec:Kato}
Recall that the operator $ \cL_{\mu,\e}  : Y \subset X \to X $   
has domain $Y:=H^{m_\star}(\mathbb{T},\C)$, with $m_\star=\max \{1,1+m \}$ and $m$ as in Equation (\ref{orderpsedodiff}), and range $X:=L^2(\mathbb{T}):=L^2(\mathbb{T},\C)$.\\
We now state the following Lemma, that is a variant of Lemma 3.1 in \cite{BMV1}. We denote by $B(r)$, $r>0$, the ball of center 0 and radius $r$ in $\R$.
\begin{lemma}\label{lem:Kato1}
{\bf (Kato theory for separated eigenvalues)}
 Let $\Gamma$ be a closed, counterclockwise-oriented curve around $0$ in the complex plane separating $\sigma'\left(\cL_{0,0}\right)=\{0\}$
  and the other part of the spectrum $\sigma''\left(\cL_{0,0}\right)$ in \eqref{spettrodiviso0}.
There exist $\e_0, \mu_0>0$  such that for any $(\mu, \e) \in B(\mu_0)\times B(\e_0)$ (where $B(x,r)$ is the ball centered at $x$ with radius $r$ inside $\R$)  the following statements hold:
\\[1mm] 
1. The curve $\Gamma$ belongs to the resolvent set of 
the operator $\cL_{\mu,\e} : Y \subset X \to X $ defined in \eqref{Lmue}.
\\[1mm] 
2.
The operators
\begin{equation}\label{Pproj}
 P_{\mu,\e} := -\frac{1}{2\pi\im}\oint_\Gamma (\cL_{\mu,\e}-\lambda)^{-1} \di \lambda : X \to Y 
\end{equation}  
are well defined projectors commuting  with $\cL_{\mu,\e}$,  i.e. 
$ P_{\mu,\e}^2 = P_{\mu,\e} $ and 
$ P_{\mu,\e}\cL_{\mu,\e} = \cL_{\mu,\e} P_{\mu,\e} $. 
The map $(\mu, \epsilon)\mapsto P_{\mu,\epsilon}$ is of class $\cC^\reg$.
\\[1mm] 
3.
The domain $Y$  of the operator $\cL_{\mu,\e}$ decomposes as  the direct sum
\begin{equation}\label{projdec}
    Y= \mathcal{V}_{\mu,\e} \oplus \text{Ker}(P_{\mu,\e}) \, , \quad \mathcal{V}_{\mu,\e}:=\text{Rg}(P_{\mu,\e})=\text{Ker}(\uno-P_{\mu,\e}) \, ,
\end{equation}
of   closed invariant  subspaces, namely 
$ \cL_{\mu,\e} : \mathcal{V}_{\mu,\e} \to \mathcal{V}_{\mu,\e} $, $
\cL_{\mu,\e} : \text{Ker}(P_{\mu,\e}) \to \text{Ker}(P_{\mu,\e}) $.  
Moreover 
$$
\begin{aligned}
&\sigma(\cL_{\mu,\e})\cap \{ z \in \C \mbox{ inside } \Gamma \} = \sigma(\cL_{\mu,\e}\vert_{{\mathcal V}_{\mu,\e}} )  = \sigma'(\cL_{\mu, \e}) , \\
&\sigma(\cL_{\mu,\e})\cap \{ z \in \C \mbox{ outside } \Gamma \} = \sigma(\cL_{\mu,\e}\vert_{Ker(P_{\mu,\e})} )  = \sigma''( \cL_{\mu, \e}) \, .
\end{aligned}
$$
\\[1mm] 
4.  The projectors $P_{\mu,\e}$ 
are similar one to each other: the  transformation operators
\begin{equation} \label{OperatorU} 
U_{\mu,\e}   := 
\big( \uno-(P_{\mu,\e}-P_{0,0})^2 \big)^{-1/2} \big[ 
P_{\mu,\e}P_{0,0} + (\uno - P_{\mu,\e})(\uno-P_{0,0}) \big] 
\end{equation}
are bounded and  invertible in $ Y $ and in $ X $, with inverse
$$
U_{\mu,\e}^{-1}  = 
 \big[ 
P_{0,0} P_{\mu,\e}+(\uno-P_{0,0}) (\uno - P_{\mu,\e}) \big] \big( \uno-(P_{\mu,\e}-P_{0,0})^2 \big)^{-1/2} \, , 
$$
 and 
$$ 
U_{\mu,\e} P_{0,0}U_{\mu,\e}^{-1} =  P_{\mu,\e}   \ , 
\qquad 
U_{\mu,\e}^{-1} P_{\mu,\e}  U_{\mu,\e} = P_{0,0}  \ .
$$
The map $(\mu, \epsilon)\mapsto  U_{\mu,\e}$ is of class $\cC^\reg$.
\\[1mm] 
5. The subspaces $\mathcal{V}_{\mu,\e}=\text{Rg}(P_{\mu,\e})$ are isomorphic one to each other: 
$
\mathcal{V}_{\mu,\e}=  U_{\mu,\e}\mathcal{V}_{0,0}.
$
 In particular $\dim \mathcal{V}_{\mu,\e} = \dim \mathcal{V}_{0,0}=3 $, for any 
 $(\mu, \e) \in B(\mu_0)\times B(\e_0)$.
\end{lemma}
\begin{proof}
The proof follows the same lines as Lemma 3.1 in \cite{BMV1}, replacing the analyticity of the map $(\mu,\e) \mapsto \cL_{\mu,\e}$ with the regularity $\cC^\reg$.
\end{proof}

We now prove that Kato's transformation operator $U_{\mu,\e}$ is symplectic and reversibility preserving. 
We first introduce, for $\mu \neq 0$  the inverse operator
\begin{equation}\label{Emu}
\cE_\mu := \cJ_\mu^{-1} \colon  L^2(\T)\to H^{1}(\T) , \qquad \mu >0   \ , 
\end{equation}
which acts as a Fourier multiplier of symbol $\dfrac{1}{\im (\xi+ \mu)} $.
When $\mu = 0$, 
 the operator $\cJ_0 = \pa_x$ is invertible in the  
subspace 
$$H^s_0(\T):= \{ u \in H^s(\T)\colon \int_\T u(x) \di x = 0 \} \ . $$
Hence, we define 
 \begin{equation}
 \cE_0 = \cJ_0^{-1} : L^2_0 (\T) \to   H^{1}_0 (\T) \ .
 \end{equation}
We now define  $\mu$-symplectic and reversible basis.
\begin{definition}[$\mu$-symplectic and reversible basis]\label{def:symp}

A linearly independent set $\{ f_1^+,f_1^-,f_0\}$ is 
\begin{itemize}
    \item[(i)] {\em $\mu$-symplectic}, $\mu \neq 0$,  if  one has 
            \begin{equation} \label{Emubasis}
            \begin{aligned}
                   & (\cE_\mu f_1^+ ,f_1^+)=(\cE_\mu f_1^- ,f_1^-)=\im\frac{\mu}{1-\mu^2} , 
   \qquad     (\cE_\mu f_0 ,f_0)= - \frac{\im}{\mu}               \ , 
                   \\
                    &
                    (\cE_\mu f_1^+ ,f_1^-)= \frac{1}{1-\mu^2} \ , \qquad    (\cE_\mu f_0 , f_1^+) =(\cE_\mu f_0 , f_1^-) =0  \ , 
                    \end{aligned}
                    \end{equation}
    and {\em $0$-symplectic} if both $f_1^+,f_1^- \in L_0^2(\T)$, $f_0 \in \ker\cJ_0 = \textup{span}\la 1 \ra$ and 
    \begin{equation}\label{E0basis}
                    (\cE_0 f_1^+, f_1^-) = 1,\hspace{0,3cm}(\cE_0 f_1^+, f_1^+) = (\cE_0 f_1^-, f_1^-) = 0 \ ;
            \end{equation}
    \item[(ii)] {\em reversible} if \hspace{0,2cm} $\bar \rho f_1^+ = f_1^+,$ \hspace{0,2cm} $\bar \rho f_1^- = -f_1^-,$\hspace{0,2cm} $\bar \rho f_0 = f_0$, and $\bar \rho$ 
    is the complex involution in \eqref{barrho}. 
\end{itemize}
\end{definition}
\begin{remark}
    The basis $ f_1^+ = \cos x$, $f_1^- =  \sin x$, $f_0 =  \frac{1}{\sqrt{2}}$ is $\mu$-symplectic for any $\mu \in [0, \frac12)$.
\end{remark}
In the previous Lemma \ref{lem:Kato1}, it has been shown that the operator $U_{\mu,\e}$ is an isomorphism between $\cV_{0,0}$ and $\cV_{\mu,\e}$. In the following Lemma we show that, in addition, the previous isomorphism preserves the symplectic and reversible  structure of a basis.
Here the properties are different with respect to the case of \cite{BMV1}, since here the symplectic tensor $\cJ_\mu$ is not invertible at  $\mu = 0$. 
\begin{lemma} \label{lemmanonzero}
For any $(\mu,\e)\in B(\mu_0)\times B(\e_0)$ we have that 
\begin{itemize}
    \item[(i)] $P_{\mu,\e}$ is \textit{skew-Hamiltonian}, i.e. 
   \begin{equation}\label{Psym}
    \cE_\mu P_{\mu,\e}= P_{\mu,\e}^* \cE_\mu , \quad \mbox{for } \mu \neq 0, \qquad 
        \cE_0 P_{0,\e}= P_{0,\e}^* \cE_0 \quad \mbox{ in } L^2_0(\T) , 
    \end{equation}
    and reversibility preserving, i.e.
\begin{equation}\label{Prev}
    \bar \rho P_{\mu,\e} = P_{\mu,\e}  \bar \rho \ .
  \end{equation}
  Moreover,   $P_{0,\e}[1] = 1$, $P_{0,\e} \cJ_0 = \cJ_0 P_{0,\e}^*$ and thus $P_{0,\e}$ leaves $L^2_0(\T)$ invariant.
    \item[(ii)] 
    $U_{\mu,\e}$ are symplectic, i.e. 
    \begin{equation}\label{Usymp}
    U_{\mu,\e}^* \cE_\mu U_{\mu,\e} = \cE_\mu   \quad \mbox{for } \mu \neq 0 \ ,        
    \qquad  U_{0,\e}^* \cE_0 U_{0,\e} = \cE_0 \quad \mbox{ in } L^2_0(\T)  \ ,
    \end{equation}
     and reversibility preserving. 
  Moreover,   $U_{0,\e}[1] = 1$, $U_{0,\e} \cJ_0 = \cJ_0 U_{0,\e}^{-*}$ and thus $U_{0,\e}$ leaves $L^2_0(\T)$ invariant.
    \item[(iii)] 
    $P_{0,\e}$ and $U_{0,\e}$ are real operators, i.e. 
    $    \overline{P_{0,\e}} = P_{0,\e}, \hspace{0,3cm} \overline{U_{0,\e}} = U_{0,\e}$.
    \item[(iv)] The operators  $\dot  P_{0,\e}:= \pa_\mu P_{\mu,\e}\big|_{\mu=0} $   and $\dot U_{0,\e}:= \pa_\mu U_{\mu,\e}\vert_{\mu = 0}$ are  purely imaginary.
\end{itemize}
\end{lemma}
\begin{proof}
$(i)$ First we have that for any $\mu \in [0, \frac12)$
   \begin{equation}\label{lemmaeq0}
    \cL_{\mu,\e} \cJ_\mu = \cJ_\mu \cB_{\mu,\e} \cJ_\mu = \cJ_\mu ( \cJ_\mu^* \cB_{\mu,\e}^* )^* = \cJ_\mu ( -\cJ_\mu \cB_{\mu,\e} )^* = -\cJ_\mu \cL_{\mu,\e}^* \ ,
   \end{equation}
   which implies that 
    \begin{equation} \label{lemmaeq1}
    (\cL_{\mu,\e} - \lambda ) \cJ_\mu = -\cJ_\mu (\cL_{\mu,\e}^* + \lambda) , \quad \forall \mu \in \big[0, \frac12 \big) \ .
    \end{equation}
Consider now $\mu > 0$. Then  $\cJ_\mu$ is invertible and 
from  \eqref{lemmaeq1} we deduce that $\lambda\in \rho (\cL_{\mu,\e}) \iff -\lambda \in \rho (\cL_{\mu,\e}^*)$; so inverting \eqref{lemmaeq1} and integrating along a circuit  $\Gamma\subseteq \rho (\cL_{\mu,\e})$ we obtain that
    \begin{equation}
       \cE_\mu P_{\mu,\e} = \cE_\mu \oint_\Gamma (\cL_{\mu,\e}-\lambda)^{-1} \frac{-\di \lambda}{2\pi\im}  = \oint_\Gamma (\cL_{\mu,\e}^*+\lambda)^{-1} \frac{\di \lambda}{2\pi\im} \circ \cE_\mu   =  P_{\mu,\e}^* \, \cE_\mu \ ,
    \end{equation}
where we used that 
$    P_{\mu,\e}^* =   \oint_\Gamma (\cL_{\mu,\e}^*+ \lambda)^{-1} \frac{\di \lambda}{2\pi \im}$. 
    This proves \eqref{Psym} in case $\mu > 0$.

Take now  $\mu = 0$. First we have that 
$P_{0,\e}[1] = 1$  $\forall \e$, since $1\in \cV_{0,\e}= Rg P_{0,\e}$ (see Remark \ref{rem:V0e}). 
Next consider \eqref{lemmaeq1} at $\mu = 0$, apply $(\cL_{0,\e} - \lambda)^{-1}$ to the left, $(\cL_{0,\e}^* + \lambda)^{-1}$ to the right and integrate along the circuit $\Gamma$ to get the identity  $P_{0,\e} \cJ_0 = \cJ_0 P_{0,\e}^*$. 
This last identity implies that $P_{0,\e}$ leaves $L^2_0(\T)$ invariant; indeed if $u \in L^2_0(\T)$ then 
\begin{equation}\label{P1}
( P_{0,\e} u, 1 ) = (  P_{0,\e}  \cJ_0 \, \cE_0 u, 1 ) 
= 
(  \cJ_0 P_{0,\e}^*  \, \cE_0 u, 1 )  = - ( P_{0,\e}^*  \, \cE_0 u,   \cJ_0[1] ) = 0 \ , 
\end{equation}
showing that $P_{0,\e} \colon L^2_0(\T) \to L^2_0(\T)$. Then, on $L^2_0(\T)$, apply $\cE_0$ to the left and to the right of the identity
$P_{0,\e} \cJ_0 = \cJ_0 P_{0,\e}^*$ to get the second of \eqref{Psym}.
    Property \eqref{Prev} follows as in \cite[Lemma 3.2]{BMV1}.

$(ii)$ Consider first the case $\mu > 0$. By \eqref{Bmue}, 
 $\cL_{0,0} = \pa_x (\fm(1) - \cM(D))$,  hence 
 $\cL_{0,0} \cJ_\mu = -\cJ_\mu \cL_{0,0}^*$. Then arguing as above we deduce that 
 $\cE_\mu P_{0,0}= P_{0,0}^* \cE_\mu$.
 Using also \eqref{Psym} we deduce 
    \begin{equation}\label{EU}
    \begin{split}
        \cE_\mu U_{\mu,\e} = \cE_\mu \big( \uno-(P_{\mu,\e}-P_{0,0})^2 \big)^{-1/2} \big[ 
P_{\mu,\e}P_{0,0} + (\uno - P_{\mu,\e})(\uno-P_{0,0}) \big] \\
=\big( \uno-(P_{\mu,\e}^*-P_{0,0}^*)^2 \big)^{-1/2} \big[ 
P_{\mu,\e}^*P_{0,0}^* + (\uno - P_{\mu,\e}^*)(\uno-P_{0,0}^*) \big] \cE_\mu \\
=\big( \uno-(P_{\mu,\e}-P_{0,0})^2 \big)^{-*/2} \big[ 
P_{0,0} P_{\mu,\e}+ (\uno-P_{0,0})(\uno - P_{\mu,\e}) \big]^* \cE_\mu\\
= \left(  \big[ 
P_{0,0} P_{\mu,\e}+(\uno-P_{0,0}) (\uno - P_{\mu,\e}) \big] \big( \uno-(P_{\mu,\e}-P_{0,0})^2 \big)^{-1/2} \right)^*\cE_\mu
\end{split}
    \end{equation}
    By Lemma \ref{lem:Kato1}, point $4$, this is equal to $U^{-*}_{\mu,\e} \cE_\mu$, concluding the proof of \eqref{Usymp} in case $\mu \neq 0$.
    
  Consider now  $\mu = 0$. Using the  expression \eqref{OperatorU} of  $U_{0,\e}$ and the fact that $P_{0,\e}[1]=1$ $\forall \epsilon$,    we obtain  that $U_{0,\e}[1] =1$. 
  Then, using  $P_{0,\e} \cJ_0 = \cJ_0 P_{0,\e}^*$ and arguing as in \eqref{EU} we get  $U_{0,\e} \cJ_0 = \cJ_0 U_{0,\e}^{-*} $, from which we deduce  that   $U_{0,\e} $ leaves $L^2_0(\T) $ invariant. 
  The second of \eqref{Usymp} follows easily from $U_{0,\e} \cJ_0 = \cJ_0 U_{0,\e}^{-*} $ on $L^2_0(\T)$.

 $(iii)$ It follows as in  \cite[Lemma 3.2]{BMV1}.

 $(iv)$  By \eqref{Pproj}, 
     \begin{equation*}
        \dot P_{0,\e}  = \frac{1}{2\pi\im} \oint_\Gamma (\cL_{0,\e}-\lambda)^{-1} \dot \cL_{0,\e} (\cL_{0,\e}-\lambda)^{-1} \de\lambda 
    \end{equation*}
Since $\cL_{0,\e}$ is real, 
    $$
    \overline{\dot P_{0,\e}} = -\frac{1}{2\pi \im}\oint_{\bar\Gamma} (\cL_{0,\e}-\lambda)^{-1} \overline{\dot \cL_{0,\e}} (\cL_{0,\e}-\lambda)^{-1} \de\lambda  \ ,
    $$
    where $\bar \Gamma$ denotes the curve $\Gamma$ with inverted orientation. Since $\dot \cL_{0,\e}$ is purely imaginary by Lemma \ref{lem:LB.im}, in conclusion $  \overline{\dot P_{0,\e}} = - \dot P_{0,\e} $, proving that $\dot P_{0,\e}$ is purely imaginary. 
   
Then consider  $\dot U_{0,\e}$. 
Denote $g(A):= (\uno - A^2)^{-\frac12}$, defined for $\norm{A}< 1$. Then 
$\bar{g(A)} = g(\bar A)$, and differentiating one gets $\overline{\di g(A)[B]} = \di g(\bar A)[\bar B]$. Thus,
$$
\dot U_{0,\e} = \di g(P_{0,\e} - P_{0,0})[\dot P_{0,\e}] \left( 
 P_{0,\e} P_{0,0} + (\uno - P_{0,\e})(\uno - P_{0,0}) \right) + 
 g(P_{0,\e} - P_{0,0}) \, \dot P_{0,\e} (2 P_{0,0} - \uno)
$$
Using that $P_{0,\e}$ is real and $\dot P_{0,\e}$ is purely imaginary, one gets the thesis.
\end{proof}

%
%\begin{remark}
%Even though $U_{\mu,\e}$ maps $\cV_{0,0}$ to $\cV_{\mu,\e}$, it takes $\mu$-symplectic basis to $\mu$-symplectic basis, not $0$-symplectic to $\mu$-symplectic ones. This is fine since we defined a $\mu$-symplectic basis in such a way that the basis $\{ f_1^+=\cos x,\text{ }f_1^-=\sin x, \text{ }f_0=1\}$ is $\mu$-symplectic for any $\mu$.
%\end{remark}

In the next lemma we show how a general vector in $\cV_{\mu,\e}$ is decomposed on a $\mu$-symplectic basis.
\begin{lemma} \label{jdecomposition}
    Let $\{ f_1^+,f_1^-,f_0 \}$ a $\mu$-symplectic basis of $\cV_{\mu,\e}$ according to Definition \ref{def:symp}.
     Then for every $f\in \cV_{\mu,\e}$, it holds that
     \begin{itemize}
     \item[(i)] if $\mu \neq 0$, then $f=\alpha_1^+ f_1^+ + \alpha_1^- f_1^- + \alpha_0 f_0$
     with
         \begin{equation}\label{coeff.basis}
        \begin{pmatrix}
            \alpha_1^+ \\
            \alpha_1^- \\
            \alpha_0
        \end{pmatrix}
        =
        \begin{pmatrix}
            \im\mu & 1 & 0 \\
            -1 & \im\mu & 0 \\
            0 & 0 & \im\mu 
        \end{pmatrix}
        \cdot 
        \begin{pmatrix}
            (\cE_\mu f, f_1^+) \\
            (\cE_\mu f, f_1^-) \\
            (\cE_\mu f, f_0)  
        \end{pmatrix} \ ; 
    \end{equation}
    \item[(ii)] if $\mu  = 0$, then, denoting $\breve f:= f -  (f,f_0) f_0 \in L^2_0(\T)$, one has 
     \begin{equation}
        f = (\cE_0 \breve f, f_1^- )f_1^+ -   (\cE_0 \breve f, f_1^+)f_1^- + (f, f_0) f_0  \ . 
    \end{equation}
     \end{itemize}
\end{lemma}
\begin{proof} $(i)$ Write $f=\alpha_1^+ f_1^+ + \alpha_1^- f_1^- + \alpha_0 f_0$ with some coefficients $\alpha_1^\pm$, $\alpha_0$ to be determined. Taking the inner product of $\cE_\mu f$ with every element of the basis  gives the system 
    \begin{equation*}
        \begin{pmatrix}
            (\cE_\mu f, f_1^+) \\
            (\cE_\mu f, f_1^-) \\
            (\cE_\mu f, f_0) 
        \end{pmatrix}
        =
        \begin{pmatrix}
            ( \cE_\mu f_1^+ , f_1^+) & ( \cE_\mu f_1^- , f_1^+) & ( \cE_\mu f_0 , f_1^+) \\
            ( \cE_\mu f_1^+ , f_1^-) & ( \cE_\mu f_1^- , f_1^-) & ( \cE_\mu f_0 , f_1^-) \\
            ( \cE_\mu f_1^+ , f_0) & ( \cE_\mu f_1^- , f_0) & ( \cE_\mu f_0 , f_0) 
        \end{pmatrix}
        \cdot
        \begin{pmatrix}
            \alpha_1^+ \\
            \alpha_1^- \\
            \alpha_0
        \end{pmatrix}
       \stackrel{ \eqref{Emubasis}}{ =}
        \begin{pmatrix}
            \im \frac{\mu}{1-\mu^2} & -\frac{1}{1-\mu^2} & 0 \\
            \frac{1}{1-\mu^2} & \im \frac{\mu}{1-\mu^2} & 0 \\
            0 & 0 & -\frac{\im}{\mu}
        \end{pmatrix}
        \cdot
        \begin{pmatrix}
            \alpha_1^+ \\
            \alpha_1^- \\
            \alpha_0
        \end{pmatrix}
    \end{equation*}
    Inverting the previous relation gives \eqref{coeff.basis}.
    
    $(ii)$  By definition of $0$-symplectic basis,  $f_1^\pm \in L^2_0(\T) \cap Y$, whereas $f_0$ is a constant function.
    Writing $\breve f  = \alpha_1^+ f_1^+ + \alpha_1^- f_1^- + \alpha_0 f_0$ and taking the scalar product with $f_0$, we deduce that $\alpha_0 = 0$.
     Then take  the scalar product of $\cE_0 \breve f$ with   $f_1^+$ and $f_1^-$ and use \eqref{E0basis} to get $(ii)$.
\end{proof}

 Next we outline a property of a reversible basis.  
We denote by $even(x)$ a real $2\pi$-periodic function which is even in $x$, and by
$odd(x)$ a real $2\pi$-periodic function which is odd in $x$.

\begin{lemma} \label{paritybasis}
    Let $\{ f_1^+,f_1^-,f_0 \}$ be a reversible basis according to Definition \ref{def:symp}. Then 
    \begin{equation}\label{parity}
    f_1^+(x) = even(x) + \im \, odd(x) \  , \quad 
    f_1^-(x) = odd(x) + \im\,  even(x)  \ , 
\quad    
    f_0(x) = even(x) + \im \, odd(x)  \ . 
    \end{equation}
\end{lemma}
\begin{proof}  Using the definition of the involution $\bar \rho$ in 
\eqref{barrho} one has 
  $f_1^+(x) = a(x) + \im b(x) = \bar \rho f_1^+(x) = a(-x) - \im b(-x)$, implying that $a(x)$ is even and $b(x)$ odd. The other cases are analogous.
\end{proof}

In the next lemma we  express the operator $\cL_{\mu,\e}\big|_{\cV_{\mu,\e}}$ on any $\mu$-symplectic and reversible basis according to Definition \ref{def:symp}.
\begin{lemma}[Matrix representation of $\cL_{\mu,\e}$ on $\cV_{\mu,\e}$] \label{matrixraplemma}
    The $3\times 3$ matrix that represents the Hamiltonian and reversible operator $\cL_{\mu,\e} = \cJ_\mu \cB_{\mu,\e}$ with respect to any $\mu$-symplectic and reversible basis $\tF$ of $\cV_{\mu,\e}$, (according to Definition \ref{def:symp}) is 
    \begin{equation}\label{tL}
        \tL_{\mu,\e} = \tJ_\mu \tB_{\mu,\e}
    \end{equation}
    where
    \begin{equation}\label{tJtB}
        \tJ_\mu := 
        \begin{pmatrix}
            \im\mu & 1 & 0\\
            -1 & \im\mu & 0\\
            0 & 0 & \im\mu 
        \end{pmatrix},
        \qquad 
        \tB_{\mu,\e} := 
        \begin{pmatrix}
            (\cB_{\mu,\e} f_1^+ , f_1^+) & (\cB_{\mu,\e} f_1^- , f_1^+) & (\cB_{\mu,\e} f_0 , f_1^+)\\
            (\cB_{\mu,\e} f_1^+ , f_1^-) & (\cB_{\mu,\e} f_1^- , f_1^-) & (\cB_{\mu,\e} f_0 , f_1^-)\\
            (\cB_{\mu,\e} f_1^+ , f_0) & (\cB_{\mu,\e} f_1^- , f_0) & (\cB_{\mu,\e} f_0 , f_0)
        \end{pmatrix} \ .
    \end{equation}    
    Moreover, the entries of the matrix $\tB_{\mu,\e}$ are alternatively real or purely imaginary:
    \begin{equation}\label{Brealim}
        (\cB_{\mu,\e} f_k^\sigma , f_{k'}^\sigma) \in \R,\qquad (\cB_{\mu,\e} f_k^\sigma , f_{k'}^{-\sigma})\in \im\R \ .
    \end{equation}
\end{lemma}

\begin{proof}
We first consider  $\mu \neq 0$. By Lemma 
  \ref{jdecomposition} $(i)$ and using that $\cE_\mu\cL_{\mu,\e} = \cB_{\mu,\e}$
  we have
  $\cL_{\mu,\e} f=\alpha_1^+ f_1^+ + \alpha_1^- f_1^- + \alpha_0 f_0$
  with 
    \begin{equation}
        \begin{pmatrix}
            \alpha_1^+ \\
            \alpha_1^- \\
            \alpha_0
        \end{pmatrix}
        =
        \begin{pmatrix}
            \im\mu & 1 & 0 \\
            -1 & \im\mu & 0 \\
            0 & 0 & \im\mu 
        \end{pmatrix}
        \cdot 
        \begin{pmatrix}
            (\cB_{\mu, \e} f, f_1^+) \\
            (\cB_{\mu,\e} f, f_1^-) \\
            (\cB_{\mu,\e} f, f_0) 
        \end{pmatrix} \ , 
    \end{equation}
hence the matrix representing the action $\cL_{\mu,\e} \colon \cV_{\mu,\e} \to \cV_{\mu,\e}$ is given by \eqref{tL}--\eqref{tJtB}.
    
Now let $\mu = 0$. Recall that by definition of $0$-symplectic basis, $f_0$ is a constant function. Then $(\cL_{0,\e}f, f_0) = - (\cB_{0,\e} f, \cJ_0 f_0) =  0$, hence 
 Lemma \ref{jdecomposition} gives 
    $\cL_{0,\e} f=\alpha_1^+ f_1^+ + \alpha_1^- f_1^- $
  with 
    \begin{equation}
        \begin{pmatrix}
            \alpha_1^+ \\
            \alpha_1^- 
        \end{pmatrix}
        =
        \begin{pmatrix}
            0 & 1  \\
            -1 & 0 
        \end{pmatrix}
        \cdot 
        \begin{pmatrix}
            (\cB_{0, \e} f, f_1^+) \\
            (\cB_{0,\e} f, f_1^-) 
        \end{pmatrix} \ , 
    \end{equation}
  hence the matrix representing the action $\cL_{0,\e} \colon \cV_{0,\e} \to \cV_{0,\e}$ is given by \eqref{tL}--\eqref{tJtB} with $\mu = 0$.
    
    To prove \eqref{Brealim} we observe that  $(f, g) = \overline{(\bar \rho f, \bar \rho g)}$ for any $f,g \in L^2(\T)$. Then by \eqref{B.rev} and the properties of  reversible basis
    \begin{equation}
 (\cB_{\mu,\e} f_k^\sigma , f_{k'}^{\sigma'}) = 
 \overline{( \bar\rho \cB_{\mu,\e}  f_k^\sigma , \bar\rho f_{k'}^{\sigma'})} = 
       \overline{     ( \cB_{\mu,\e} \bar \rho f_k^\sigma , \bar \rho f_{k'}^{\sigma'})} = 
       \sigma\sigma' \overline{     ( \cB_{\mu,\e}  f_k^\sigma , f_{k'}^{\sigma'})} \ , 
    \end{equation}
    proving  \eqref{Brealim}.
    \end{proof}

    It will be important in the following to act with coordinate transformations that preserve the symplectic and reversible structure of $\tL_{\mu,\e}$. This necessity inspires the following definition:
\begin{definition}[$\mu$-symplectic and reversible matrix]\label{def:sympm}
A $3\times 3$ matrix $\tL = \tJ \tB$ is 
\begin{itemize}
    \item {\em complex Hamiltonian}, if $\tB$ is self-adjoint and $\tJ$ is skew-adjoint with respect to the standard scalar product of $\C^3$;
    \item {\em reversible}, if 
$\tL\circ \bar\rho = -\bar\rho \circ \tL$, where     
    $$
    \bar\rho = \begin{pmatrix}
        \mathfrak{c} & 0 & 0\\
        0 & -\mathfrak{c} & 0 \\
        0 & 0 & \mathfrak{c}
    \end{pmatrix}
    $$
and  $\mathfrak{c}: z \mapsto \bar z$ is the  conjugation in the complex plane. 
\end{itemize}
Moreover, we say that an invertible matrix $Y$ is $\tJ$-{\em symplectic} if 
$    Y \tJ  Y^* = \tJ$.
\end{definition}
In our context we will  have that $\tJ$ is reversible and $\tB$ is reversibility preserving, i.e. 
$ \tB\circ \bar\rho = \bar\rho \circ \tB$.

\begin{lemma}\label{expontentiallemma}
    Let $\Sigma$ be a self-adjoint and reversible matrix and $\tJ$ be a skew-adjoint and reversible matrix.  Then $\forall \tau\in\R$, $\exp{(\tau \tJ \Sigma)}$ is $\tJ$-symplectic and reversibility preserving.
\end{lemma}
\begin{proof}
    Let $\varphi (\tau) = \exp{(\tau \tJ \Sigma)}$.
     Let us prove that $\psi(\tau):=\varphi (\tau) \tJ \varphi (\tau)^*-\tJ =0 $. Indeed,  $\psi(0)=0$ and 
    $$
    \frac{\de}{\de\tau}\psi(\tau) = \varphi (\tau) (\tJ \Sigma \tJ + \tJ \Sigma^* \tJ^* ) \varphi (\tau) = 0.
    $$
    To show the reversibility preserving property, let us observe that $\tJ \Sigma$ is reversibility preserving since both $\tJ$ and $\Sigma$ are reversible.
    Then $\exp{(\tau \tJ \Sigma)} = \sum_{n\geq 0 } \frac{\tau^n}{n!}(\tJ \Sigma)^n
    $  is reversibility preserving as well.
\end{proof}

\section{Matrix representation of $\cL_{\mu,\e}$ on $\cV_{\mu,\e}$}\label{matrixrapsection}
In this section we use the transformation operators $U_{\mu,\e}$ obtained in the previous section to 
construct a symplectic and reversible basis of $\cV_{\mu,\e}$, and we compute in Proposition \ref{matrixrepprop} the $3 \times 3$ 
Hamiltonian and reversible matrix representing the action of $\cL_{\mu,\e}\colon \cV_{\mu,\e} \to \cV_{\mu,\e}$ on such basis.
The symplectic and reversible basis of $\cV_{\mu,\e}$  we  choose is 
$\tF=\{ f_1^+(\mu,\e),f_1^-(\mu,\e),f_0(\mu,\e) \}$, with 
\begin{equation}\label{fksigma}
f_k^\sigma (\mu,\e) := U_{\mu,\e}f_k^\sigma, \qquad 
f_1^+ = \cos x, \ \ f_1^- = \sin x, \ \ f_0^+ =\frac{1}{\sqrt 2}  \ .
\end{equation}

\begin{lemma}\label{lem:fks}
    The  Kato basis $\{f_1^+ (\mu,\e), f_1^- (\mu,\e), f_0 (\mu,\e) \}$ in \eqref{fksigma} is $\mu$-symplectic  $\forall (\mu, \e) \in B(\mu_0)\times B(\e_0)$ and  reversible, according to Definition \ref{def:symp}. Each map $(\mu,\e)\mapsto f_k^\sigma(\mu,\e)$ is $\cC^\reg$ as a map $B(\mu_0)\times B(\e_0) \to H^{m_\star}(\T)$.
\end{lemma}
\begin{proof}
  The basis $\{ f_1^\pm, f_0 \}$ is $\mu$-symplectic $\forall \mu$ and reversible. Then the result follows from Lemma \ref{lemmanonzero}.
\end{proof}

In the next lemma we expand the vectors $f_k^\sigma(\mu,\e)$ in $\mu,\e$. 
\begin{lemma}[Expansion of the basis $\tF$] \label{totalbasisexpansion}
For small values of $\mu,\e$ the vectors $f_k^\sigma(\mu,\e)$ in \eqref{fksigma} expand as
    \begin{equation}\label{basis.exp}
        \begin{aligned}
            & f_1^+ (\mu, \e) = \cos x + \e \, \fa \cos (2x) -\im \mu \e \, \fb \sin (2x)+\e^2 even_0(x)  +\cO^\reg (\mu^2\e,\mu\e^2, \e^3) \\
            & f_1^- (\mu, \e) = \sin x + \e \, \fa \sin (2x) + \im \mu\e \, \fb \cos (2x) + \e^2 odd(x)  + \cO^\reg(\mu^2\e,\mu\e^2,\e^3) \\
&f_0^+ (\mu,\e) = \frac{1}{\sqrt 2} + \cO^\reg(\mu^2\e,\mu\e^2)
        \end{aligned}
    \end{equation}
    with $\fb$ a real number and 
\begin{equation}\label{fa}
\fa:= \frac{1}{\fm(1) - \fm(2)} \  , 
\end{equation}
and  $even_0(x)$ is a real valued,  even function with zero average, and the $\cO(\cdot)$s are vectors in $H^{m_\star}(\T)$. 
Moreover, for any $|\e| \leq \e_0$, $f_0^+(0,\e) = f_0^+$.

Finally, one has the expansions 
\begin{equation}\label{basis.der}
\dot f_1^- (0,\e) = \im \e \fb \cos (2x)+\cO(\e^2) , \quad  
 \dot f_0(0,\e) = \cO(\e^2) \ , \quad 
   \ddot f_k^\sigma(0,\e) = \cO(\e) \ .
   \end{equation}
\end{lemma}
\begin{proof}
    The long computations are left in Appendix \ref{appendixA}.
\end{proof}
We now state the main result of this section:
\begin{proposition}\label{matrixrepprop}
The action of the Hamiltonian and reversible operator $\cL_{\mu,\e}$ in the symplectic and reversible basis $\{  f_1^+(\mu,\e),f_1^-(\mu,\e),f_0^+(\mu,\e) \}$ of $\cV_{\mu,\e}$, defined  in 
    \eqref{fksigma}, is represented by the $3\times 3$ Hamiltonian and reversible matrix 
    \begin{equation}\label{tLmue}
        \tL_{\mu,\e} = \tJ_\mu \tB_{\mu,\e},
    \end{equation}
    where $\tJ_\mu$ is in \eqref{tJtB} and  $\tB_{\mu,\e}$ is the self-adjoint and reversibility preserving  $3\times 3$-matrix
    \begin{equation}\label{matrixrapresentationeq}
        \tB_{\mu,\e}=\begin{pmatrix}
            E & \rvline &  \tf \\
            \hline 
            \tf^\dag & \rvline & g
        \end{pmatrix} \ .
    \end{equation}
Here $E$ is the $2\times 2$ selfadjoint matrix 
    \begin{equation}\label{E}
            E = 
             \begin{pmatrix}
                 -\fa \e^2(1+r_1'(\e)) + \te_{22}\mu^2 (1 + r_1'' (\e,\mu)) & \im \mu (\te_{12}  + r_2(\e^2, \mu\e,\mu^2)) \\
                -\im \mu (\te_{12}  + r_2(\e^2, \mu\e,\mu^2)) & \te_{22}\mu^2 (1+ r_4(\e,\mu)) 
            \end{pmatrix}\ , 
            \end{equation}
 the vector $\tf$ and the number  $g$ are given by 
 \begin{equation}\label{tf}
    \tf = 
            \begin{pmatrix}
            -\sqrt 2 \e  \left( 1  + r_3 (\e^2,\mu^2) \right)  \\
            \im \mu \e \,  r_5(\e,\mu)\\
            \end{pmatrix},\qquad
            g  = \te_{33} + \frac{ \fa}{2} \e^2+\tg_{33}\mu^2 + r_6 (\e^3, \mu^2\e,\mu^3) \ ,
\end{equation}
with  $\fa$ in \eqref{fa} and
    \begin{equation}\label{te}
\te_{22} := -\frac12 \ddot \fm (1), \qquad \te_{12}:= \dot \fm(1), \qquad 
 \te_{33} := \fm(1)-\fm(0),  \qquad \tg_{33} := -\frac12 \ddot \fm (0)   \, .
    \end{equation}
    The reminders $r_k$ are functions of class  $\cC^{\reg-2}$.
\end{proposition}
 The remaining part of the section will be devoted to the proof of Proposition \ref{matrixrepprop}.
 
\smallskip 
 
By Lemma \ref{lem:fks},  the matrix $\tB_{\mu,\e}:=(\cB_{\mu,\e}f_k^\sigma(\mu,\e), f_{k'}^{\sigma'}(\mu,\e))$  is of class $\cC^\reg$, $\reg \geq 3$. This allows us to expand the matrix  $\tB_{\mu,\e}$ in $\mu$, up to second order:
\begin{equation}\label{exptBmue}
    \tB_{\mu,\e} = \tB_{0,\e} + \mu\dot \tB_{0,\e} + \frac{\mu^2}{2}\ddot \tB_{0,\e} +  \mu^2 \varphi(\mu,\e) , \quad \varphi(\mu,\e):= \frac12 \int_0^1 (1-\tau)^2 \, \left( \ddot \tB_{\tau \mu, \e} - \ddot \tB_{0,\e} \right) \di \tau  \ \in \cC^{\reg -2} \ . 
\end{equation}
We now consider each term separately.

\smallskip
{\bf \noindent Expansion of $\tB_{0,\e}$.} We start with the matrix 
$\tB_{0,\e} =
\{  (\cB_{0,\e} f_k^\sigma (0,\e),  f_{k'}^{\sigma'} (0,\e) )  \}
 $. 

\begin{lemma}[Expansion of the matrix $\tB_{0,\e}$]\label{tB0exp1}
The  $3\times 3$ selfadjoint, real   and reversibility preserving matrix  $\tB_{0,\e} $ 
 expands as
 \begin{equation}\label{B0e}
 \tB_{0,\e} =
  \begin{pmatrix}
      - \fa 
\e^2 + r_1(\e^3)  & 0 &  -\sqrt 2 \e + r_3(\e^3) \\
        0 &  0  & 0 \\
       -\sqrt 2 \e + r_3(\e^3)  & 0 &  \fm(1) - \fm(0) +\frac12 \fa  \e^2 + r_6(\e^3)
    \end{pmatrix} \ ,
 \end{equation}
 with $\fa$ in \eqref{fa} and $r_k$ are functions of class  $\cC^{\reg}$ independent of $\mu$.
\end{lemma}
\begin{proof}
The matrix $\tB_{0,\e}$ is  real since $\cB_{0,\e}$ is a real operator and the vectors $f_k^\sigma(0,\e)$ are real, being $U_{0,\e}$ a real operator by Lemma 
\ref{lemmanonzero} $(iii)$.
Moreover,  by Lemma \ref{matrixraplemma},  the entries $(\cB_{0,\e} f_k^\sigma (0,\e),  f_{k'}^{-\sigma} (0,\e) )$ are purely imaginary,  so we get
\begin{equation}\label{tB0e}
    \tB_{0,\e} = 
    \begin{pmatrix}
        E_{11}(\e)  & 0 & E_{13}(\e)  \\
        0 &  E_{22}(\e) & 0 \\
       E_{13}(\e)  & 0 & E_{33}(\e)
    \end{pmatrix}
\end{equation}
for some real functions $E_{jk}(\e)$.

\smallskip 
\underline{ Step 1: proof that  $E_{22}(\e) \equiv 0$ for any $|\e|<\e_0$.}
Recall that, by Lemma  \ref{basis0,eps}, the operator $\cL_{0,\e}$
has  0 as eigenvalue with algebraic multiplicity 3 and geometric multiplicity 2. 
The matrix $\tJ_0 \tB_{0,\e}$, that represents the action of this operator on the generalized kernel $\cV_{0,\e}$, has therefore  two Jordan blocks: one of dimension $1$ and the other of dimension $2$. Then, we have that $\tJ_0 \tB_{0,\e}$ is nilpotent of index $2$, i.e.
    \begin{equation*}
    0 =     (\tJ_0 \tB_{0,\e})^2 = \begin{pmatrix}
            - E_{11}(\e) E_{22}(\e) & 0 & - E_{22}(\e)E_{13}(\e) \\
            0 & -E_{11}(\e) E_{22}(\e) & 0 \\
            0 & 0 & 0
        \end{pmatrix}  \ .
    \end{equation*}  
        We show below that  $E_{11}(\e)$ and $E_{13}(\e)$ are different from  zero for $\e$ sufficiently small,  therefore 
    \begin{equation}\label{E220}
    E_{22}(\e) \equiv  0  \ . 
    \end{equation}
    \underline{ Step 2: expansion of $E_{11}(\e)$, $E_{33}(\e)$ and $E_{13}(\e)$.} 
    We start by Taylor expanding  both the operator $\cB_{0,\e}$ and the elements of the basis $\tF$ in orders of $\e$:
\begin{equation}\label{expBf}
    \cB_{0,\e} = \cB_0 + \e \cB_1 + \e^2 \cB_2 + \cO^\reg(\e^3),  
    \qquad
      f_k^\sigma (0,\e) = f_{k}^\sigma +\e f_{k_1}^\sigma +\e^2 f_{k_2}^\sigma +\cO^\reg(\e^3),
\end{equation}
with 
\begin{equation}\label{expBf2}
        \cB_0:= \fm(1) - \cM (D),\qquad \cB_1:= -2\cos x, \qquad \cB_2:=c_2 -2 u_2^{[0]} -2 u_2^{[2]} \cos (2x)
\end{equation}
with  $c_2, u_2^{[0]}$, $ u_2^{[2]}$ given in Theorem \ref{existPTW} and 
 $f_{k_1}^\sigma$ in Lemma \ref{totalbasisexpansion}. 
We will see that $f_{k_2}^\sigma$ does not need to be computed explicitly, since it won't bring any relevant contribution in the expansion.
Moreover, a direct computation  gives 
\begin{equation} \label{actionexpansion}
\begin{aligned}
& \cB_0 f_{1}^\sigma = 0 \ ,\quad 
\cB_0 f_0^+ = (\fm(1) - \fm(0)) f_0^+  \ ,  \quad \cB_0 f_{1_1}^+ = (\fm(1) - \fm(2)) \fa \cos (2x)
\ ,  \\
& \cB_1 f_{1}^+=-(1+\cos (2x) ) , \quad \cB_1 f_{1_1}^+ = -\fa ( \cos x + \cos (3x) ) \ , 
\quad \cB_1 f_0^+ = - \sqrt{2} \cos(x) \ , 
\\
& \cB_2 f_{1}^+ = (c_2 -2u_2^{[0]} - u_2^{[2]}) \cos x - u_2^{[2]} \cos (3x) \ , 
\quad 
\cB_2 f_0^+ = (c_2 -2 u_2^{[0]} -2 u_2^{[2]} \cos (2x))f_0^+ \ .
\end{aligned}
\end{equation}
Using the decomposition \eqref{expBf} we get 
\begin{equation}\label{expBf1}
    \begin{split}
        (\cB_{0,\e} f_k^\sigma (0,\e),& f_{k'}^{\sigma'} (0,\e) ) = (\cB_0 f_{k} ^\sigma, f_{k'} ^{\sigma'}) + \e \left[(\cB_1 f_{k} ^\sigma, f_{k'} ^{\sigma'})+(\cB_0 f_{k_1} ^\sigma, f_{k'} ^{\sigma'})+(\cB_0 f_{k} ^\sigma, f_{k_1'} ^{\sigma'}) \right] +\\
        + \e^2 &\left[ (\cB_2 f_{k} ^\sigma, f_{k} ^{\sigma'}) 
        +(\cB_1 f_{k_1} ^\sigma, f_{k'} ^{\sigma'})
        +(\cB_1 f_{k} ^\sigma, f_{k_1'} ^{\sigma'})
        +(\cB_0 f_{k_1} ^\sigma, f_{k_1'} ^{\sigma'})
        +(\cB_0 f_{k_2} ^\sigma, f_{k'} ^{\sigma'}) 
        + (\cB_0 f_{k} ^\sigma, f_{k_2'} ^{\sigma'})\right]
    \end{split}
\end{equation}
We are now ready to compute the expansion of the elements $E_{11}(\e)$, $E_{33}(\e)$ and $E_{13}(\e)$.\\
\noindent$\bullet$ \textit{Expansion of $E_{1,1}(\e) =(\cB_{0,\e} f_1^+ (0,\e), f_{1}^{+} (0,\e) ) $.} By \eqref{actionexpansion} and the selfadjointness of $\cB_0$, the term of order zero  in \eqref{expBf1} does not contribute, together with the following others:
\begin{equation*}
 (\cB_0 f_{1}^+, f_{1_1}^+) = (\cB_0 f_{1_1}^+, f_{1}^+) =(\cB_0 f_{1}^+, f_{1_2}^+)=(\cB_0 f_{1_2}^+,f_{1}^+)=  (\cB_1 f_{1}^+, f_{1}^+) = 0
\end{equation*}
The nonvanishing terms, which we compute again using \eqref{actionexpansion}, are 
 $$
 (\cB_2 f_{1}^+, f_{1} ^{+}) = c_2 - 2u_2^{[0]} - u_2^{[2]} \, , \qquad
(\cB_1 f_{1_1}^+,f_{1}^+) = - \fa = (\cB_1 f_{1}^+, f_{1_1}^+)\, , \qquad
 (\cB_0 f_{1_1}^+,f_{1_1}^+)= (\fm(1) - \fm(2)) \fa^2 \ . $$
Summing up, and using the expressions of $u_2^{[0]}$ and $u_2^{[2]}$ in 
 Theorem \ref{existPTW} and of $\fa$ in  \eqref{fa}, we get 
\begin{equation}\label{E11}
    E_{1,1}(\e) = (\underbrace{c_2 - 2u_2^{[0]} - u_2^{[2]}}_{=0} -\underbrace{2 \fa+(\fm(1) - \fm(2))\fa^2}_{=- \fa})\e^2 + r (\e^3) = - \fa 
\e^2 + r(\e^3) \ ,
\end{equation}
where $r(\e^3)$ is of class $\cC^\reg$.\\
\noindent$\bullet$ \textit{Expansion of $E_{1,3}(\e)=(\cB_{0,\e} f_0^+ (0,\e), f_{1}^{+} (0,\e) )$.}
Again the term of order $0$ in \eqref{expBf1} vanishes. About the terms of order $\e$, by \eqref{actionexpansion} the only one not vanishing is
$(\cB_1 f_1^+, f_0^-) = - \sqrt{2}$.
Finally, all  terms of order $\e^2$ vanish, using also that 
$f_{1_2}^+$ has zero average by Lemma \ref{totalbasisexpansion}. 
 In conclusion 
\begin{equation}\label{E13}
    E_{1,3}(\e) = -\sqrt 2 \e + r(\e^3) \ 
\end{equation}
with $r(\e^3)$  of class $\cC^\reg$.\\
\noindent$\bullet$ \textit{Expansion of $E_{3,3}(\e) = (\cB_{0,\e} f_0^+ (0,\e), f_{0}^{+} (0,\e) )$: }
Recall that  $f_0^+(0,\e) = f_0^+$ for any $\epsilon$ by Lemma  \ref{totalbasisexpansion}.
Since $\cB_{0,\e} f_0 = (c_\e - \fm(0) - 2 u_\e) f_0^+$, by  Theorem \ref{existPTW} we readily get 
\begin{equation}\label{E33}
    E_{3,3}(\e) = \fm(1) - \fm(0) + \e^2 \underbrace{(c_2 - 2 u_2^{[0]})}_{\frac{1}{2}\fa} + r(\e^3)
\end{equation}
with $r(\e^3)$  of class $\cC^\reg$.\\
Then the expansion \eqref{B0e} follows from \eqref{tB0e}, \eqref{E220}, \eqref{E11}, \eqref{E13} and \eqref{E33}.
\end{proof}

{\bf \noindent Expansion of $\dot\tB_{0,\e}$.} 
Next we compute the expansion of the matrix $\dot \tB_{0,\e}$ in \eqref{exptBmue}. 
Remark that 
\begin{equation}\label{dB0e}
    \dot \tB_{0,\e} =( (\dot \cB_{0,\e}f_k^\sigma(0,\e), f_{k'}^{\sigma'}(0,\e)) + 2\textup{Sym}( \cB_{0,\e}f_k^\sigma(0,\e), \dot f_{k'}^{\sigma'}(0,\e))) = \dot \tB_{0,\e}^{[1]} + \dot \tB_{0,\e}^{[2]} + \dot \tB_{0,\e}^{[2]*}, 
\end{equation}
where $X^*$ denotes the transpose conjugate matrix to $X$.

%We shall use the following expansion of the vectors $\dot f_k^\sigma(\mu,\e) := \pa_\mu f_k^\sigma(\mu,\e)$ and
%$ \ddot f_k^\sigma(\mu,\e) := \pa^2_\mu f_k^\sigma(\mu,\e)$
%in \eqref{fksigma}:
%\red{
%\begin{lemma}\label{lem:dfks}
%Provided $\reg \geq 5$ ($3$ secondo antonio), 
%\end{lemma}
%}
%\begin{proof}
%By Lemma \ref{totalbasisexpansion}, 
%we need only to prove the estimates of the remainders. Recall that  $\cO^\reg(\mu^2 \e, \mu \e^2, \e^3)$ means a term which is the sum , call it $f$.  By Lemma \ref{lem:fg}, there are three functions $g_1, g_2, g_3$ in $\cC^{\reg-3}$ such that
%$f(\mu,\e) = \mu^2 \e g_1(\mu,\e) + \mu \e^2 g_2(\mu,\e) + \e^3 g_3(\mu,\e)$. 
%Then, using  $\reg \geq 4$, one gets  $\dot f(0,\e)  = \cO(\e^2)$.
%The estimate for $\ddot f_k^\sigma(0,\e)$ is achieved taking one further derivative in $\mu$, which requires  $\reg \geq 5$.\\
%Da antonio: We know that the functions $\dot f_0(\mu,\e)$ and 
%\end{proof}

We now prove:

\begin{lemma}[Expansion of the matrix $\dot \tB_{0,\e}$]\label{tB0exp2}
 The $3\times 3$ selfadjoint, purely imaginary and reversibility preserving matrix 
$\dot \tB_{0,\e}$ in \eqref{dB0e} 
 expands as
 \begin{equation}\label{dtB0e}
    \dot \tB_{0,\e} = \begin{pmatrix}
        0 &  \im (\dot \fm (1) +  r_2(\e^2)) & 0 \\
        - \im (\dot \fm (1) +  r_2(\e^2))& 0 & \im r_5(\e^2)\\
        0 & - \im r_5(\e^2) & 0
    \end{pmatrix} \ 
\end{equation} 
with $r_k$ are functions of class  $\cC^{\reg-1}$ independent of $\mu$.
 \end{lemma}
\begin{proof}
We consider separately the matrices  $\dot \tB_{0,\e}^{[1]}$ and $\dot \tB_{0,\e}^{[2]}$ in \eqref{dB0e}.

\noindent
\underline{ Expansion of  $\dot \tB_{0,\e}^{[1]}$}. Recall $\dot \tB_{0,\e}^{[1]} = (\dot \cB_{0,\e}f_k^\sigma(0,\e), f_{k'}^{\sigma'}(0,\e))_{k,k' = 0,1}^{\sigma,\sigma' = \pm}$. 
By Lemma  \ref{lem:LB.im}, 
 $\dot \cB_{0,\e}$ is reversibility preserving, so the entries of the matrix are  alternatively real and purely  imaginary. 
 But  $\dot \cB_{0,\e}$ is also purely imaginary, so every entry must also be  purely imaginary. 
As a result, the entries of the type $(\dot \cB_{0,\e}f_k^\sigma(0,\e), f_{k'}^{\sigma}(0,\e))$ are zero for any $|\e|<\e_0$:
\begin{equation}\label{dtB0e1}
    \dot \tB_{0,\e}^{[1]} = \begin{pmatrix}
        0 & \dot E_{12}^{[1]} (\e) & 0 \\
        - \dot E_{12}^{[1]}  (\e)& 0 & \dot E_{23}^{[1]} (\e) \\
        0 & -\dot E_{23}^{[1]} (\e) & 0
    \end{pmatrix} \ .
\end{equation}
We now compute the expansion of the remaining elements. Recall that 
 $\dot \cB_{0,\e} = -\dot \cM(D)$.\\
$\bullet$ {\em Expansion of $ \dot E_{12}^{[1]} (\e) =(\dot \cB_{0,\e}f_1^-(0,\e), f_{1}^{+}(0,\e))$}:
By  \eqref{basis.exp}, 
\begin{equation}\label{E121}
    \dot E_{12}^{[1]} (\e) = (\dot \cB_{0,\e}f_1^- (0,\e),f_1^+(0,\e))
= (\im\dot \fm (1) \cos x + \im \e \dot \fm (2) \fa  \cos (2x) + \cO^{\reg-1}
(\e^2),  f_1^+(0,\e))    
     = \im (\dot \fm (1) +  r(\e^2)) \, ,
\end{equation}
where $r(\e^2)$ is of class $\cC^{\reg-1}$.\\
$\bullet$ \textit{Expansion of $\dot E_{23}^{[1]} (\e)= (\dot \cB_{0,\e}f_0^+(0,\e), f_{1}^{-}(0,\e)) $:}  By Lemma \ref{totalbasisexpansion},   $\dot \cB_{0,\e}f_0^+(0,\e)  =
-\dot \cM(D) f_0^+ = 0$, 
 thus 
\begin{equation}
 \label{E231}
 \dot E_{23}^{[1]} (\e) = 0 \ . 
 \end{equation}

\noindent \underline{ Expansion of  $\dot \tB_{0,\e}^{[2]}$}. Recall
$ \dot \tB_{0,\e}^{[2]} =  \big( ( \cB_{0,\e}f_k^\sigma(0,\e), \dot f_{k'}^{\sigma'}(0,\e)) \big)$.
Since $\tF$ is a reversible basis, $\dot f_k^\sigma(0,\e)$ have the same parity properties in \eqref{parity}. 
But   $\dot f_k^\sigma (0,\e) = \dot U_{0,\e} f_k^\sigma$ is purely imaginary, as  $\dot U_{0,\e}$  is  purely imaginary by Lemma \ref{lemmanonzero} and $f_k^\sigma$ is real. 
Being  instead $f_k^\sigma(0,\e)$ real, we conclude  that 
\begin{equation}\label{dfk}
f_k^+(0,\e) = even(x) \ , \quad f_k^-(0,\e) = odd(x) \ , \quad 
    \dot f_k^+ (0,\e)  = \im \, odd(x)   \ ,   \qquad \dot f_k^- (0,\e) =   \im \, even(x) \ . 
\end{equation}
Now we use that $\cB_{0,\e}$ is parity preserving, as $\cB_{0,\e} =  c_\e - \cM (D) -2 u_\e$ and 
 $\cM(D)$ is an even Fourier multiplier (Assumption (A1))  and $u_\e$ of Theorem \ref{existPTW} is an even function. 
So the matrix $ \dot \tB_{0,\e}^{[2]} + \dot \tB_{0,\e}^{[2]*} $ has the form 
\begin{equation}\label{dB0e2}
    \dot \tB_{0,\e}^{[2]} + \dot \tB_{0,\e}^{[2]*} = 
    \begin{pmatrix}
        0 & \dot E_{12}^{[2]} (\e) & 0 \\
        - \dot E_{12}^{[2]}  (\e)& 0 & \dot E_{23}^{[2]} (\e) \\
        0 & -\dot E_{23}^{[2]} (\e) & 0        
    \end{pmatrix}
\end{equation}
with some coefficients that we now compute. \\
$\bullet$ \textit{Expansion of $ \dot E_{12}^{[2]} (\e) = ( \cB_{0,\e} f_1^- (0,\e), \dot f_1^+ (0,\e) )  +
( \cB_{0,\e} \dot f_1^- (0,\e),  f_1^+ (0,\e) )$:} 
First note that the action of $\cL_{0,\e}$ on the basis $\{ f_1^\pm(0,\e), f_0^+(0,\e)\}$ is given by the matrix $\tJ_0 \tB_{0,\e}$, with $\tB_{0,\e}$ in \eqref{B0e}.
Then one checks that $\cL_{0,\e} f_1^-(0,\e) = 0$, as it is clear by Lemma \ref{tB0exp1}. 
Using this identity together with the fact that 
 $\dot f_1^+(0,\e)$ is mean free (see  \eqref{dfk}), and \eqref{expBf}, \eqref{expBf2}, and the expression for $ \dot f_1^- (0,\e) $ in \eqref{basis.der},  we get 
\begin{equation}\label{E122}
    \dot E_{12}^{[2]} (\e) =  \underbrace{( \cB_{0,\e} f_1^- (0,\e), \dot f_1^+ (0,\e) )}_{ =  - (\cL_{0,\e} f_1^- (0,\e),  \cE_0\dot f_1^+ (0,\e) )  = 0 } + \overline{( \cB_{0,\e} f_1^+ (0,\e), \dot f_1^- (0,\e) )} = \im \e^2 \fb \underbrace{\left(\fa (\fm(1) - \fm(2))-1\right)}_{=0} + \im r(\e^3)  \, 
\end{equation}
where $r(\e^3)$ is of class $\cC^{\reg-1}$.\\
$\bullet$ \textit{Expansion of $ \dot E_{23}^{[2]} (\e)=
( \cB_{0,\e} f_0 (0,\e), \dot f_1^- (0,\e) )
+
( \cB_{0,\e} \dot f_0^+ (0,\e), f_1^- (0,\e) ) $:} As $f_0^+(0,\e) = f_0^+$ $\forall \e$ small enough, we have by \eqref{actionexpansion} that $\cB_{0,\e} f_0^+ (0,\e) = a + b \e f_1^+ + \cO(\e^2)$ for some constants $a,b$,
and 
 by 
 Lemma \ref{totalbasisexpansion} 
 that $\dot f_1^- (0,\e) = \im \e \fb \cos (2x)+\cO(\e^2)$ and  $\dot f_0(0,\e) = \cO(\e^2)$. So 
\begin{equation}\label{E232}
    \dot E_{23}^{[2]} (\e) = ( \cB_{0,\e} f_0 (0,\e), \dot f_1^- (0,\e) ) + \overline{( \cB_{0,\e} f_1^- (0,\e), \dot f_0 (0,\e) )} = \im r(\e^2)
\end{equation}
where $r(\e^2)$ is of class $\cC^{\reg-1}$.

\smallskip
In conclusion  the expansion \eqref{dtB0e} follows from 
\eqref{dtB0e1}, \eqref{E121}, \eqref{E231}, \eqref{dB0e2}, 
\eqref{E122}, \eqref{E232}. 
\end{proof}

\noindent {\bf  Expansion of $\ddot\tB_{0,\e}$.} 
Finally, we compute the expansion of the matrix $\ddot \tB_{0,\e}$ in \eqref{exptBmue}.  Such matrix is given by 
\begin{equation} \label{quadraticterms}
    \ddot \tB_{0,\e} = \left( (\ddot \cB_{0,\e} f_k^\sigma, f_{k'}^{\sigma'} )+ ( \cB_{0,\e} \ddot f_k^\sigma, f_{k'}^{\sigma'} ) + ( \cB_{0,\e} f_k^\sigma, \ddot f_{k'}^{\sigma'} ) + 2(\dot \cB_{0,\e} \dot f_k^\sigma, f_{k'}^{\sigma'} ) + 2(\dot \cB_{0,\e} f_k^\sigma, \dot f_{k'}^{\sigma'} ) +  2( \cB_{0,\e} \dot f_k^\sigma, \dot f_{k'}^{\sigma'} ) \right)\vert_{\mu = 0}
\end{equation} 
\begin{lemma}[Expansion of the matrix $\ddot\tB_{0,\e}$]\label{tBexp3}
The  $3\times 3$ selfadjoint, real   and reversibility preserving matrix  $\ddot \tB_{0,\e}$ in \eqref{quadraticterms}
 expands as
 \begin{equation}\label{B0e2}
\ddot \tB_{0,\e} =
  \begin{pmatrix}
        -\ddot \fm (1)  & 0 & 0\\
        0 & -\ddot \fm (1)  & 0 \\
        0 & 0 & - \ddot \fm(0) 
    \end{pmatrix} + \cO^{\reg -2}(\e)  \ 
 \end{equation}
 with $\cO^{\reg-2}(\e)$ independent of $\mu$.
\end{lemma}
\begin{proof}
By Lemma \ref{totalbasisexpansion}, $\dot f_k^\sigma(0,\e), \ddot f_k^\sigma(0,\e) = \cO(\e)$. Hence, 
  all the terms in  \eqref{quadraticterms} are  $\cO^{\reg -2}(\e)$, except $(\ddot \cB_{0,\e} f_k^\sigma, f_{k'}^{\sigma'} )$. Since $\ddot \cB_{0,\e} = - \ddot \cM(D)$, the result follows.
\end{proof}
In the following lemma we improve the order of the remainders of the off-diagonal terms. This will be necessary in order to block diagonalize the matrix $\tB_{\mu,\e}$.
\begin{lemma}[off-diagonal errors]
\label{lem:off}
    The off diagonal terms of the matrix $\tB_{\mu,\e}$ have the following form:
    \begin{equation}
    \label{B23}
            (\tB_{\mu,\e})_{1,3} = -\sqrt 2 \e \big( 1 +  r_1 (\e^2,\mu^2) \big) ,     \quad (\tB_{\mu,\e})_{2,3} = \im  \mu \e \,  r_2 (\mu,\e)
    \end{equation}
    with remainders $ r_k \in \cC^{\reg-2}$.
\end{lemma}

\begin{proof}
So far,  from \eqref{exptBmue}  using Lemmata \ref{tB0exp1}, \ref{tB0exp2} and \ref{tBexp3}, we have that
 \begin{equation}\label{B23_1}
            (\tB_{\mu,\e})_{1,3} = -\sqrt 2 \e + r_3(\e^3) +\mu^2 \varphi_1(\mu,\e)  \quad \text{and}\quad 
            (\tB_{\mu,\e})_{2,3} = \im  \mu \, r_5 (\e^2) + \im \mu^2 \varphi_2 (\mu,\e)
    \end{equation} 
    with $r_3 \in \cC^\reg$, $r_5 \in \cC^{\reg-1}$, 
    $\varphi_i (\mu,\e) \in \cC^{\reg -2}$, $i=1,2$.

Consider first $(\tB_{\mu,\e})_{1,3} \in \cC^\reg$. By 
     Lemma \ref{lem:fg} $(ii)$, write $r_3(\e^3)  = \e \, \tilde r_3(\e^2)$
    for  $\tilde r_3 \in \cC^{\reg-1}$.
   Next put
   \begin{equation}\label{h1}
   h_1(\mu,\e):=  \mu^2\varphi_1(\mu,\e) = (\tB_{\mu,\e})_{1,3} +\sqrt 2 \e - \e \, \tilde r_3(\e^2) \in \cC^\reg \ .   
   \end{equation}
We claim that 
\begin{equation}\label{h1_claim}
   |h_1(\mu,\e)| \leq C | \mu^2 \e| \ , \quad \forall (\mu,\e) \in B(\mu_0)\times B(\e_0) \ . 
   \end{equation}   
Then, from Lemma \ref{lem:fg} $(ii)$, it follows that
   $h_1(\mu,\e) = \e  \, r(\mu^2)$ for some  $r \in \cC^{\reg -1}$, proving the first of \eqref{B23}. \\
We prove now    \eqref{h1_claim}. First  we show that  $h_1(\mu, 0)\equiv 0$.
     Indeed,      
  $f_k^\sigma(\mu,0) = f_k^\sigma$ 
  (see Lemma  \ref{nopuremu}) and  therefore 
$ (\cB_{\mu,0}f_1^\pm (\mu,0), f_0^+(\mu,0)) = 0$ $\, \forall \mu$ (recall that $\cB_{\mu,0}$  a Fourier multiplier). 
It follows that 
\begin{equation}
\label{tBmu0}
(\tB_{\mu,0})_{i,3} = 0  \  \ \  \forall \mu \ , \quad  i =1,2 \ . 
\end{equation}
From  \eqref{h1} we then deduce   $h_1(\mu, 0) \equiv 0  $. 
Then, since  from \eqref{h1} 
$ \mu^2 \varphi_1(\mu,0) = h_1(\mu, 0) = 0$,  
it follows that also $\varphi_1(\mu,0) \equiv 0$ $\, \forall \mu$. Then by Lemma   \ref{lem:fg} 
  $\varphi_1(\mu,\e) =  \e \, \tilde \varphi_1(\mu,\e)$ for  some 
    $\tilde \varphi_1 \in \cC^{\reg -3}$, and in conclusion 
 $h_1(\mu,\e) = \mu^2 \e \, \tilde \varphi_1(\mu,\e)$. The local boundedness of $\tilde \varphi_1$ proves the claim \eqref{h1_claim}.

Next consider $  (\tB_{\mu,\e})_{2,3} $. 
By  \eqref{tBmu0}, \eqref{B23_1} and arguing as before, we have again that 
$\varphi_2(\mu,\e) = \e \, \tilde \varphi_2(\mu,\e)$ with 
$ \tilde \varphi_2 \in \cC^{\reg -3}$. 
Thus from \eqref{B23_1} we have that
$$
| \,  (\tB_{\mu,\e})_{2,3} |\, \leq C |\mu\, \e|\ ( |\e| + |\mu|) \ , \quad \forall (\mu,\e) \in B(\mu_0)\times B(\e_0) \ 
$$
and by Lemma \ref{lem:fg} $(ii)$ it follows 
$ (\tB_{\mu,\e})_{2,3}  = \mu \e \, r(\mu,\e)$ for some $r \in \cC^{\reg -2}$, as claimed.
\end{proof}

We are finally ready to prove Proposition \ref{matrixrepprop}.
\begin{proof}[Proof of Proposition \ref{matrixrepprop}]
The expansions in \eqref{E}, \eqref{tf} follow from 
\eqref{exptBmue} and 
Lemmata \ref{tB0exp1}, \ref{tB0exp2}, \ref{tBexp3} and \ref{lem:off}. \\
In particular, the entries of the matrix $\tB_{\mu,\e}$ have the expressions
\begin{equation*}
    \begin{split}
        (\tB_{\mu,\e})_{1,1}&= - \fa \e^2 + r^\reg_1(\e^3) + \frac12 \mu^2 (-\ddot \fm (1) + r^{\reg-2}_2(\e)) + \mu^2 \varphi_{1,1}(\mu,\e)  \\
        (\tB_{\mu,\e})_{1,2}&=\im (\dot \fm(1) \mu + \mu r^{\reg-1}_2(\e^2) + \mu^2 r^{\reg-2}_2(\e))    + \mu^2 \varphi_{1,2}(\mu,\e)\\
        (\tB_{\mu,\e})_{2,2}&= -\frac12 \ddot \fm(1) \mu^2 + \mu^2 r^{\reg-2}_4(\e) + \mu^2 \varphi_{2,2}(\mu,\e)\\
        (\tB_{\mu,\e})_{3,3}&= \fm(1) - \fm(0) + \frac12 \fa \e^2 - \frac12  \mu^2 (\ddot\fm(0) + r^{\reg-2}_6 (\e)) + \mu^2 \varphi_{3,3}(\mu,\e)
    \end{split}
\end{equation*}
where $r^\reg(\e^a) $ denotes a real valued function in $\cO^\reg(\e^a)$  and $ \varphi(\mu,\e)$ is the remainder matrix in equation \eqref{exptBmue}. 
Since the matrix-valued function $\varphi(\mu,\e)\in \cC^{\reg-2}$ in \eqref{exptBmue} fulfills $\varphi(0,0) = 0$, by Lemma \ref{lem:fg}-$(iii)$ $\varphi_{i,j}(\mu,\e) = r^{\reg-2}(\mu,\e)$. \\
Lemma \ref{lem:fg}-$(i)$ allows gathering two orders of $\e$ from $r^\reg_1(\e^3)$, by lowering its regularity to $\cC^{\reg-2}$, proving equation  \eqref{E}. Finally, the expansions in Lemma \ref{lem:off} complete the proof of Proposition \ref{matrixrepprop}.
\end{proof}

\section{Block-decoupling}\label{sec5}
The aim of this section is to block diagonalize the matrix $\tL_{\mu,\e}$ 
in \eqref{tLmue}. The first step is a singular scaling transformation, which is  not symplectic (according to  Definition \ref{def:symp}) and thus we compute how the Poisson tensor $\tJ_{\mu}$ is transformed. 
\begin{lemma}\label{Rescaling}
    The conjugation of the Hamiltonian and reversible matrix $\tL_{\mu,\e}$ with the reversibility preserving  matrix \footnote{From now on we  use the symbol $f^\dagger$ for the conjugate transpose of vector}
    \begin{equation}
        Y := \begin{pmatrix}
Q & \rvline & {\tt 0} \\
\hline 
{\tt 0}^\dag & \rvline &  \sqrt{\mu}
\end{pmatrix}  , \quad Q = \begin{pmatrix}
    \sqrt{\mu} & 0 \\
    0 & \frac{1}{\sqrt{\mu}} 
\end{pmatrix} , \quad \mu >0 
    \end{equation}
    yields the Hamiltonian and reversible matrix 
    \begin{equation}\label{tLmue1}
        \tL_{\mu,\e}^{(1)} := Y^{-1}\tL_{\mu,\e}Y = \mu \,  \tJ_\mu^{(1)}\tB_{\mu,\e}^{(1)}
    \end{equation}
    where $\tJ_\mu^{(1)}$ is the skew-adjoint and reversible matrix
    \begin{equation}\label{tJ1}
    \tJ_\mu^{(1)} := Y^{-1}\tJ_\mu Y^{-*} = \begin{pmatrix}
        \im & 1 & 0\\
        -1 & \im \mu^2 & 0 \\
        0 & 0 & \im
    \end{pmatrix}\,   \ , 
    \end{equation}
 $ \tB_{\mu,\e}^{(1)}$ is the selfadjoint and reversibility preserving matrix
   \begin{equation}\label{tB1}
    \tB_{\mu,\e}^{(1)} := Y^* \tB_{\mu,\e} Y = \begin{pmatrix}
        E^{(1)} & \rvline & \tf^{(1)} \\
        \hline
        \tf^{(1)\dag} & \rvline & g^{(1)}
    \end{pmatrix}
\end{equation}
 and  the $2\times2$ symmetric and reversibility preserving matrix $E^{(1)}$, the vector $\tf^{(1)}$ and the number $g^{(1)}$ expand  as 
    \begin{equation}\label{E1f1g1}
        \begin{split}
    &E^{(1)}=\begin{pmatrix}
        -\fa \e^2 (1+r_1'(\e)) + \te_{22} \mu^2(1+r_1''(\e,\mu)) 
         & \im  ( \te_{12}  + r_2(\e^2, \mu\e,\mu^2))\\
         -\im ( \te_{12}  + r_2(\e^2, \mu\e,\mu^2)) &
         (\te_{22} + r_4(\e,\mu))  
    \end{pmatrix}\, ,\\
    &\tf^{(1)}= \e\begin{pmatrix}
        -\sqrt 2 + r_3(\e^2,\mu^2)  \\
        \im r_5(\e,\mu)
    \end{pmatrix}\, ,
    \quad 
    g^{(1)} =  \te_{33}+ \frac{ \fa}{2} \e^2+\tg_{33}\mu^2 + r_6 (\e^3, \mu^2\e,\mu^3) \ . 
        \end{split}
    \end{equation}
      The reminders $r_k$ are functions of class  $\cC^{\reg-2}$.
\end{lemma}
\begin{proof} Compute 
$
        Y^* \tB_{\mu,\e} Y = \begin{pmatrix}
            QEQ & \rvline & \sqrt \mu Q \vec \tf\\
            \hline 
            (\sqrt \mu Q \vec \tf)^\dag & \rvline & \mu g
        \end{pmatrix}
  $
and collect a common factor $\mu$ from each entry. 
\end{proof}

\begin{remark}
1.  Before the conjugation, every entry of the matrix $\tB_{\mu,\e}$ is of class $\cC^\reg$ in $(\mu,\e)$. After the conjugation, the (2,2)-entry is divided by $\mu$, and thus it's regularity is decreased to $\cC^{\reg-1}$ at $\mu=0$.\\
2.  The  spectrum of $\tL_{\mu,\e}^{(1)}$ coincides with that of $\tL_{\mu,\e}$ whenever $\mu \neq 0$.
\end{remark}

\subsection{Non-perturbative step of block decoupling}\label{sectionnpbd}
  The next step is to block-diagonalize the matrix $\tL_{\mu,\e}^{(1)}$ in \eqref{tLmue1}. 
The main result of this section is the following.
\begin{lemma}\label{firstblockdec}
    There exists a vector $\ts (\mu,\e)$, with values in $\C^2$ and of class $\cC^{\reg-2}$ of the form 
 \begin{align}\label{ts}
\ts:=  \ts (\mu,\e) &=\frac{ \sqrt 2 \, \e}{\te_d^2} \begin{pmatrix}
           - \im \te_d + r_1(\e^2,\mu\e,\mu^2) )
                      \vspace{.3em}
                      \\
             \te_{b }  + r_2(\e,\mu)
        \end{pmatrix} 
 \end{align}    
    with  $\te_b, \te_d$,  defined in \eqref{tedet}, real  and non-zero by Assumption B, such that the following holds true. 
Conjugating the Hamiltonian and reversible matrix $\tL_{\mu,\e}^{(1)}$ in \eqref{tLmue1} with the symplectic and reversibility-preserving matrix 
    \begin{equation}\label{S.1}
    \exp{(S)},\qquad S:=\tJ_\mu^{(1)} \begin{pmatrix}
        \mathbf{0} & \rvline & \ts (\mu,\e)\\
        \hline 
        \ts (\mu,\e)^\dag & \rvline & 0
    \end{pmatrix} = \tJ_\mu^{(1)} X\ , 
    \end{equation}
    we obtain the Hamiltonian and reversible matrix 
    \begin{equation}\label{tLmue2}
        \tL_{\mu,\e}^{(2)} :=  \mu \exp{(S)}\, \tL^{(1)}_{\mu,\e}\, \exp{(-S)} = \mu\,  \tJ_\mu^{(1)}\tB_{\mu,\e}^{(2)}
    \end{equation}
    where $\tJ_\mu^{(1)} $   in \eqref{tJ1} and 
    \begin{equation}\label{tBmue2}
        \tB_{\mu,\e}^{(2)}:=\exp{(S)}^{-*}\,
        \tB_{\mu,\e}^{(1)}\, \exp{(S)}^{-1}= \begin{pmatrix}
            E^{(2)} & \rvline & \tf^{(2)} \\
            \hline
            \tf^{(2)\dag} & \rvline & g^{(2)}
        \end{pmatrix}
    \end{equation}
   with the $2\times2$ symmetric and reversibility preserving matrix $E^{(2)}$, the vector $\tf^{(2)}$ and the number $g^{(2)}$ expanding as 
    \begin{equation}\label{Efg2}
        \begin{split}
            E^{(2)} = &\begin{pmatrix}
                -\te_w \e^2 (1+r_1'(\e,\mu)) + \te_{22} \mu^2(1+r_1''(\e,\mu)) 
                 & \im ( \te_{12}  + r_2(\e^2, \mu\e,\mu^2))\\
                 -\im ( \te_{12}  + r_2(\e^2, \mu\e,\mu^2)) &
                 \te_{22} + r_4(\e,\mu)  
            \end{pmatrix}\, , \\
            & \tf^{(2)}=\e^3\begin{pmatrix}
                r_3(1)\\
                \im r_5(1)
            \end{pmatrix}\, , \qquad
            g^{(2)}= \te_{33} + r_6(\e^2,\mu^2) 
        \end{split}
    \end{equation}
and 
\begin{equation}\label{tewb}
\te_w:=\fa+\frac{2}{\te_d} \ , \quad \fa \mbox{ in } \eqref{fa} \ . 
\end{equation}
The functions $r_k$ are all   $\cC^{\reg-2}$ in $(\mu,\e)$.
\end{lemma}

The rest of the section is devoted to prove Lemma \ref{firstblockdec}.\\
We look for $\ts = \begin{pmatrix} \im s_1(\mu,\e) \\ s_2(\mu,\e) \end{pmatrix}$ and $s_j(\mu,\e)$ real valued, so that  the matrix $X$ is symmetric and reversible. 
By  \eqref{tJ1} and Lemma \ref{expontentiallemma},  the matrix   $\exp{(S)}$ is $\tJ^{(1)}$-symplectic and reversibility preserving for any $\mu\neq 0$. 
Next we compute the  Lie series  of the matrix  $\tL_{\mu,\e}^{(2)}$ in \eqref{tLmue2}.
 We first split $\tL_{\mu,\e}^{(1)}$ in the block-diagonal  and off diagonal parts:
\begin{equation}\label{D1R1}
\tL_{\mu,\e}^{(1)}= \mu \left(  D^{(1)} +  R^{(1)}\right) , \qquad  D^{(1)} := \begin{pmatrix}
        \Hat{\tJ}_\mu E^{(1)} & \rvline & {\tt 0} \\
        \hline
     {\tt 0}^\dag &\rvline & \im  g^{(1)}
    \end{pmatrix}, \qquad
 R^{(1)} := \begin{pmatrix}
       {\bf 0} & \rvline & \Hat{\tJ}_\mu  \tf^{(1)} \\
        \hline
        \im   \tf^{(1)\dag} &\rvline & 0
    \end{pmatrix} 
\end{equation}
where we adopted the notation
\begin{equation}\label{hJmu1}
    \tJ_\mu^{(1)} = \begin{pmatrix}
        \Hat{\tJ}_\mu & \rvline & \mathtt{0} \\
        \hline
        \mathtt{0}^\dag &\rvline & \im 
    \end{pmatrix},\qquad \Hat{\tJ}_\mu = \begin{pmatrix}
        \im & 1 \\
        -1 & \im \mu^2
    \end{pmatrix} \ . 
\end{equation}
The Lie expansion of  $\tL_{\mu,\e}^{(2)}$ is 
\begin{equation}
\label{lie.L2}
\begin{split}
    &\tL_{\mu,\e}^{(2)} = \mu \Big(  D^{(1)} + R^{(1)} + [S,  D^{(1)}] + [S, R^{(1)}]+ \frac12 [S,[S, D^{(1)}]]+ \\
    +\frac12 \int_0^1  (1-\tau)^2 &\exp{(\tau S)} \, \textup{ad}_S^3( D^{(1)})\, \exp{(-\tau S)}\de\tau + \int_0^1  (1-\tau) \exp{(\tau S)}\, \textup{ad}_S^2( R^{(1)})\, \exp{(-\tau S)}\de\tau \Big) \,.
    \end{split}
\end{equation}
We look for $S$ in order to solve the \textit{homological equation}
$ R^{(1)}+[S, D^{(1)}]=0$, which, by the  explicit expressions  \eqref{ts},  \eqref{D1R1}, reads 
\begin{equation}\label{homoleq}
    \begin{pmatrix}
      {\bf 0}  & \rvline & (\im g^{(1)} - \Hat{\tJ}_\mu  E^{(1)})\Hat{\tJ}_\mu \ts + \Hat{\tJ}_\mu  \tf^{(1)} \\
        \hline 
        \im \ts ^\dag (\Hat{\tJ}_\mu  E^{(1)} -\im g^{(1)} ) + \im  \tf^{(1)\dag}  & \rvline & 0
    \end{pmatrix} = 0   \ . 
\end{equation}
One readily checks that the  two equations are equivalent, 
%\begin{equation}\label{JE-ig}
%    \begin{split}
%        \im \ts ^\dag (\Hat{\tJ}_\mu \wt E^{(1)} -\im \wt g^{(1)} ) + \im \wt \tf^{(1)\dag} =0 \iff  \ts^\dag (\Hat{\tJ}_\mu \wt E^{(1)} -\im \wt g^{(1)} ) + \wt \tf^{(1)\dag} =0 \\
%        \iff (\im \wt g^{(1)} - \wt E^{(1)}\Hat{\tJ}_\mu )\ts +\wt\tf^{(1)} =0 \iff (\im\wt g^{(1)} - \Hat{\tJ}_\mu\wt E^{(1)})\Hat{\tJ}_\mu \ts +\Hat{\tJ}_\mu \wt\tf^{(1)} =0   \ , 
%    \end{split}
%\end{equation}
so it is enough to solve 
\begin{equation}\label{homeq1}
    ( \Hat{\tJ}_\mu E^{(1)}-\im g^{(1)} )\Hat{\tJ}_\mu \ts  =\Hat{\tJ}_\mu \tf^{(1)}  \ . 
\end{equation}
By Proposition \ref{Rescaling}  
the matrix $\Hat{\tJ}_\mu E^{(1)}-\im  g^{(1)}$ has the form 
\begin{equation}\label{inversedeterminingS}
    \Hat{\tJ}_\mu E^{(1)}-\im  g^{(1)} = \begin{pmatrix}
        -\im\te_d   + \im r_1 (\e^2, \mu\e,\mu^2)& \te_b   + r_2(\e,\mu) \\
         \fa\e^2(1+r_3'(\e))-  \te_b \mu^2(1+r_3''(\e,\mu))   & - \im \te_d  + \im r_4 (\e^2, \mu\e,\mu^2)
    \end{pmatrix}
\end{equation}
where we used that by \eqref{te}
\begin{equation}\label{te.id}
\te_{12} + \te_{33} \equiv \te_d  , \quad \te_{22}- \te_{12} \equiv   \te_b \ 
\end{equation}
with $\te_d$, $\te_b$ the numbers in \eqref{tedet}.
Its  determinant is
\begin{equation}
    \det (\Hat{\tJ}_\mu  E^{(1)}-\im g^{(1)})  = 
     - \left( \te_d^2+r(\e^2,\mu\e,\mu^2)\right) 
\end{equation}
which is not zero for sufficiently small $(\mu,\e)$ since $\te_d \neq 0$ by Assumption B.
A direct computation than gives the following result:
\begin{lemma}
    The vector $\ts(\mu,\e)$ in  \eqref{ts} solves  the homological equation \eqref{homeq1} and it  is $\cC^{\reg-2}$ .
    \end{lemma}
%
%\begin{proof} 
%By Assumption B, $\det (\Hat{\tJ}_\mu \wt E^{(1)}-\im \wt g^{(1)}) \neq 0$ at $\mu=0$ and thus $[ (\Hat{\tJ}_\mu \wt E^{(1)}-\im \wt g^{(1)})]^{-1}$ is a well defined matrix in a neighborhood of $\mu = 0$. It is moreover of class $\cC^{\reg-2}$ 
%    \begin{equation*}
%    \begin{split}
%        &(\Hat{\tJ}_\mu \wt E^{(1)}-\im \wt g^{(1)})^{-1}= 
%       -\frac{1}{ \te_d^2+r(\mu\e,\e^2,\mu^2)}\cdot \begin{pmatrix}
%        -\im\te_d   + \im r_1 (\e^2, \mu\e,\mu^2)& -\te_b   + r_2(\e,\mu) \\ 
%         -\fa\e^2(1+r_3'(\e,\mu))+  \te_b \mu^2(1+r_3''(\e,\mu))   & - \im \te_d  + \im r_4 (\e^2, \mu\e,\mu^2)
%    \end{pmatrix}
%    \end{split}
%    \end{equation*}
%By \eqref{hJmu1} and \eqref{E1f1g1}, the vector
%    \begin{equation*}
%     \Hat{\tJ}_\mu  \wt \tf^{(1)} =
%         \sqrt 2 \e\begin{pmatrix}
%            \im (-1 + \wt r_1(\e^2, \mu))\\
%            (1+ \wt r_2(\e^2,\mu^2))
%        \end{pmatrix} \ . 
%    \end{equation*}
%    is a well defined vector of class $\cC^{\reg-2}$.
%Hence  one gets  
%    \begin{equation*}
%        \Hat{\tJ}_\mu\ts = \sqrt 2 \e\begin{pmatrix}
%           (\te_d+\te_b) \, \te_d^{-2}  + r(\e,\mu) \\
% \im \, \te_d^{-1} + \im  r(\e^2,\mu\e,\mu^2)
%        \end{pmatrix}  \ . 
%    \end{equation*}
%Using \eqref{hJmu1} the vector $ \ts$ has the claimed form \eqref{ts}, and $\cC^2$ regularity.
%\end{proof}
\noindent
Since the matrix $S$ solves the homological equation $[S,D^{(1)}]+R^{(1)}=0$, the Lie expansion of $\tL^{(2)}_{\mu,\e}$ in \eqref{lie.L2} reduces to
\begin{equation}\label{L2.lie2}
    \tL^{(2)}_{\mu,\e}=\mu \Big(  D^{(1)} + \frac12 [S, R^{(1)}] + \frac12 \int_0^1  (1-\tau^2) \exp{(\tau S)}\, \textup{ad}_S^2( R^{(1)}) \, \exp{(-\tau S)}\de\tau \Big) \ . 
\end{equation}
In particular, the block-diagonal correction $\frac12[S,  R^{(1)}]$ is the Hamiltonian and reversible matrix 
\begin{equation}\label{tildeEg}
    \frac12 \tJ_\mu^{(1)} \begin{pmatrix}
        \im ( \ts  \tf^{(1)\dag}  -  \tf^{(1)}  \ts^\dag ) & \rvline & \mathtt{0}\\
        \hline 
        \mathtt{0} & \rvline &  \ts^\dag \hat \tJ_\mu  \tf^{(1)} - \tf^{(1)\dag}\hat \tJ_\mu  \ts 
    \end{pmatrix} =: \tJ_\mu^{(1)}\begin{pmatrix}
       \Delta E^{(1)} & \rvline & 0 \\
        \hline 
        0 & \rvline & \ \Delta g^{(1)}
    \end{pmatrix}
\end{equation}
where     the $2\times 2$ selfadjoint a reversibility preserving matrix $ \Delta E^{(1)}$ 
 and the real number $ \Delta g^{(1)}$ expand as 
    \begin{equation}
      \Delta E^{(1)} = \e^2 \begin{pmatrix}
            -2\te_d^{-1} (1+r_1(\e,\mu)) & \im \te_b\te_d^{-2}(1+r_2(\e,\mu))\vspace{0,15cm}\\ 
            -\im\te_b\te_d^{-2}(1+r_2(\e,\mu))& 0
        \end{pmatrix}\ , \quad \Delta g^{(1)} =  r_3(\e^2) \ . 
    \end{equation}
Notice that, whereas the other corrective terms are perturbative (i.e. they are of the same order of the remainders in $\tB_{\mu,\e}$), the $(1,1)$ entrance of the matrix $  \Delta E^{(1)} $ is of the same order $\e^2$ as that of $E^{(1)}$, 
and will give an essential contribute  to the Benjamin-Feir instability. 
In particular,  the block-diagonal matrix $D^{(1)}+\frac12 [S,R^{(1)}]$ is now  of the form
\begin{equation}
    \tJ_\mu^{(1)} \begin{pmatrix}
         E^{(1)} + \Delta E^{(1)}  & \rvline & 0\\
        \hline
        0 & \rvline & g^{(1)} + \Delta g^{(1)} 
    \end{pmatrix}=:\tJ_\mu^{(1)}\begin{pmatrix}
            E^{(2)} & \rvline & 0 \\
            \hline
            0 & \rvline & g^{(2)}
        \end{pmatrix}
\end{equation}
with $E^{(2)}$ and $g^{(2)}$ given in Lemma \ref{firstblockdec}. In particular note the new coefficient $\te_w$ in \eqref{tewb} at order $\e^2$ in the entrance (1,1) of $E^{(2)}$.\\
Finally, we show that the reminder of the Lie expansion in \eqref{L2.lie2} 
is small.
\begin{lemma}
    The $3\times 3$ Hamiltonian and reversible matrix 
    $$
    \frac12 \int_0^1  (1-\tau^2) \exp{(\tau S)} \, \textup{ad}_S^2(R^{(1)})\, \exp{(-\tau S)}\de\tau = \e^3 \tJ_\mu^{(1)} \begin{pmatrix}
        \widehat{E} & \rvline & \widehat \tf^{(2)} \\
        \hline
       \widehat  \tf^{(2)\dag} & \rvline & \widehat{g}
    \end{pmatrix}
    $$
    with 
    $\widehat{E}$, $\widehat{g}$ and 
    $\widehat \tf^{(2)}$ functions  $\cC^{\reg-2}$ of $(\mu,\e)$.
\end{lemma}
\begin{proof}
    Each entry of $S$  in \eqref{S.1}, \eqref{ts} and $R^{(1)}$ in \eqref{D1R1}, \eqref{E1f1g1} has the form  $\e \varphi(\mu,\e)$ with $\varphi \in \cC^{\reg-2}$. 
    The entries of  $\exp{(\tau S)}$ are $\cC^{\reg-2}$  functions of $(\mu,\e)$, and so $\textup{ad}_S^2(R^{(1)})$ has entries of the form $\e^3 \varphi(\mu,\e)$ with $\varphi \in \cC^{\reg-2}$.
\end{proof}
This concludes the proof of Lemma \ref{firstblockdec}.
\subsection{Complete block decoupling}
We now aim to completely block diagonalize the matrix $\tL_{\mu,\e}$. In order to do that, after the first step of non-perturbative block decoupling in the previous  section, we will perform a further $\tJ^{(1)}_\mu$-symplectic change of basis, that will completely remove the off-diagonal blocks of 
$\tL_{\mu,\e}^{(2)} $. This change of basis will be found solving a nonlinear equation in the entries of the matrix $\tB_{\mu,\e}^{(2)}$.
First we split
\begin{equation}\label{}
\begin{aligned}
\tL^{(2)}_{\mu,\e}  & = \mu \left( D^{(2)} + R^{(2)} \right) \ , \\
& D^{(2)} := \begin{pmatrix}
        \Hat{\tJ}_\mu E^{(2)} & \rvline & {\tt 0} \\
        \hline
     {\tt 0}^\dag &\rvline & \im  g^{(2)}
    \end{pmatrix}, \qquad
 R^{(2)} := \begin{pmatrix}
       {\bf 0} & \rvline & \Hat{\tJ}_\mu  \tf^{(2)} \\
        \hline
        \im  \tf^{(2)\dag} &\rvline & 0
    \end{pmatrix} 
\end{aligned}
\end{equation}
with   $\Hat{\tJ}_\mu$ in \eqref{hJmu1}, $E^{(2)}$, $g^{(2)}$, $\tf^{(2)}$ in \eqref{Efg2}.
\begin{lemma}\label{finallemma}
    There exist $\mu_0>0,\e_0>0$ such that for any $|\mu |<\mu_0$, $|\e|<\e_0$ there exist a reversibility preserving, Hamiltonian matrix $S^{(2)}=S^{(2)}(\mu,\e)$ and a Hamiltonian, reversible, block-diagonal matrix $P=P(\mu,\e)$, both of class $\cC^{\reg-2}$ in $(\mu,\e)\in B(\mu_0)\times B(\e_0) $, such that
    \begin{equation}
  \tL_{\mu,\e}^{(3)}:=       \exp(S^{(2)})\, \tL_{\mu,\e}^{(2)}\, \exp(-S^{(2)}) = \mu \big( D^{(2)}+ \e^6 P\big) \ . 
    \end{equation}
\end{lemma}
\begin{proof}
    We write for brevity $S^{(2)}=S$, and we look for  a Hamiltonian reversibility preserving matrix of the form 
    \begin{equation}\label{S2}
    S= \e^3 \tJ_\mu^{(1)}\begin{pmatrix}
        \mathbf{0} & \rvline & \ts\\
        \hline 
        \ts^\dag & \rvline & 0
    \end{pmatrix} \ , \quad \ts =   \begin{pmatrix} \im s_1(\mu,\e) \\ s_2(\mu,\e) \end{pmatrix} \ , \quad s_1, s_2 \in \cC^{\reg-2} \ .
\end{equation}

    Decomposing $\exp(S)\tL_{\mu,\e}^{(2)}\exp(-S)$ in the diagonal and off-diagonal blocks we solve for $S$ and $P$ the system 
    \begin{equation}\label{eqcases}
        \begin{cases}
            \Pi_D (\exp(S)\tL_{\mu,\e}^{(2)}\exp(-S))= \mu \big( D^{(2)} + \e^6 P\big) \\
            \Pi_\varnothing(\exp(S)\tL_{\mu,\e}^{(2)}\exp(-S))=0
        \end{cases} \ , 
    \end{equation}
    where we denoted $\Pi_D,\, \Pi_\varnothing$ the  projections respectively onto the diagonal and off diagonal matrices. \\
        Rewriting the second one in Lie expansion, the equation reads
        \begin{equation}\label{liesereq}
            R^{(2)}+[S,D^{(2)}] + \underbrace{\Pi_\varnothing(\int_0^1 (1-\tau)\exp(\tau S)\textup{ad}_S^2 ( D^{(2)} + R^{(2)}) \exp(-\tau S)\de\tau)}_{\cR(\ts,\mu,\e)}=0 \, .
        \end{equation}
By \eqref{Efg2},  
$$R^{(2)} = \e^3 \wt R^{(2)} = \e^3 \begin{pmatrix}
       {\bf 0} & \rvline & \Hat{\tJ}_\mu  \widetilde \tf^{(2)} \\
        \hline
        \im \wt  \tf^{(2)\dag} &\rvline & 0
    \end{pmatrix} 
$$ with $ \widetilde \tf^{(2)} $ of class $\cC^{\reg-2}$.
Writing also $S  = \e^3 \wt S$, \eqref{liesereq} reads
        \begin{equation}\label{liesereq2}
            \wt R^{(2)}+[\wt S,D^{(2)}] + \e^3 \underbrace{\Pi_\varnothing(\int_0^1 (1-\tau)\exp(\tau \e^3 \wt S)\textup{ad}_{\wt S}^2 ( D^{(2)} + \e^3 \wt R^{(2)}) \exp(-\tau \e^3 \wt S)\de\tau)}_{\cR(\ts,\mu,\e)}=0 \, ,
        \end{equation}
     where, due to the operator $\Pi_\varnothing$, $\cR(\ts,\mu,\e) $ has the form 
        \begin{equation*}
            \cR(\ts',\mu,\e) = \tJ_\mu^{(1)}\begin{pmatrix}
                0 & \rvline & \eta (\ts, \mu, \e)\\
                \hline
                \eta (\ts, \mu, \e)^\dag & \rvline & 0
            \end{pmatrix}
        \end{equation*}
        with $\eta (\ts, \mu, \e)$ of class $\cC^{\reg-2}$ in its variables, and at least quadratic in  $\ts$. 
        Hence, equation \eqref{liesereq2}  reduces to 
        \begin{equation}\label{secondblockdeceq}
            F(\ts,\mu,\e)=(\im g^{(2)} - \Hat{\tJ}_\mu E^{(2)})\Hat{\tJ}_\mu \ts  + \Hat{\tJ}_\mu \wt \tf^{(2)} + \Hat{\tJ}_\mu \eta (\ts, \mu, \e) = 0
        \end{equation}
        where $F\in \cC^{\reg-2}(B_r(0)\times B_{\mu_0}(0) \times B_{\e_0}(0))$.
       Now 
        $$\di_{\ts}F(0,0,0) = (\im g^{(2)} - \Hat{\tJ}_\mu E^{(2)})\Hat{\tJ}_\mu 
      \quad   \mbox{ and } \quad 
        \det(\im g^{(2)} - \Hat{\tJ}_\mu E^{(2)})\bigg|_{\e=\mu=0} = - \te_d^2 \neq 0 $$ 
        by Assumption B. 
        By the Implicit Function Theorem we find  a $\cC^{\reg-2}$ vector $\ts(\mu,\e)$ that solves equation \ref{secondblockdeceq}. 
 Finally, we prove that   $\ts$  has the claimed form \eqref{S2}. Indeed, a direct computation shows that 
 the matrix $\bar \rho \circ S \circ \bar \rho$ solves \eqref{liesereq}, and so by uniqueness of the solution $S  = \bar \rho \circ S \circ \bar \rho$, proving that $S$ is  reversibility preserving as  claimed  in  \eqref{S2}.
 
 Consider now the  first equation  in \eqref{eqcases}. Again Lie expanding we find 
        \begin{equation}
        \begin{aligned}
            \Pi_D  &(\exp(S)\tL_{\mu,\e}^{(2)}\exp(-S)) \\
            & = \mu
             \Big( 
             D^{(2)} +
              \e^6\underbrace{\big( [\wt S, \wt R^{(2)}]
              +
               \Pi_D
              \int_0^1 (1-\tau)\exp(\tau S)\, \textup{ad}_{\wt S}^2( D^{(2)} +R^{(2)}) \, \exp(-\tau S) \de \tau\big)}_{ =: P(\mu,\e)\in \cC^{\reg-2}}\big) \Big)
              \end{aligned}
        \end{equation}
        concluding the proof of the lemma.
\end{proof}

We are now ready to prove the main result:
\begin{proof}[Proof of Theorem  \ref{mainres2}]
By Lemma \ref{finallemma}, 
$       \tL_{\mu,\e}^{(3)} 
     = \mu \,  \tJ_{\mu}^{(1)}   \begin{pmatrix}  E^{(3)} & \rvline & 0 \\
    \hline
    0 & \rvline &  g^{(3)}
\end{pmatrix}
$
with  $ E^{(3)}$ and $ g^{(3)}$ expanding as   $E^{(2)}$ and $g^{(2)}$ in \eqref{Efg2}. 
Then the matrix $\tU:= \mu \hat \tJ_\mu  E^{(3)} $ expands as in   \eqref{tU} 
(recall $\te_{12} = \dot\fm(1)$ by \eqref{tedet} and $\te_{22}- \te_{12} = \te_b$ by  \eqref{te.id}) and the number $\tg:= \mu \im g^{(3)}$ as  in \eqref{tg.fin} (recall  $\te_{33} = \fm(1) - \fm(0)$ by \eqref{te}).
\end{proof}

\section{Applications}\label{appl}
In this section we apply Theorem \ref{mainres1} to several generalized KdV equations, listed in the table in Section \ref{sec1}

\subsection{Whitham's type equations}
We first focus on  Whitham's type equations. The Whitham equation is a famous nonlinear dispersive partial differential equation, modelling unidirectional, two-dimensional water waves at the surface of an irrotational, inviscid fluid. It is characterized by the same dispersion relation as the original water waves equation, and can be adapted to include effects such as surface tension or a constant vorticity, by modifying the dispersion relation. In particular, it has the form \eqref{eq}, where the operator $\cM(D)$ 
has  symbol 
\begin{align}
\label{m.whitham}
& \fm_\tth(\xi)=\sqrt{\frac{\tanh (\tth \xi)}{\xi}} \hspace{7em} \mbox{(classical Whitham)} \\
\label{m.c.whitham}
& \fm_{\tth, \kappa}(\xi)=\sqrt{(1+\kappa\xi^2)\frac{\tanh (\tth \xi)}{\xi}}
 \qquad \mbox{(gravity-capillary Whitham)} \\
 \label{m.v.whitham}
 &\fm_{\tth, \gamma} (\xi)= \frac{\gamma}{2} \frac{\tanh(\tth \xi)}{\xi} 
 + 
 \sqrt{\frac{\tanh (\tth \xi)}{\xi} 
 + \frac{\gamma^2}{4} \frac{\tanh^2 (\tth \xi)}{\xi^2} 
 } \qquad \mbox{(vorticity-modified Whitham)}
\end{align}
where $\tth >0$ is a parameter representing the depth of the fluid, $\kappa>0$ the surface tension coefficients and $\gamma\in \R$ a constant vorticity.

Modulational instability for the Whitham equations in \eqref{m.whitham}--\eqref{m.v.whitham} has been studied in \cite{hur2014modulational, hur2015modulational}, where it has been proved that, for certain values of the parameters $\tth, \kappa, \gamma$, the linearized Whitham equations at a small amplitude traveling wave solutions have unstable eigenvalues near the origin.

Applying our abstract Theorem \ref{mainres1} we now show that the spectrum of the linearized operator actually depicts a closed "figure 8".

\paragraph{Whitham equation.} We first consider the  classical Witham equation with symbol $\fm_\tth(\xi)$ in \eqref{m.whitham}. 
For any $\tth >0$, the function $\xi \mapsto \fm_\tth(\xi)$ is real valued, even, smooth, it fulfills \eqref{orderpsedodiff} with $m = - \frac12$ and it is strictly decreasing on $\R_{>0}$. 
Hence,  condition  \eqref{m1diff} is verified for all $\tth >0$ and so is Assumption A.
Then,  by Theorem \ref{existPTW}, there exists $\e_0 = \e_0(\tth)>0$ and for any $\e \in (0, \e_0)$ a traveling wave
solution    $u_\e$ of  \eqref{eq}. 
Denote by  $\cL_\e$ the linearized operator \eqref{def:cLe} at  $u_\e$. We obtain:
\begin{theorem}[{\bf Whitham}]
Consider equation \eqref{eq} with $\cM(D)$ having symbol  \eqref{m.whitham}.
  There exists a critical threshold  $\tth_{\scaleto{\mathrm{WB}}{3pt}} \sim 1.146\dots$ such that:
\begin{itemize}
\item if $\tth > \tth_{\scaleto{\mathrm{WB}}{3pt}} $,  there exists $\e_1(\tth)>0$ such that for any $\e \in [0,\e_1(\tth))$, the spectrum of the operator $\cL_\e$ depicts near the origin a complete  figure ``8'' as in Figure \ref{figure8image}.
\item If $\tth <\tth_{\scaleto{\mathrm{WB}}{3pt}} $, the spectrum of $\cL_\e$ near the origin is purely imaginary.
\end{itemize}
\end{theorem}
\begin{proof}
We have already verified Assumption A, and   now we check Assumption B.
The coefficient $\te_{12}  = \dot \fm_\tth(1)<0 $ since $\xi \mapsto \fm_\tth(\xi)$ is strictly decreasing for $\xi \geq 0$. Also, 
\begin{equation}
 \label{ed.w}
 \te_d = \dot \fm_\tth(1)  + \fm_\tth(1) - \fm_\tth(0) <0 , \quad 
\fm(1)-\fm(2)>0
 \end{equation} for the same reason. 
 The coefficient $\te_b$ is
 \begin{equation}\label{tebwhit}
\te_{b} = \frac{\big( \tth(1-\tc_\tth^4) - \tc_\tth^2 \big)^2 + 4 \tth^2 \tc_\tth^4 (1- \tc_\tth^4) }{8 \tc_\tth^3} >0 
 \ , \qquad \tc_\tth:= \sqrt{\tanh(\tth)} \ . 
\end{equation}
Finally, the coefficient
$$
\te_w = \frac{\left(\tth(1-\ch^4)+5\ch^2 -2\ch(\sqrt{\tth}+2\ch\sqrt{\frac{1}{1+\ch^4}}) \right)}{2\ch(\tth(1-\ch^4)-2\sqrt{\tth}\ch-3\ch^2)(1-\sqrt{\frac{1}{1+\ch^4}})}\ 
$$
is plotted in Figure \ref{tewwitham} as a function of $\tth$. Numerically one sees that it is a  strictly increasing function of $\tth$. 
Moreover,  $\lim_{\tth \to 0+}\te_w = -\infty$ and $\lim_{\tth \to \infty} \te_w = \frac{\sqrt{2}}{\sqrt{2} -1}$. It has only one zero at the critical depth $\tth_{\scaleto{\mathrm{WB}}{3pt}} \approx 1.146...$

\noindent\underline{Sign of $\te_{\scaleto{\mathrm{WB}}{3pt}}$:} 
The function $\te_{\scaleto{\mathrm{WB}}{3pt}}=\te_w\te_b$
\begin{comment}
$$
\te_{\scaleto{\mathrm{WB}}{3pt}} =
\frac{\big( \tth(1-\tc_\tth^4) - \tc_\tth^2 \big)^2 + 4 \tth^2 \tc_\tth^4 (1- \tc_\tth^4) }{8 \tc_\tth^3} 
\cdot 
 \frac{\left(\tth(1-\ch)+5\ch^2 -2\ch(\sqrt{\tth}+2\ch\sqrt{\frac{1}{1+\ch^4}}) \right)}{2\ch(\tth(1-\ch^4)-2\sqrt{\tth}\ch-3\ch^2)(1-\sqrt{\frac{1}{1+\ch^4}})}\ 
$$\
\end{comment}
is plotted in  Figure \ref{tewwitham}. Since $\te_b$ is strictly positive by \eqref{tebwhit}, the function $\te_{\scaleto{\mathrm{WB}}{3pt}} $ has a zero at
$\tth_{\scaleto{\mathrm{WB}}{3pt}} \approx 1.146...$.
Therefore, whenever $\tth > \tth_{\scaleto{\mathrm{WB}}{3pt}}$, $\te_{\scaleto{\mathrm{WB}}{3pt}}>0$ and item $(i)$ of Theorem \ref{mainres1} applies. 
Otherwise, item $(ii)$ applies.

\end{proof}
\begin{figure}[H]
    \centering
    \includegraphics[width=0.42\linewidth]{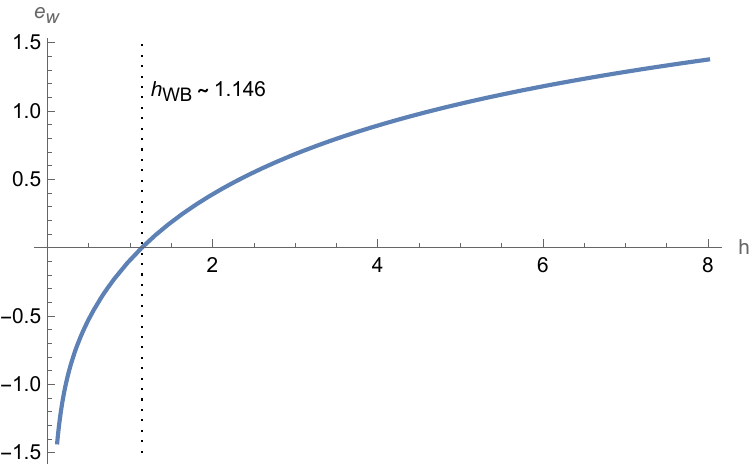}
    \includegraphics[width=0.42\linewidth]{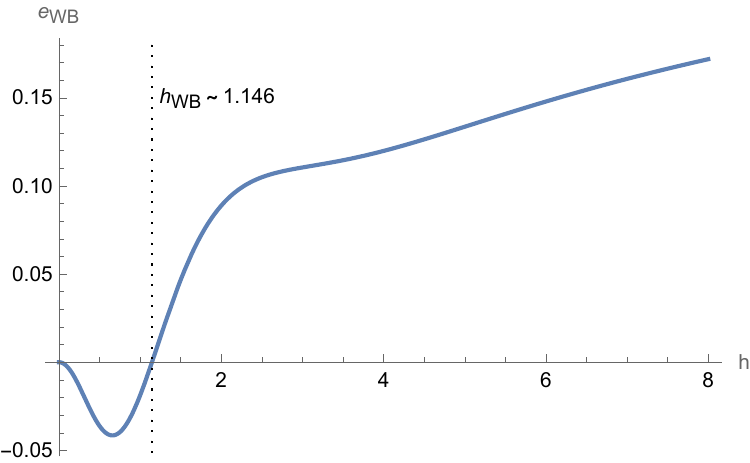}

    \caption{
    In the left figure is represented the dependence of $\te_w$ on $\tth$. We can see that it is in accordance with the one in \cite{hur2014modulational}.
    In the right figure is represented the dependence of $\te_{\scaleto{\mathrm{WB}}{3pt}}$ on $\tth$.}
    \label{tewwitham}
    \end{figure}

\textsc{Relation with Hur-Johnson \cite{hur2014modulational}.} 
Hur and Johnson proves the existence of unstable eigenvalues of the  $\cL_{\mu,\e}$  near zero provided the  \textit{stability-index function}
$$
\Gamma(\tth):=3\alpha (\tth) -2\alpha (2\tth) -1 -\tth\alpha' (\tth)\, , \qquad \alpha (\tth) := \sqrt{\frac{\tanh (\tth)}{\tth}}
$$
is negative. 
A direct computation shows that 
$$
\sqrt{\tth}\, \Gamma(\tth) = \frac{\te_{\scaleto{\mathrm{WB}}{3pt}}}{\te_b} \, (\fm_\tth(1) - \fm_\tth(2) ) \, \te_d \ , 
$$
and so, in view of \eqref{ed.w} and \eqref{tebwhit}, $\Gamma(\tth) < 0$ if and only if $\te_{\scaleto{\mathrm{WB}}{3pt}} >0$.
So our work shows that whenever the stability-index function is negative, actually the spectrum of the operator $\cL_\e$ depicts a complete figure 8 around the origin. 
%
%This work is in complete accordance with Hur's and Johnson's: we relate the occurence of the instability to the change of sign of $\te_{\scaleto{\mathrm{WB}}{3pt}}$. 
%As one can notice from equation \ref{formulaewb}, the factor $3\fm(1)+\dot \fm(1) -2\fm(2) - \sqrt{\tth}$ of $\te_{\scaleto{\mathrm{WB}}{3pt}}$ is proportional to $\Gamma$, as functions of $\tth$, through
%$$
%3\fm(1)(\tth)+\dot \fm(1)(\tth) -2\fm(2)(\tth) - \sqrt{\tth} = \sqrt{\tth} (3\alpha (\tth) -2\alpha (2\tth) -1 -\tth\alpha' (\tth))
%$$
%whereas the other factors in $\te_{\scaleto{\mathrm{WB}}{3pt}}$ have constant sign. Thus $\te_{\scaleto{\mathrm{WB}}{3pt}}>0 $ whenever $\Gamma(\tth)<0$.
%
%\begin{comment}
%\begin{equation}\label{multiplier}
%\dot \cM(D) := \frac12 \frac{ \tth D (1 - \tanh^2 (\tth |D|)) - \tanh (\tth D) }{|D|^{3/2} \sqrt{\tanh(\tth D)}} \ . 
%\end{equation}
%\end{comment}

\paragraph{Gravity-capillary Whitham equation.}
We now consider the Whitham equation with gravity and surface tension, namely \eqref{eq} with symbol $\fm_{\tth, \kappa}(\xi)$ in \eqref{m.c.whitham}.
Let us focus first on Assumption A.  The symbol $\fm_{\tth, \kappa}(\xi)$ is even, smooth and of order $\frac12$. 
Regarding Assumption (A3), it fails whenever $(\tth, \kappa)$  provoke a resonance, namely an integer  $n\in\N_0\setminus\{1\}$ such that $\fm_{\tth, \kappa}(1)=\fm_{\tth, \kappa}(n)$. 
These conditions define on the plane $(\kappa,\tth)$ a family of curves 
\begin{equation}
    \kappa_n(\tth) = \begin{cases}
        \dfrac{\tth}{\tanh (\tth)}-1 \, \qquad &\mbox{ if } n=0 \\
        \dfrac{n\tanh (\tth) - \tanh (n\tth)}{n^2\tanh(n\tth)-n\tanh (\tth)} \, \qquad &\mbox{ if } n\geq 2
    \end{cases}
\end{equation}
that are represented in dotted lines in Figure \ref{signwithtens} below, and move from top to bottom  as $n$ increases. In particular, curve (c) in Figure \ref{signwithtens} is due to the resonance between the first and second mode, that actively produce a change in the stability of the periodic traveling wave.
Denote by $\cR$ the resonant set 
\begin{equation}\label{}
\cR:= \{ (\tth, \kappa) \in \R_{>0} \times \R_{>0} \colon \kappa = \kappa_n(\tth) \ \ \mbox{ for some } n \in \N_0 \setminus \{1\} \} \ .
\end{equation}
In case $(\tth, \kappa) \not\in \cR$, \eqref{m1diff} is verified. Indeed, $\fm_{\tth,\kappa}(\xi)\mapsto +\infty$ for $\xi\mapsto+\infty$, and thus $\exists \ \bar n>0$ such that $\forall n >\bar n$, $|\fm_{\tth,\kappa}(n)-\fm_{\tth,\kappa}(1)|>1$. Therefore, $|\fm_{\tth,\kappa}(n)-\fm_{\tth,\kappa}(1)|>\min \{1, \min\{|\fm_{\tth,\kappa}(n)-\fm_{\tth,\kappa}(1)| \ : \ n\leq \bar n  \} \}$.
Then,  by Theorem \ref{existPTW}, for any $(\tth, \kappa) \not \in \cR$ there exists $\e_0 = \e_0(\tth,\kappa)>0$ and for any $\e \in (0, \e_0)$ a traveling wave
solution    $u_\e$ of  \eqref{eq}. 
Denote by  $\cL_\e$ the linearized operator \eqref{def:cLe} at  $u_\e$. We obtain:
\begin{theorem}[{\bf Gravity-capillary Whitham}]
Consider equation \eqref{eq} with $\cM(D)$ having symbol  \eqref{m.c.whitham}.
  There exist two open regions $\cU$ and $\cS$, delimited by analytic curves in $\R_{>0} \times \R_{>0}$, such that:
  \begin{itemize}
\item for any $(\tth, \kappa) \in \cU\setminus \cR$  there exists $\e_1(\tth, \kappa)>0$ such that for any $\e \in [0,\e_1(\tth,\kappa))$, the spectrum of the operator $\cL_\e$ depicts near the origin a complete  figure ``8'' as in Figure \ref{figure8image}.
\item for any $(\tth, \kappa) \in \cS\setminus \cR$  the spectrum of $\cL_\e$ near the origin is purely imaginary.
\end{itemize}
\end{theorem}
\begin{proof}
Assumption A has already been verified. 
The coefficient $\te_{12}$ is given by
\begin{equation}
\te_{12} = \sqrt{1+\kappa} \, \ch \ . 
\end{equation}
Its zero set gives a curve $\gamma_{12}$ plotted in purple  in Figure \ref{signwithtens}.
The function 
\begin{equation}
\label{}
\te_d = \frac{\left(\tth(1+\kappa)(1-\ch^4)  - (3+\kappa)\ch^2 -2\sqrt{\tth(1+\kappa)}\ch \right)}{2\ch\sqrt{1+\kappa}}
\end{equation}
has zero set given by curve  (d) in Figure \ref{signwithtens},
the function 
\begin{equation}
\label{}
\begin{split}
\te_b = \frac{1-3\kappa(2+\kappa)}{(2\sqrt{1+\kappa})^3}\ch+\frac12\tth^2(1+\kappa)^{\frac12}(1-\ch^4)\ch^{-3}-\frac14 \tth(1+k)^{-\frac12}(1+3k)(1-\ch^4)
\end{split}
\end{equation}
has zero set given by curve (b) in Figure \ref{signwithtens}
and finally the function
\begin{equation}
\label{}
\te_w =\frac{1}{\te_d(\fm_{\tth,\kappa}(1)-\fm_{\tth,\kappa}(2))}\left( \frac{\left(\tth(1+\kappa)(1-\ch^4)+(5+7\kappa)\ch \right)}{2\sqrt{1+\kappa}\ch}-\sqrt \tth + \sqrt{1+4\kappa}\frac{2\ch}{\sqrt{1+\ch^4}}\right)
\end{equation}
has zero set given by two connected components (a) and (e).

Finally, the function 
$$
\te_{\scaleto{\mathrm{WB}}{3pt}}=\te_w\te_b
$$
Then put $\cU:= \{\te_{\scaleto{\mathrm{WB}}{3pt}} >0 \} \setminus \gamma_{12}$, and $\cS:=  \{\te_{\scaleto{\mathrm{WB}}{3pt}} <0 \} \setminus \gamma_{12}$.
%Assumption B is instead no longer always verified. In particular, it fails when some factor of $\te_{\scaleto{\mathrm{WB}}{3pt}}$ vanish. In particular, this last occurance defines the curves delimitating the two regions $\cU$ and $\cS$, represented in \ref{stfigure}. The implicit function theorem then ensures the analyticity of these curves delimitating the stable and unstable regions. Finally, Assumption B fails also when $\te_{12}=\dot\fm(1)=0$, where the thesis of Theorem \ref{mainres2}  is still true, but the curves of eigenvalues in \eqref{figure8param} does no longer parametrize a figure 8. $\te_{12}=0$ defines the curve in the plane $(\tth,\kappa)$ rapresented in figure \ref{signwithtens} in purple.
%\\
%
\end{proof}

\begin{figure}[H]
    \centering
    \includegraphics[width=0.42\linewidth]{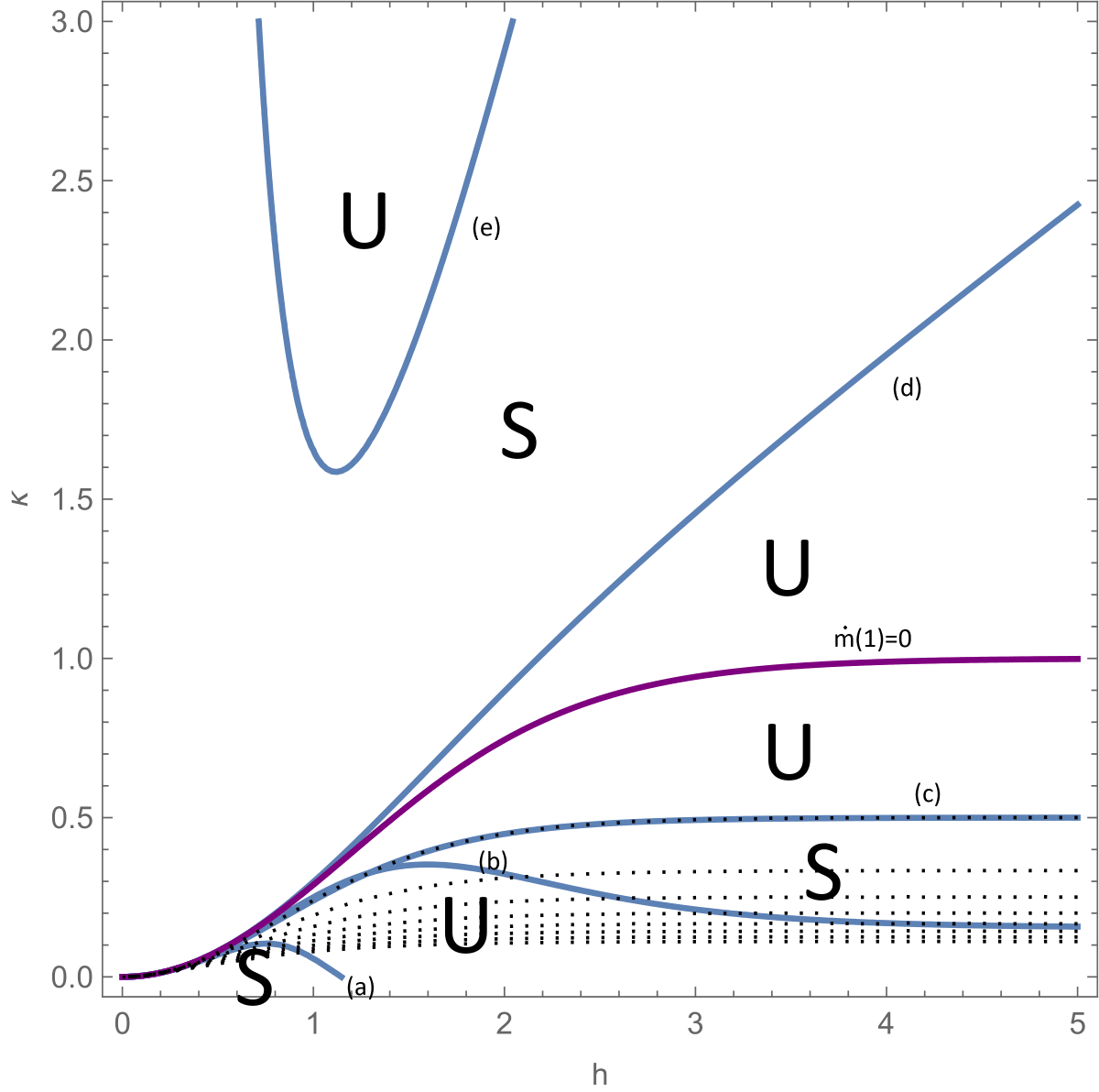}
    \hspace{1cm}
    \caption{In the figure is represented the sign of the function $\te_{\scaleto{\mathrm{WB}}{3pt}}$ dependent on the parameters $\tth$ (on the horizontal axis) and $\kappa$ (on the vertical axis). The areas denoted by the letter U are the parts of the plane where $\te_{\scaleto{\mathrm{WB}}{3pt}}$ is positive, the areas denoted by the letter S are the ones where it is negative. 
    Moreover, in the figure are reported the curves where Assumption A fails as dotted lines, namely there exists $n\in \N \setminus \{ 1\}$ such that $\fm(1)=\fm(n)$. In addition, in correspondence of the curves in blue and the curve in purple, Assumption B fails. 
    In particular, the curve represented in blue are the ones where the sign of $\te_{\scaleto{\mathrm{WB}}{3pt}}$ changes:
    (a) and (e) are given by the condition $\te_w= 0$; 
    (b) is given by $\te_b =0$; 
    (c) is given by $\fm(1)-\fm(2)=0$ and it is a resonance contributing to the sign of $\te_{\scaleto{\mathrm{WB}}{3pt}}$, and finally (d) is given by 
    $\te_d=0$. The curve in purple is where $\te_{12}=0$.
    }
    \label{signwithtens}
    \end{figure}

    \textsc{Relation with Hur, Johnson \cite{hur2015modulational}.}
In \cite{hur2015modulational} the authors prove that whenever the  index function 
    $$
    \Gamma(\tth,\kappa):= \frac{(z m(z))'' \, \big((z m(z))' - m(0) \big)}{m(z) - m(2z)}  \cdot \big(2(m(z) - m(2z)) + (zm(z))' - m(0) \big)\big |_{z=\tth} \ , \quad m(z) := \fm_{\tth, \kappa}(z/\tth)
$$    
is negative, the operator $\cL_{\mu,\e}$ has two unstable eigenvalues near zero.
Again a direct computation using \eqref{eWB} shows that 
$$
\tth\Gamma(\tth,\kappa) = -2 \te_{\scaleto{\mathrm{WB}}{3pt}} \, \te_d^2 \ ,
$$
so $\Gamma(\tth,\kappa) < 0$ iff $\te_{\scaleto{\mathrm{WB}}{3pt}} >0$.
%    Therefore it 
%    y extend the previous result to the case where the surface tension and capillarity are present. In this work, we group the phisically relevant parameters into the two free parameters $(\tth,\kappa)$.
As a consequence,  the stability-instability regions that we find in Figure \ref{signwithtens} are in one to one correspondence with those of  \cite{hur2015modulational}, see e.g. Figure 2.

    \paragraph{Vorticity-modified Whitham equation.}
    We now consider the Whitham equation with constant vorticity $\gamma$, namely \eqref{eq} with symbol $\fm_{\tth, \gamma}(\xi)$ in \eqref{m.v.whitham}. 
    Let us focus first on Assumption A. Similarly to the original Whitham equation, the Fourier multiplier $\cM(D)$  has order $-\frac12$, and the symbol $\fm_{\tth, \gamma}(\xi)$ is smooth $\forall (\tth,\gamma)\in \R_{>0}\times \R$.
    Moreover, $\fm_{\tth, \gamma}(\xi)$ is strictly decreasing in $\xi>0$, $\forall \ (\tth,\gamma)\in\R_{>0}\times\R$: indeed, $\xi\in (0,+\infty)\mapsto \frac{\tanh (\tth\xi)}{\xi}$ is strictly decreasing $\forall\tth\in\R_{>0}$, whereas $x\in(0,+\infty)\mapsto \frac{\gamma}{2} x + \sqrt{x + \frac{\gamma^2}{4}x^2}$ is strictly increasing $\forall\gamma\in\R$. Thus, no resonance occur, verifying Assumption A. 
    Then,  by Theorem \ref{existPTW}, there exists $\e_0 = \e_0(\tth)>0$ and a traveling wave
    solution    $u_\e$ of  \eqref{eq}, parametrized by any $\e \in (0, \e_0)$ . 
    Denote by  $\cL_\e$ the linearized operator \eqref{def:cLe} at  $u_\e$. We obtain:
    \begin{theorem}[{\bf Vorticity-modified Whitham}]
    Let $\cM(D)$ having symbol  \eqref{m.v.whitham}.
      There exists an analytic curve in the semiplane $(\tth,\gamma)$, defined by equation 
      $$
      \dot\fm_{\tth, \gamma}(1)+3\fm_{\tth, \gamma}(1)-2\fm_{\tth, \gamma}(2)-\fm_{\tth, \gamma}(0)=0
      $$
       separating the regions $\cU$ and $\cS$, such that:
    \begin{itemize}
    \item if $(\tth,\gamma)\in\cU$,  there exists $\e_1(\tth)>0$ such that, for any $\e \in [0,\e_1(\tth))$, the spectrum of the operator $\cL_\e$ depicts near the origin a complete  figure ``8'' as in Figure \ref{figure8image}.
    \item If $(\tth,\gamma)\in\cS$, the spectrum of $\cL_\e$ near the origin is purely imaginary.
    \end{itemize}
    \end{theorem}
    \begin{proof}
        We only need to show that, outside the critical curve, Assumption B is verified. In particular, by the monotonicity of $\fm_{\gamma,\tth}$ it is clear that $$\dot\fm_{\gamma,\tth}(1)<0,\quad \fm_{\gamma,\tth}(1)-\fm_{\gamma,\tth}(2)>0, \quad \te_d = \dot \fm_{\gamma,\tth}(1)+\fm_{\gamma,\tth}(1)-\fm_{\gamma,\tth}(0)<0.$$ 
        Moreover, a direct computation shows that 
        \begin{equation*}
            \begin{split}
        \te_b= \frac{\ch}{2}\left(2+8\tth^2\left(\frac{1-\ch^4}{2\ch^2}\right)^2-4\tth\frac{1-\ch^4}{\ch} +\tth^2(1-\ch^4)(4+\gamma^2\ch^2)(2+\gamma^2\ch^2+\gamma\ch\sqrt{4+\gamma^2\ch^2})\right)\\/(4+\gamma^2\ch^2)^{\frac32} \ .
    \end{split}
        \end{equation*}
        A numerical evaluation show that $\te_b>0\quad \forall (\tth,\gamma)\in\R_{>0}\times\R$, and thus Assumption B is verified outside the curve defined by $\te_w=0$.
    \end{proof}
    \begin{figure}[H]
        \centering
        \includegraphics[width=0.42\linewidth]{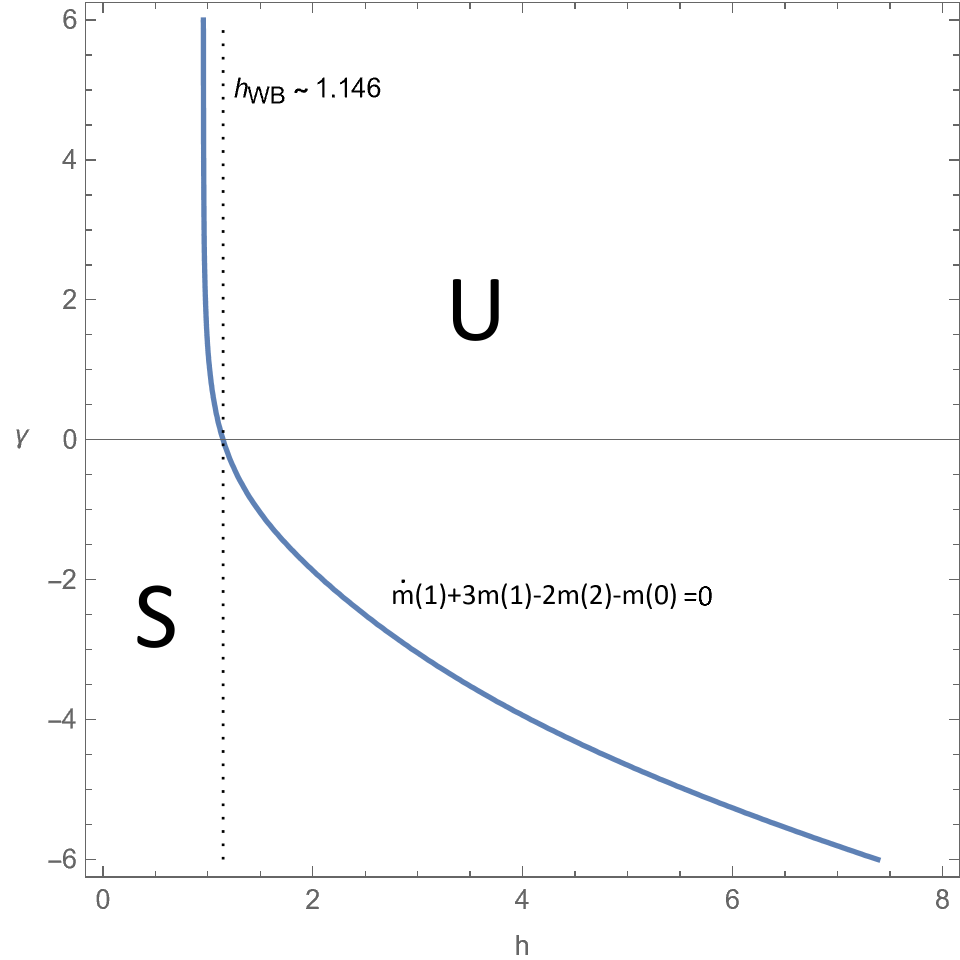}
        \hspace{1cm}
        \caption{
        In the figure is represented the sign of the function $\te_{\scaleto{\mathrm{WB}}{3pt}}$ dependent on the parameters $\tth$ (on the horizontal axis) and $\gamma$ (on the vertical axis). The areas denoted by the letter U are the parts of the plane where $\te_{\scaleto{\mathrm{WB}}{3pt}}$ is positive, and thus where the equation is unstable; the areas denoted by the letter S are the ones where it is negative, and thus the equation results to be stable. 
        The curve in blue represent the critical line where $\dot\fm_{\tth, \gamma}(1)+3\fm_{\tth, \gamma}(1)-2\fm_{\tth, \gamma}(2)-\fm_{\tth, \gamma}(0)=0$, sepearating the two regimes. }
        \label{signwithvort}
        \end{figure}
\subsection{Kawahara equation}
In this section we consider the  Kawahara equation, namely equation  \ref{eq}
with Fourier multiplier $\cM(D)$ having as symbol
\begin{equation}
\label{kawa}
\fm_{\ta, \tb}(\xi) =\ta  \xi^2 + \tb \xi^4, \qquad \ta \in \R, \ \ \tb>0  \ . 
\end{equation}
Kawahara equation is a model for nonlinear dispersive waves, that emerges as a correction to the well known KdV equation in the regimes where the fifth order term of the dispersion relation of the water waves equation is relevant. \\
We first verify Assumption A. Clearly the symbol $\fm_{\ta, \tb}$ is even, smooth, and fulfills \eqref{orderpsedodiff} with $m = 4$.
Assumption (A3) is verified provided $(\ta,\tb)$ is outside the resonant set 
\begin{equation}
\label{R.k}
\cR = \{ (\ta, \tb) \in \R \times \R_{>0} \colon  \ \ \ta = -(n^2 +1) \tb , \ \forall n \in \N \setminus\{1\} \} \ , 
\end{equation}
composed by infinitely many critical lines.
Then for any $(\ta, \tb) \not\in \cR$, Theorem \ref{existPTW} guarantees the existence of  $\e_0 = \e_0(\ta,\tb)>0$ and for any $\e \in (0, \e_0)$ a traveling wave
solution    $u_\e$ of  \eqref{eq}. 
Denote by  $\cL_\e$ the linearized operator \eqref{def:cLe} at  $u_\e$. We obtain:
\begin{theorem}[{\bf Kawahara}]
Consider equation \eqref{eq} with $\cM(D)$ having symbol  \eqref{kawa}.
Let 
 \begin{equation}
\begin{aligned}
 \label{}
& \cU :=
\{ (\ta, \tb) \in \R \times \R_{>0} \colon
-\frac{25}{3} \tb < \ta < - 5\tb \  \ 
\vee  \ \ 
-\frac{10}{3} \tb < \ta < -\frac{5}{3}\tb \ , \ta\neq -2\tb
 \}\setminus\cR \ , 
  \\
& \cS := 
\{ (\ta, \tb) \in \R \times \R_{>0} \colon  
\ta < -\frac{25}{3} \tb  \ \ 
 \vee \ \ 
-5 \tb < \ta < -\frac{10}{3}\tb \  \ 
\vee  \ \ 
\ta > -\frac53 \tb \ , \ta\neq -2\tb
\}\setminus\cR 
 \end{aligned} 
 \end{equation}
 where $\cR$ is the resonant set in \eqref{R.k}.
  \begin{itemize}
\item For  any $(\ta, \tb) \in \cU$  there exists $\e_1=\e_1(\ta, \tb)>0$ such that for any $\e \in [0,\e_1)$, the spectrum of the operator $\cL_\e$ depicts near the origin a complete  figure ``8'' as in Figure \ref{figure8image}.
\item For  any $(\ta, \tb) \in \cS$  the spectrum of $\cL_\e$ near the origin is purely imaginary.
\end{itemize}
\end{theorem}
\begin{proof}
We check Assumption B. Since
\begin{equation}
\begin{aligned}
\label{}
\te_{12} = 2\ta +4\tb \ ,   \quad
\te_d =3\ta + 5 \tb \ ,
\quad 
\te_b = -3\ta -10\tb \ , 
\quad
\te_w = \frac{3\ta + 25\tb}{3(\ta +5\tb)(3\ta + 5\tb)} \ .
\end{aligned}
\end{equation}
hence, the non-degeneracy condition defines the 4 critical lines 
$ \ta = - c_j \tb $, $c_j = -\frac53, -2, - \frac{10}{3}, -\frac{25}{3}$ where Assumption B is not met (note that the line $\ta = -5 \tb$ is already included in $\cR$). We plot these lines in Figure \ref{kawahara}.
The positive superlevel of the  function 
\begin{equation}
\label{ewb.k}
\te_{\scaleto{\mathrm{WB}}{3pt}}  = -\frac{(3\ta +10\tb)(3\ta + 25\tb)}{3(\ta +5\tb)(3\ta + 5\tb)} 
\end{equation}
defines the region of parameters $(\ta, \tb)$ where $\cL_\e$ has unstable spectrum near the origin, whereas the negative one defines the region where $\cL_\e$ is stable, see Figure \ref{kawahara}.
\end{proof}

\begin{figure}[H]\label{kawahara}
    \centering
    \includegraphics[width=0.35\linewidth]{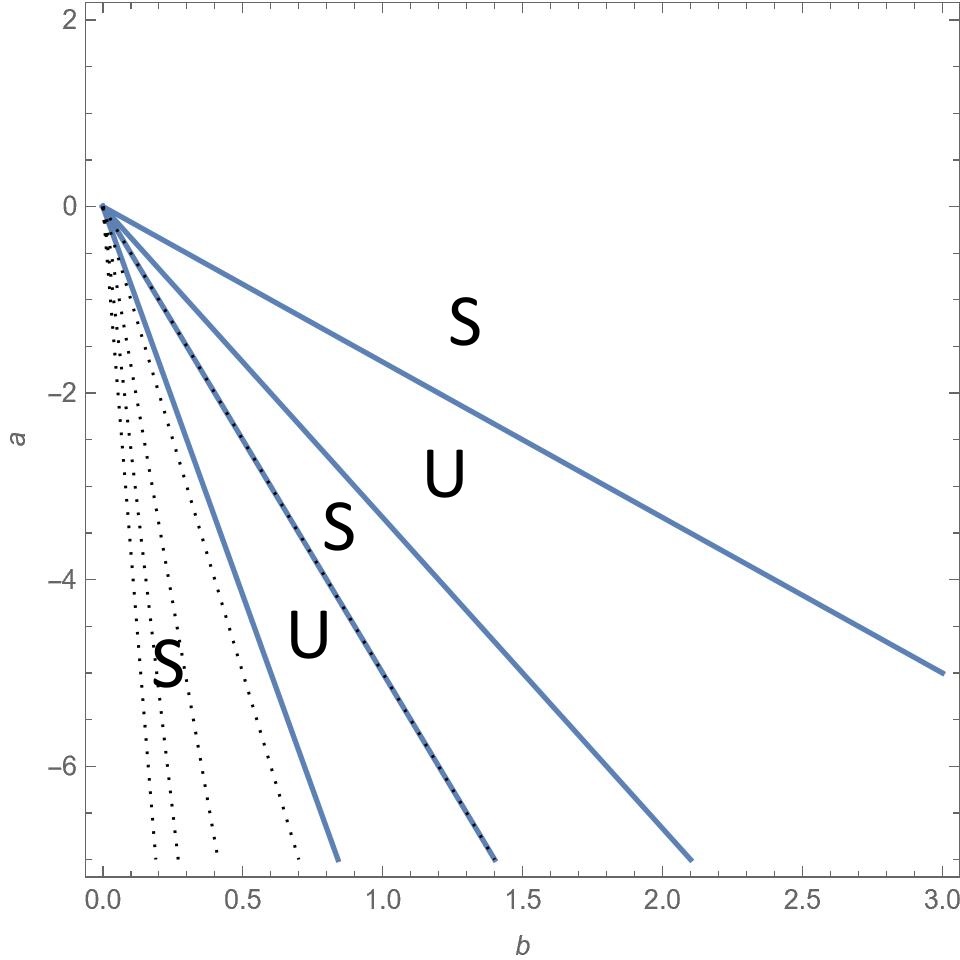}
    \hspace{1cm}
    \label{signkawa}
    \caption{
The blue lines are the critical lines
$\ta = - \frac53 \tb$, $\ta = -\frac{10}{3}\tb$,
$\ta = - 5\tb$,
 $\ta = -\frac{25}{3}\tb$, 
where $\te_d$, $\te_b$, $\te_w$ are identically zero or not defined,  
  the green one is $\ta = -2\tb$ where $\te_{12} \equiv 0$, 
$\te_{d}$, 
and the dotted line are the resonant lines in $\cR$ in \eqref{R.k}.
The regions $U$ are where $\te_{\scaleto{\mathrm{WB}}{3pt}} >0$ and the $S$ ones where $\te_{\scaleto{\mathrm{WB}}{3pt}} <0$.}
    \end{figure}

    \textsc{Relation with Creedon, Deconinck, Trichtchenko \cite{creedon2021high}.}
The paper  \cite{creedon2021high} studies the spectrum of $\cL_\e$ both near and away from zero. They prove the  existence of unstable eigenvalues of $\cL_{\mu,\e}$ near zero when 
\eqref{ewb.k} is strictly positive. 
Moreover, they prove the existence of isola of instabilities also away from zero.

\subsection{Intermediate long-wave equation}
The intermediate long-wave equation is equation \eqref{eq} with $\cM(D)$ having symbol 
\begin{equation}
\label{ilw}
\fm_\tth(\xi) = \frac{\xi}{\tanh(\tth \xi)} - \frac{1}{\tth}\ , \qquad \tth>0 
\end{equation}
It models unidirectional long waves at the surface of a fluid of finite depth, and it is particularly famous as an example of integrable Hamiltonian system. It can be derived from the Whitham equation, as it is described in \cite{joseph1977solitary}.\\  
To apply our abstract result to the intermediate long-wave equation, we first verify Assumption A.  The symbol is smooth, even,  of order 1. Moreover, $\xi \mapsto \fm_\tth(\xi)$  is strictly increasing for $\xi\geq 0$, and so Assumption A is satisfied.
Theorem \ref{existPTW} guarantees the existence of small amplitude traveling waves. 
In this case our abstract Theorem \ref{mainres1} yields:
\begin{theorem}[{\bf Intermediate long-wave}]
Consider equation \eqref{eq} with $\cM(D)$ having symbol  \eqref{ilw}.
 For  any $\tth>0$  the spectrum of $\cL_\e$ near the origin is purely imaginary.
\end{theorem}
\begin{proof}
Since $\xi \mapsto \fm_\tth(\xi)$ is strictly increasing,  $\te_{12}$ and $\te_d$ are strictly positive. $\te_b$ is strictly negative: by direct computations one has
$$
\te_b = -\ch^{-2}+\tth(\ch^{-4}-1+(1-\tth\ch^{-2}-\tth\ch^2)(\frac14(1+\ch^4)^2\ch^{-4}-1))
$$
and $$
\te_w =\frac{ 2\tth (-2+\tth\ch^{-2})(1-\ch^4)+2\ch^2}{-(1+2\tth^2)(1-\ch^4)+1+\ch^4}
$$
The sign of $\te_b$ and $\te_{\scaleto{\mathrm{WB}}{3pt}}$ is represented in Figure \ref{signILW} below, proving the Theorem.

\end{proof}

\begin{figure}[H]
    \centering
    \includegraphics[width=0.45\linewidth]{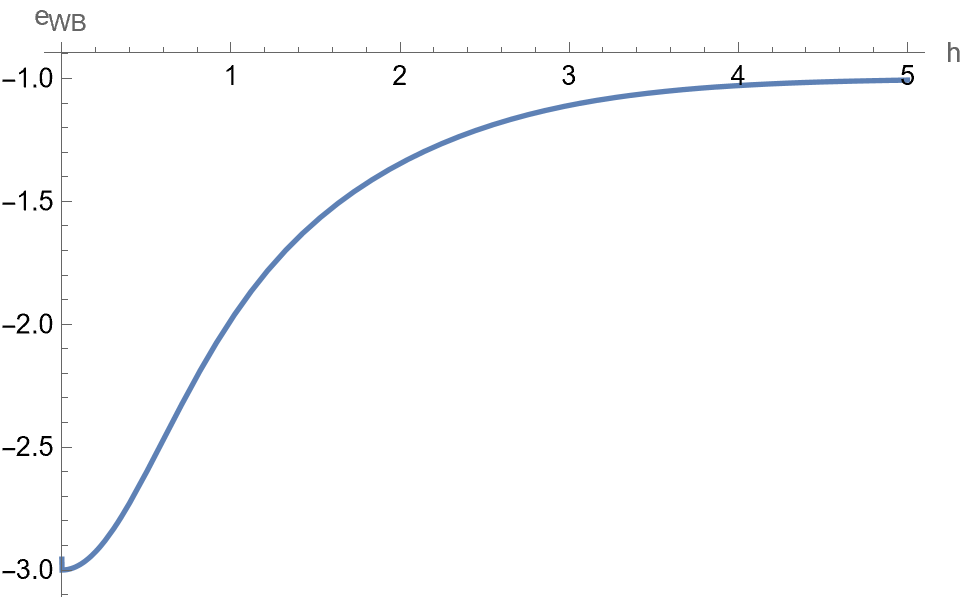}
    \includegraphics[width=0.45\linewidth]{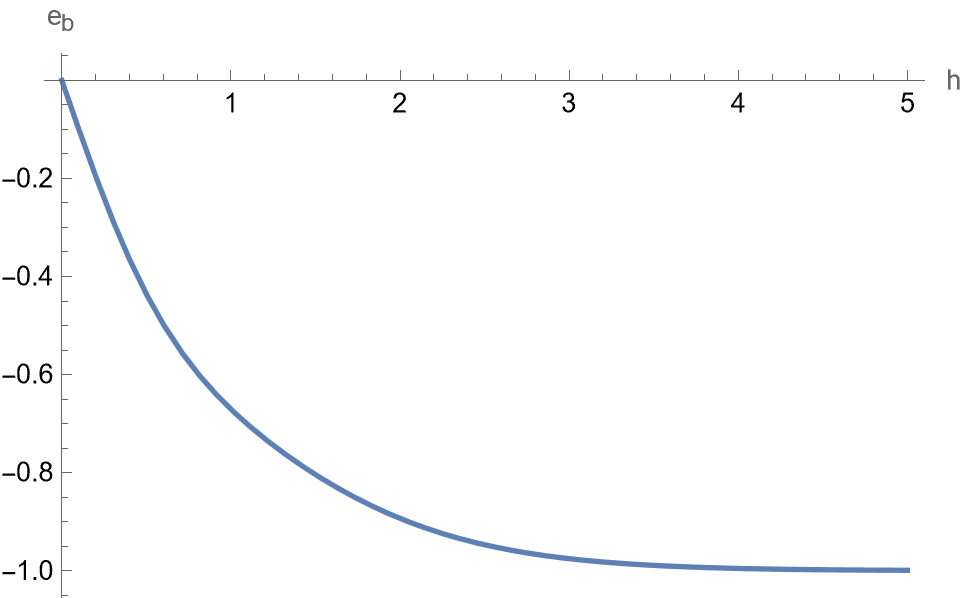}
    \hspace{1cm}
    \caption{
    In the figure is represented of the functions $\te_b$ and $\te_{\scaleto{\mathrm{WB}}{3pt}}$ in dependence of the parameters $\tth>0$.}
    \label{signILW}
    \end{figure}

   \begin{comment} \textsc{Relation with Pava-Cardoso-Natali \cite{pava2017stability}.}
  In \cite{pava2017stability} the authors prove orbital and linear stability of spatially periodic wavetrain solutions with mean zero. 
   \end{comment}
   \subsection{Benjamin-Ono, KdV and Fractional KdV}
   \textit{Benjamin-Ono, KdV} and \textit{fractional KdV} are equations of type \eqref{eq}, with Fourier multiplier having symbol respectively
   $\fm(\xi)=|\xi|$, $\fm(\xi)=\xi^2$ and $\fm(\xi)=\xi^\alpha$, $\alpha\in\R_{>\frac12}$,  respectively.     For the first two equations, Assumption A is satisfied, noticing that the symbol of the Benjamin-Ono equation is infinitely many times differentiable from the left at $\xi=0$. Concerning the fractional KdV equation, the $\cC^3$ condition implies $\alpha\geq 3$. \\
    Assumption B is satisfied for the KdV and fractional KdV equations under analysis, whereas it fails in the case of the Benjamin-Ono equation, that represents  the critical case for which $\te_{\scaleto{\mathrm{WB}}{3pt}}=0$. 
Thus, for the former two equations Theorem \ref{mainres1} can be applied, establishing modulational  stability near the origin, as already proved by Johnson \cite{johnson2013stability}. 
The Benjamin-Ono equation  is further analyzed in  \cite{BH}, where it is  proved that traveling waves are modulationally stable. 
%    Analysing the modulational instability of the Benjamin-Ono equation, instead, requires further approximation orders in $\mu$ and $\e$, in a similar fashion to \cite{BMV3} where the authors proved the modulational instability for the water waves equation at the critical depth.\\
%    In the regimes where Theorem \ref{mainres1} can be applied, this work is in perfect accordance with the results in \cite{johnson2013stability}, where the author analyze the modulational instability of a large class of fractional kdV equation. In analogy to our result, the Benjamin-Ono equation constitute a critical case of study for which modulational stability/instability is still open.

    \appendix

\section{Expansion of the Kato basis} \label{appendixA}
 We first Taylor-expand the transformation operators $U_{\mu,\e}$ defined in \eqref{OperatorU}. 
  We denote  $\pa_\e$ with a prime and  $\pa_\mu$ with a dot. 
  The  following lemma follows with the same lines of \cite[Lemma A.1]{BMV1}, and we omit the proof.
\begin{lemma}\label{B1}
The first jets of $U_{\mu,\e}P_{0,0}$ are 
 \begin{align}
  U_{0,0}P_{0,0}&=P_{0,0} \, , \quad U_{0,0}'P_{0,0}=P_{0,0}'P_{0,0} \, , \quad \dot U_{0,0}P_{0,0}=\dot P_{0,0}P_{0,0} \, , \label{Ufirstorder}\\
\dot U_{0,0}'P_{0,0}&=
\big(\dot P_{0,0}'- \tfrac12 P_{0,0}\dot P_{0,0}' \big)P_{0,0} \, , \label{Umix} 
 \end{align}
where
\begin{align}
\label{Pdereps}
 P_{0,0}' &= \frac{1}{2\pi\im} \oint_\Gamma (\cL_{0,0}-\lambda)^{-1} \cL_{0,0}' (\cL_{0,0}-\lambda)^{-1} \de\lambda \, ,  \\ 
\dot P_{0,0}  \label{Pdermu} &= \frac{1}{2\pi\im} \oint_\Gamma (\cL_{0,0}-\lambda)^{-1} \dot \cL_{0,0} (\cL_{0,0}-\lambda)^{-1} \de\lambda \, ,
\end{align}
and
\begin{subequations}
\label{Pmisto}
\begin{align}
\dot P_{0,0}' &= -\frac{1}{2\pi\im} \oint_\Gamma (\cL_{0,0}-\lambda)^{-1} \dot \cL_{0,0} (\cL_{0,0}-\lambda)^{-1}  \cL_{0,0}' (\cL_{0,0}-\lambda)^{-1} \de\lambda  \label{Pmisto1}\\
&\qquad   -\frac{1}{2\pi\im} \oint_\Gamma (\cL_{0,0}-\lambda)^{-1} \cL_{0,0}' (\cL_{0,0}-\lambda)^{-1} \dot \cL_{0,0} (\cL_{0,0}-\lambda)^{-1} \de\lambda \label{Pmisto2} \\ 
&\qquad   + \frac{1}{2\pi\im} \oint_\Gamma (\cL_{0,0}-\lambda)^{-1} \dot \cL_{0,0}' (\cL_{0,0}-\lambda)^{-1}  \de\lambda \label{Pmisto3} \, .
\end{align}
\end{subequations}
The operators $\cL_{0,0}'$ and $\dot \cL_{0,0}$ are
\begin{equation}
\cL_{0,0}' =
-2  \pa_x \circ \cos(x) 
 \qquad 
 \dot \cL_{0,0} =  
\im (\fm(1) - \cM(D))  - \pa_x \dot\cM(D)
   \label{cLfirstorder}
\end{equation}
where $\dot \cM(D)$ is the purely imaginary, odd Fourier multiplier with symbol $\dot \fm(\xi) \equiv \frac{\di \fm}{\di \xi}(\xi)$, i.e. 
\begin{equation}\label{multiplier}
\widehat{ [\dot \cM(D)u]}(k) := \dot \fm(k) \widehat u_k  \ , \quad k \in \Z \ . 
\end{equation}
The operator $\dot \cL_{0,0}'$ is
\begin{equation}\label{cLmisto}
\dot \cL_{0,0}' =  -2\im \cos x\,.
\end{equation}
\end{lemma}
%\begin{proof}
% One has the Taylor expansion  in $\cL(Y)$
%$$
%   U_{\mu,\e}P_{0,0}  = P_{\mu,\e}P_{0,0} + \frac{1}{2}(P_{\mu,\e}-P_{0,0})^2P_{\mu,\e}P_{0,0} +\cO(P_{\mu,\e}-P_{0,0})^4   \, ,
%  $$
%  where  $\cO(P_{\mu,\e}-P_{0,0})^4 = \cO(\e^4,\e^3\mu,\e^2\mu^2,\e\mu^3,\mu^4) \in \cL(Y)$.
%Then \eqref{Ufirstorder} follows.
%To compute  $\dot U_{0,0}'P_{0,0}$ use that 
%    $        \pa_\mu\pa_\e \frac{1}{2}(P_{\mu,\e}-P_{0,0})^2P_{\mu,\e}P_{0,0} \big|_{\mu=\e=0} = \frac12 (\dot P_{0,0}P_{0,0}' + P_{0,0}' \dot P_{0,0})P_{0,0}$ and the projector identity 
%  $  \dot P_{0,0}P_{0,0}' + P_{0,0}' \dot P_{0,0} = \dot P_{0,0}'$ (which follows differentiating  $P_{\mu,\e}^2=P_{\mu,\e}$).
%\end{proof}
The next lemma describes the projectors $P_{\mu,\e}$ and the transformation operators $U_{\mu,\e}$  at $\e = 0$:
\begin{lemma}\label{nopuremu}
    For every $\mu$ small enough, one has 
$ P_{\mu,0}P_{0,0} = P_{0,0}$ and 
$U_{\mu,0} P_{0,0} =  P_{0,0}$. 
    In particular $f_k^\sigma(\mu, 0) = f_k^\sigma$ for $(k,\sigma) \in \{(1, \pm), (0,+) \}$.
\end{lemma}
\begin{proof}
The basis $\{f^\pm_1, f_0^+\}$ in \eqref{fksigma} is a basis of  $\cV_{\mu,0}$   for any $\mu$ sufficiently small. Hence, $P_{\mu,0} = P_{0,0}$ for any $\mu$ small enough, so using Definition  \ref{OperatorU} one obtains also $U_{\mu,0}P_{0,0} = P_{0,0}$.
\end{proof}
By  Lemma  \ref{B1} and \ref{nopuremu} we have the Taylor expansion
\begin{equation}\label{ordinibase}
f_k^\sigma(\mu,\e) = f_k^\sigma + \e P_{0,0}' f_k^\sigma  + \mu\e  \big(\dot P_{0,0}'- \frac12 P_{0,0}\dot P_{0,0}' \big) f_k^\sigma + \cO^\reg(\mu^2\e,\e^2) \, , \qquad  (k, \sigma) \in \{(1, \pm), (0,+) \}
\end{equation}
where  the remainders $\cO^\reg(\cdot)$ are vectors in  $H^{m_\star}(\T)$. Note that the term $\mu \dot P_{0,0} f_k^\sigma$ must vanish because $f_k^\sigma(\mu,0) \equiv  f_k^\sigma$.
In order to compute the vectors
 $P_{0,0}' f_k^\sigma$  using 
 \eqref{Pdereps}, it is useful to know the action of  $(\cL_{0,0} - \lambda)^{-1}$ on the vectors 
\begin{equation}
\begin{aligned}
\label{fksigma2}
& f_k^+:= \cos(kx) , 
\quad  f_k^- :=\sin(kx) \, , \quad  k \in \N  \ . 
\end{aligned}  
\end{equation}
We have the following result:
\begin{lemma}\label{lem:VUW}
The space $ H^{m_\star}(\T) $ decomposes as 
$
H^{m_\star}(\T) =  \cV_{0,0} \oplus {\cW_{H^{m_\star}}}$
with  $\cW_{H^{m_\star}}=\overline{\bigoplus\limits_{k=2}^\infty \cW_k}^{H^{m_\star}}$
where the subspaces $\cV_{0,0}$ and $ \cW_k $, defined below, are 
invariant  under   $\cL_{0,0} $ and  the following properties hold:
\begin{itemize}
\item[(i)] $ \cV_{0,0} = \textup{span} \{ f^+_1, f^-_1, f_0^+ \}$  is the kernel of $\cL_{0,0}$. For any $ \lambda  \neq 0 $  the operator 
$ \cL_{0,0}-\lambda :  \cV_{0,0} \to \cV_{0,0} $ is invertible and  
 \begin{align}\label{primainversione1}
& (\cL_{0,0}-\lambda)^{-1}f_k^\sigma = -\frac1\lambda f_k^\sigma \, ,
\quad \forall  (k, \sigma) \in \{(1, \pm), (0,+) \} \ . 
\end{align} 
\item[{(ii)}]  Each
subspace $\cW_k:= \textup{span}\left\{f_k^+, \ f_k^-,  \right\}$ is  invariant under $ \cL_{0,0} $.  Let $\cW_{L^2}=\overline{\bigoplus\limits_{k=2}^\infty \cW_k}^{L^2}$. For any
$|\lambda| < 
\delta_0$ small enough, the operator 
$ {\cL_{0,0}-\lambda :  \cW_{H^{m_\star}}\to \cW_{L^2} }$ is invertible and in particular 
\begin{equation}
\begin{aligned}
\label{primainversione4}
 (\cL_{0,0}-\lambda)^{-1} f_k^\sigma
&= \sigma \frac{1}{k (\fm(1)-\fm(k))} f_k^{-\sigma}
-  \frac{\lambda }{k^2 (\fm(1) - \fm(k))^2} f_k^\sigma  
+ \lambda^2 \varphi_k^\pm(\lambda, x) \,  \ , 
 \end{aligned}
\end{equation}
for some analytic  functions  $ B(\delta_0) \ni \lambda \mapsto \varphi_k^\pm(\lambda, \cdot) \in H^1(\T, \C)$.
\end{itemize}
\end{lemma}
\begin{proof}
$(i)$ For  $f\in \cV_{0,0} = \ker (\cL_{0,0} )$ and $\lambda \neq 0$  we have that
$  (\cL_{0,0}-\lambda) f = -\lambda f$.
        Inverting both sides of the equation gives \eqref{primainversione1}.

$(ii)$  Notice that $\cW_k$ is invariant under the action of $\cL_{0,0}$. The matrix representing the action  $\cL_{0,0}\colon \cW_k \to \cW_k$ in the basis $\{ f_k^+, f_k^-\}$ is given by 
        \begin{equation}
            \tL_{0,0} = \begin{pmatrix}
                0 & k ( \fm(1) - \fm(k)) \\
                - k ( \fm(1)- \fm(k)) & 0
            \end{pmatrix} \ . 
        \end{equation}
        By Assumption (A3), $\det  \tL_{0,0} = k^2 ( \fm(1) - \fm(k))^2 \neq 0$, so the matrix $\tL_{0,0}$ is invertible with inverse
        \begin{equation}
           \tL_{0,0}^{-1} = \begin{pmatrix}
                0 & \frac{-1}{ k ( \fm(1) -\fm(k))} \\
                \frac{1}{ k ( \fm(1)- \fm(k))} & 0
            \end{pmatrix} \ .
        \end{equation}
        The invertibility of $\cL_{0,0}-\lambda$ and the formulas in \eqref{primainversione4}  follow using the Neumann series expansion
        \begin{equation}
            (\cL_{0,0}-\lambda)^{-1} = \cL_{0,0}^{-1} (1-\lambda\cL_{0,0}^{-1})^{-1} = \sum_{j=0}^\infty \lambda^j \cL_{0,0}^{-j-1} = \cL_{0,0}^{-1} +  \sum_{j=0}^\infty \lambda^{j+1} \cL_{0,0}^{-j-2}
        \end{equation}
     which converges provided  $|\lambda|$ is small enough.
\end{proof}

We now compute the action of the operator  $(\cL_{0,0} - \lambda)^{-1} \cL_{0,0}'$ on the vectors $f_1^\pm, f_0^+$.
\begin{lemma}[Action of $(\cL_{0,0} - \lambda)^{-1} \cL_{0,0}'$ on $\cV_{0,0}$] \label{explambda'} One has
    \begin{equation} \label{expansionRlambda}
        \begin{aligned}
           & (\cL_{0,0} - \lambda)^{-1} \cL_{0,0}' f_1^\sigma = - \frac{1}{\fm(1) - \fm(2)} f_2^\sigma  - \sigma  \frac{\lambda}{2(\fm(1) - \fm(2))^2} f_2^{-\sigma} +\cO_\cW(\lambda^2) \ ,\\
           & (\cL_{0,0} - \lambda)^{-1} \cL_{0,0}' f_0^+ = -\frac{\sqrt2}{\lambda} f_1^- \ .
        \end{aligned}
    \end{equation}
\end{lemma}
\begin{proof} Recall
$
\cL_{0,0}' 
=- 2  \pa_x \circ \cos (x),$
 and so
$     \cL_{0,0}'f_1^\pm =   \pm 2 f_2^\mp$ and $\cL_{0,0}' f_0^+ = \sqrt2f_1^-$.
Then apply    Lemma \ref{lem:VUW}.
\end{proof}

We are ready to compute the terms  $ P_{0,0}' f_k^\sigma$  in \eqref{ordinibase}:
\begin{lemma}\label{lem:firstjet}
One has 
    $$        
                P_{0,0}' f_1^+ = \frac{\cos (2x)}{\fm(1) - \fm(2)},\qquad
                P_{0,0}' f_1^- = \frac{\sin (2x)}{\fm(1) - \fm(2)}, \qquad
                P_{0,0}' f_0^+ = 0 \ .
    $$
\end{lemma}
\begin{proof}
 By equation \ref{Pdereps} and Lemma \ref{lem:VUW}--$(i)$ one gets
    $$
     P_{0,0}' f_k^\sigma = -\frac{1}{2\pi\im} \oint_\Gamma \frac1\lambda (\cL_{0,0}-\lambda)^{-1} \cL_{0,0}' f_k^\sigma \de\lambda \ .
    $$
 Then    by residue theorem the only relevant terms are those of order $O(1)$ in $\lambda$  in the expansion of $- (\cL_{0,0}-\lambda)^{-1} \cL_{0,0}' f_k^\sigma$. Then the  result follows by Lemma  \ref{explambda'}.
\end{proof}

\begin{lemma}[Action of $(\cL_{0,0} - \lambda)^{-1} \dot \cL_{0,0}$ on $\cV_{0,0}$] One has
    \begin{equation}\label{risdL00}
    \begin{aligned}
        &(\cL_{0,0} - \lambda)^{-1} \dot \cL_{0,0} f_1^\sigma = \frac{\im \dot{\fm}(1)}{\lambda} f_1^\sigma\, , \qquad
        (\cL_{0,0} - \lambda)^{-1} \dot \cL_{0,0} f_0^+\, \,= -\frac{\im(\fm(1) - \fm(0) )}{\lambda} f_0^+\  \\
    \end{aligned}
\end{equation}
and,  for $k\geq 2$, 
\begin{equation}\label{risdL002}
        (\cL_{0,0} - \lambda)^{-1} \dot \cL_{0,0} f_k^\sigma =\sigma  \dfrac{\im (\fm(1) - \fm(k) - k  \dot\fm(k))}{k(\fm(1) - \fm(k))} f_k^{-\sigma} + \cO_\cW (\lambda) \,  \ . 
\end{equation}
\end{lemma}
\begin{proof}
By the second of \eqref{cLfirstorder} one has 
$\dot \cL_{0,0}f_1^\sigma  = - \pa_x \dot\cM(D) f_1^\sigma = -\im \dot{\fm}(1) f_1^\sigma$.
Then 
 the first of \eqref{risdL00} follows using  \eqref{primainversione1}.
  To prove the second one we use that 
 $\pa_x \dot\cM(D) f_0^+=0$ and so 
$\dot \cL_{0,0} f_0 = \im(\fm(1)-  \fm(0)) f_0^+$
and again we conclude using \eqref{primainversione1}.
Formula \eqref{risdL002} is proved on the same lines, using \eqref{primainversione4}.
\end{proof}

\begin{lemma}\label{lem:secondjet}
One has
\begin{equation}\label{dot'P}
    \begin{split}
        &\dot P_{0,0}' f_1^\sigma = - \sigma \im \frac{\dot \fm (1)-2\dot \fm (2) }{2(\fm(1) - \fm(2))^2}f_2^{-\sigma}  \ , \qquad  \dot P_{0,0}' f_0^+ = 0 \ .
    \end{split}
\end{equation}
Consequently, 
\begin{equation}\label{dot'P2}
 \big(\dot P_{0,0}'- \frac12 P_{0,0}\dot P_{0,0}' \big) f_1^\sigma = - \sigma \im \frac{\dot \fm (1)-2\dot \fm (2) }{2(\fm(1) - \fm(2))^2}f_2^{-\sigma} \, , 
 \quad
  \big(\dot P_{0,0}'- \frac12 P_{0,0}\dot P_{0,0}' \big) f_0^+ = 0  \ . 
 \end{equation}
\end{lemma}
\begin{proof}
 By  \eqref{Pmisto} and \eqref{primainversione1} we write, for $(k,\sigma) \in \{(1,\pm), (0,+) \}$ 
\begin{align*}
\dot P_{0,0}' f_k^\sigma  = & \frac{1}{2\pi\im} \oint_\Gamma \frac{1}{\lambda} (\cL_{0,0}-\lambda)^{-1} \dot \cL_{0,0} (\cL_{0,0}-\lambda)^{-1}  \cL_{0,0}' f_k^\sigma \de\lambda  \\
&
+ \frac{1}{2\pi\im} \oint_\Gamma \frac{1}{\lambda} (\cL_{0,0}-\lambda)^{-1} \cL_{0,0}' (\cL_{0,0}-\lambda)^{-1} \dot \cL_{0,0} f_k^\sigma \de\lambda   \\ 
&
 - \frac{1}{2\pi\im} \oint_\Gamma 
 \frac{1}{\lambda} (\cL_{0,0}-\lambda)^{-1} \dot \cL_{0,0}' f_k^\sigma  \de\lambda =:  \text{I}_k^\sigma + \text{II}_k^\sigma +\text{III}_k^\sigma. 
\end{align*}
Let now $(k, \sigma) = (1,+)$.
Using \eqref{primainversione1}, \eqref{expansionRlambda}, \eqref{risdL002} and the residue theorem  we get 
    \begin{equation}\label{I}
        \begin{aligned}
            \mathrm{I}_1^+
            &= \frac{1}{2\pi\im} \oint_\Gamma \frac{1}{\lambda}(\cL_{0,0}-\lambda)^{-1} \dot \cL_{0,0} \left(- \frac{1}{\fm(1) - \fm(2)} \cos(2x)+\cO_\cW(\lambda)\right)  = -\im \frac{(\fm(1) - \fm(2)-2\dot \fm (2))}{2(\fm(1) - \fm(2))^2}\sin (2x)  \ .
        \end{aligned}
    \end{equation}
Similarly, using     \eqref{primainversione1}, \eqref{risdL00}, \eqref{expansionRlambda} and the residue theorem
        \begin{align}
        \notag
            \mathrm{II}_1^+
        &= \frac{1}{2\pi\im} \oint_\Gamma \frac1\lambda  (\cL_{0,0}-\lambda)^{-1} \cL_{0,0}'\left(\frac{\im \dot{\fm}(1)}{\lambda} f_1^+ \right)\de \lambda \\
        \notag
        &= - \frac{1}{2\pi\im} \oint_\Gamma \frac{\im \dot{\fm}(1)}{\lambda^2} ( \frac{1}{\fm(1) - \fm(2)} \cos(2x) + \frac{\lambda}{2(\fm(1) - \fm(2))^2} \sin (2x) +\cO_\cW(\lambda^2)) \de \lambda\\
        \label{II}
        &=-\im \frac{\dot \fm(1)}{2(\fm(1) - \fm(2))^2}\sin (2x)
        \end{align}
    Finally consider $\mathrm{III}_1^+$. 
By \eqref{cLmisto} we have  $\dot \cL_{0,0}' f_1^+ = -\im (\cos (2x) + 1)$, so using also 
   \eqref{primainversione1} we get 
    \begin{equation}\label{III}
            \mathrm{III}_1^+  =
        \frac{1}{2\pi\im} \oint_\Gamma \left(-\frac1\lambda\right)(\cL_{0,0}-\lambda)^{-1} \dot \cL_{0,0}' f_1^+\de\lambda = \frac{1}{2\pi\im} \oint_\Gamma \left(\frac\im\lambda(\cL_{0,0}-\lambda)^{-1} \cos (2x) -\frac{\im}{\lambda^2}\right)\de \lambda = \im \frac{\sin (2x)}{2(\fm(1) - \fm(2))}
    \end{equation}
    Summing \eqref{I}, \eqref{II}, \eqref{III} up we deduce the first of \eqref{dot'P}.
    The computations of
    $\mathrm{I}_1^- $, $\mathrm{II}_1^- $, $\mathrm{III}_1^- $ are analogous.
    
    We now consider $(k,\sigma) = (0, +)$. By \eqref{expansionRlambda} and \eqref{risdL00} and the residue theorem 
\begin{align*}
   \mathrm{I}_0^+ = 
   - \frac{1}{2\pi\im} \oint_\Gamma \frac{\im \sqrt{2} \dot\fm(1)}{\lambda^3} f_1^- \de\lambda = 0  \ , 
   \quad 
   \mathrm{II}_0^+ = 
  \frac{1}{2\pi\im} \oint_\Gamma \frac{\im \sqrt{2} (\fm(1) - \fm(0))}{\lambda^3} f_1^- \de\lambda = 0  \ .   
\end{align*}    
    Then, since $\dot \cL_{0,0}' f_0^+ = -  \im \sqrt{2} f_1^+$, by \eqref{primainversione1} we find also $\mathrm{III}_0^+ = 0$, so $\dot P'_{0,0} f_0^+ =0$.
    
    To deduce \eqref{dot'P2} just notice that 
 $\dot P_{0,0}' \cV_{0,0} \subseteq \cW_{L^2}$, and so  $P_{0,0}\dot P_{0,0}' = 0$.
\end{proof}
Finally, in the next lemma we prove the identities in Equation \eqref{basis.der}
  \begin{lemma}
    The following expansions hold:
    \begin{equation}
        \dot f_1^- (0,\e) = \im \e \fb \cos (2x)+\cO(\e^2) , \quad  
 \dot f_0(0,\e) = \cO(\e^2) \ , \quad 
   \ddot f_k^\sigma(0,\e) = \cO(\e) \ .
    \end{equation}
  \end{lemma}
  \begin{proof}
    Recall that $f_k^\sigma(\mu,\e) = U_{\mu,\e}f_k^\sigma$. Thus, we have the following Taylor expansion
    \begin{equation}\label{eqprov}
        \begin{split}
            \dot f_1^- (0,\e) &= \dot U_{0,0}f_1^- + \e\dot U'_{0,0}f_1^-+\cO(\e^2) , \\  
            \dot f_0^+(0,\e) &= \dot U_{0,0}f_0^+ + \e\dot U'_{0,0}f_0^+ + \cO(\e^2) \ , \\
              \ddot f_k^\sigma(0,\e) &= \ddot U_{0,0} f_k^\sigma +\cO(\e) \ .
        \end{split}
    \end{equation}
    By Lemma \ref{nopuremu} we have that $\dot U_{0,0} f_k^\sigma  = \ddot U_{0,0} f_k^\sigma  = 0 $. The thesis follows by Lemma \ref{lem:secondjet}, substituting the expressions for $\dot U'_{0,0} =\dot P_{0,0}'- \frac12 P_{0,0}\dot P_{0,0}'  $ in equation \eqref{eqprov}.
  \end{proof}

\begin{proof}[Proof of Lemma \ref{totalbasisexpansion}]
It follows from the expansion in equation \eqref{ordinibase}, Lemma \ref{lem:firstjet} and Lemma \ref{lem:secondjet}.
To prove that the term of order $\e^2$ of  $f_1^+(\mu,\e)$ has zero average use that, by Lemma \ref{lemmanonzero}  $(ii)$
$$
(f_1^+(0,\e), 1 ) = (U_{0,\e} f_1^+, 1 )  =  (U_{0,\e} \cJ_0 f_1^-, 1 ) =
  (\cJ_0 U_{0,\e}^{-*} f_1^-, 1 )  = 0 \ . 
$$
To prove that the remainders of the expansion are, as in Equation \eqref{basis.exp}, of order $\cO^\reg (\mu^2\e , \mu\e^2,\e^3)$ we proceed as follows. Let 
$h(\mu,\e) = f_k^\sigma (\mu,\e) - f_k^\sigma - \mu \dot U_{0,0}f_k^\sigma - \e  U_{0,0}'f_k^\sigma - \mu \e \dot U_{0,0}'f_k^\sigma -\frac12 \e^2 U_{0,0}''f_k^\sigma  $. By the Taylor expansion, we know that 
$$
h(\mu,\e) = \mu^3 \varphi_0(\mu,\e) + \mu^2\e \varphi_1(\mu,\e) + \mu\e^2 \varphi_2(\mu,\e) + \e^3 \varphi_3(\mu,\e)
$$
for some $\cC^{\reg-2}$ functions $\varphi_j$, $j=0\dots 3$. By Lemma \ref{nopuremu}, $h(\mu,0)\equiv 0$, and thus $\varphi_0 (\mu,0)\equiv 0$. Thus, by Lemma \ref{lem:fg}, $\varphi_0 (\mu,\e) = \e \tilde \varphi_0 (\mu,\e)$ for some $\cC^{\reg-3}$ function $\tilde \varphi_0 (\mu,\e)$, and in conclusion $h(\mu,\e) =\cO^\reg(\mu^2\e,\mu\e^2,\e^3)$.\\
Finally, $f_0(0,\e) = U_{0,\e} f_0 = f_0$ again by Lemma \ref{lemmanonzero}  $(ii)$.
\end{proof}
\section{Technical lemmas}
\begin{lemma}\label{lem:fg}
    Let $f:\overline{B(\mu_0)}\times \overline{B(\e_0)}\mapsto \R$ be a $\cC^n$ function, $n \geq 2$, where by $\overline{B(r)}$ we denote the closed ball in $\R$ of center $0$ and radius $r$. 
    \begin{itemize}
    \item[(i)] Assume $f(\mu,0) = 0$ for any $(\mu,\e) \in B(\mu_0)\times B(\e_0)$. Then $f(\mu,\e) = \e \, g(\mu,\e)$ for some $g \in \cC^{n-1}$.
    \item[(ii)] Assume  $|f(\mu,\e)|\leq C|\mu|^{a_0} |\e|^{b_0} \sum_{i} \mu^{a_i} \, \e^{b_i} $ on $B(\mu_0)\times B(\e_0)$ for some constant $C>0$ and natural numbers  $1 \leq a_0+b_0<n$. Then $f(\mu,\e) = \mu^{a_0} \e^{b_0} \, g(\mu,\e)$ for some $g \in \cC^{n-a_0-b_0}$ fulfilling
$|g(\mu,\e)|\leq C \sum_{i} \mu^{a_i} \, \e^{b_i} $.
\item[(iii)] Assume $f(0,0) = 0$, and $n\geq 1$. Then $\exists C>0$ such that $f(\mu,\e)\leq C(|\mu| + |\e|)$ for every $(\mu,\e)\in \overline{B(\mu_0)}\times \overline{B(\e_0)}$.
    \end{itemize}
%   
%    
%    Suppose that $|f(\mu,\e)|\leq C|\mu|^a |\e|^b$ on $B(\mu_0)\times B(\e_0)$ for some constant $C>0$ and natural numbers $a,b$ such that $a+b<n$. Then the function
%    \begin{equation}
%        g(\mu,\e):=\begin{cases}\label{apcgstatement}
%            \frac{f(\mu,\e)}{\mu^a\e^b} \qquad &\mu\neq 0\text,\ \e\neq 0 \ ,\\
%            \frac{1}{\e^ba!}\pa_\mu^a f(\mu,\e) \qquad &\mu=0,\ \e\neq 0\ ,\\
%            \frac{1}{\mu^ab!}\pa_\e^b f(\mu,\e) \qquad &\mu\neq 0,\ \e=0\ .
%        \end{cases}
%    \end{equation}
%    is of class $\cC^{n-a-b}$. In particular $\cO^\reg(\mu^a\e^b(\sum_p \mu^{a_p}\e^{b_p})) =\mu^a\e^b\cO^{\reg-a-b}(\sum_p \mu^{a_p}\e^{b_p})  $
\end{lemma}
\begin{proof}
    $(i)$ Put $h(\tau):= f( \mu, \tau  \e)$. Then 
    $$
    f(\mu,\e) =\underbrace{f(\mu,0)}_{=0} + \e \int_0^1 \pa_\tau h(\tau) \di \tau = \e g(\mu,\e)\ , \qquad g(\mu,\e)\in \cC^{n-1}
    $$
    proving the first item.

    $(ii)$ At least one among $a_0$ and $b_0$ is not zero. Assume it is  $b_0$. First notice that, if $|f(\mu,\e)|\leq C|\mu|^{a_0} |\e|^{b_0} \sum_{i} \mu^{a_i} \, \e^{b_i} $, then $f(\mu,0) = 0$. Using $i)$, we have that $f(\mu,\e) = \e h_1 (\mu,\e)$, with $h_1 (\mu,\e)\in\cC^{n-1}$ satisfying $|h_1 (\mu,\e)|\leq  C|\mu|^{a_0} |\e|^{b_0-1} \sum_{i} \mu^{a_i} \, \e^{b_i} $. \\
    We can then iterate the procedure $b_0-1$ times, obtaining a sequence of functions $h_k(\mu,\e)$, $k = 2, \ldots, b_0$, fulfilling
    $$
    h_{k-1}(\mu,\e) = \e h_k (\mu,\e) \ , \qquad h_k \in \cC^{n-k}\ ,\qquad |h_k (\mu,\e)|\leq  C|\mu|^{a_0} |\e|^{b_0-k} \sum_{i} \mu^{a_i} \, \e^{b_i} \ . 
    $$
    In particular $f(\mu,\e) = \e^{b_0} h_{b_0} (\mu,\e)$. Repeating the same procedure in the variable $\mu$ allows us to gather a factor $\mu^{a_0}$ from $h_{b_0} (\mu,\e)$, proving the result.\\
    $(iii)$ It is a standard application of the mean value theorem. 
\end{proof}

\end{document}